\documentclass[reqno]{amsart} %
\usepackage{amssymb, mathdots}
\usepackage{mathabx}
\usepackage{mathrsfs}
\usepackage[utf8x]{inputenc}
\usepackage{amsmath,  graphicx}
\usepackage{diagrams}
\usepackage{tikz}
\usetikzlibrary{matrix}
\usepackage{wrapfig}  
\usepackage{enumerate}
\usepackage{url}
\DeclareSymbolFont{bbold}{U}{bbold}{m}{n}
\DeclareSymbolFontAlphabet{\mathbbold}{bbold}
\def\qmod#1#2{{\hbox{}^{\displaystyle{#1}}}\!\big/\!\hbox{}_{
\displaystyle{#2}}}

\def\resto#1#2{{
#1\hskip 0.4ex\vline_{\hskip 0.2ex\raisebox{-0,2ex}
{{${\scriptstyle #2}$}}}}}

 \def\psp#1#2%
  {\mathop{}%
   \mathopen{\vphantom{#2}}^{#1}%
   \kern-\scriptspace%
    \hskip -0.3mm{#2}} 
 \def\psb#1#2%
  {\mathop{}%
   \mathopen{\vphantom{#2}}_{#1}%
   \kern-\scriptspace%
    \hskip -0.0mm{#2}} 
 \def\pscr#1#2#3%
  {\mathop{}%
   \mathopen{\vphantom{#3}}^{#1}_{#2}%
   \kern-\scriptspace%
    \hskip -0.3mm{#3}} 
\def\EXUX{\hskip3mm\psp{\scriptscriptstyle X}{\hskip-4.2mm\psb{\scriptscriptstyle U\hskip -0.5mm X}{\mathscr{E}}}}

\def\EUUX{\hskip3mm\psp{\scriptscriptstyle U}{\hskip-4.2mm\psb{\scriptscriptstyle U\hskip -0.5mm X}{\mathscr{E}}}} 

\def\sUX{\scriptstyle U\hskip -0.6mm X}
\def\ssUU{\scriptscriptstyle U\hskip -0.5mm U}
\def\ssUX{\scriptscriptstyle U\hskip -0.5mm X}

\def\C{{\mathbb C}}

\def\H{{\mathbb H}}

\def\N{{\mathbb N}}

\def\P{{\mathbb P}}

\def\R{{\mathbb R}}

\def\Z{{\mathbb Z}}

\def\cringle{\mathaccent23}
\def\union{\mathop{\bigcup}}

\def\textmap#1{\mathop{\vbox{\ialign{
                                  ##\crcr
      ${\scriptstyle\hfil\;\;#1\;\;\hfil}$\crcr
      \noalign{\kern 1pt\nointerlineskip}
      \rightarrowfill\crcr}}\;}}
\def\bigtextmap#1{\mathop{\vbox{\ialign{
                                  ##\crcr
      ${\hfil\;\;#1\;\;\hfil}$\crcr
      \noalign{\kern 1pt\nointerlineskip}
      \rightarrowfill\crcr}}\;}}
      
\newcommand{\cal}{\mathcal}
\def\textlmap#1{\mathop{\vbox{\ialign{
                                  ##\crcr
      ${\scriptstyle\hfil\;\;#1\;\;\hfil}$\crcr
      \noalign{\kern-1pt\nointerlineskip}
      \leftarrowfill\crcr}}\;}}
 
\def\ag{{\mathfrak a}}
\def\bg{{\mathfrak b}}
\def\cg{{\mathfrak c}}

\def\eg{{\mathfrak e}}

\def\g{{\mathfrak g}}

\def\jg{{\mathfrak j}}

\def\ng{{\mathfrak n}}

\def\sg{{\mathfrak s}}
\def\tg{{\mathfrak t}}

\def\vg{{\mathfrak v}}

\def\zg{{\mathfrak z}}
\def\Ag{{\mathfrak A}}
\def\Bg{{\mathfrak B}}
\def\Cg{{\mathfrak C}}

\def\Mg{{\mathfrak M}}

\def\Pg{{\mathfrak P}}
\def\Rg{{\mathfrak R}}

\def\Sg{{\mathfrak S}}
\def\Tg{{\mathfrak T}}
\def\Ug{{\mathfrak U}}
\def\Vg{{\mathfrak V}}

\theoremstyle{remark}

\newtheorem{qu}{Question}

\newtheorem{conj}{Conjecture}
\newtheorem{sz}{Satz}[section]
\theoremstyle{remark}
\newtheorem{re}[sz]{Remark} 
\theoremstyle{plain}
\newtheorem{thry}[sz]{Theorem}
\newtheorem{pr}[sz]{Proposition}
\newtheorem{co}[sz]{Corollary}
\newtheorem{dt}[sz]{Definition}
\newtheorem{lm}[sz]{Lemma}

\def\End{\mathrm {End}}
\def\Aut{\mathrm {Aut}}
\def\Spin{\mathrm {Spin}}
\def\U{\mathrm{U}}
\def\SU{\mathrm {SU}}
\def\SO{\mathrm {SO}}
\def\PU{\mathrm {PU}}
\def\GL{\mathrm {GL}}
\def\SL{\mathrm {SL}}

\def\su{\mathrm {su}}

\def\Pic{\mathrm {Pic}}

\def\deg{\mathrm {deg}}
\def\Hom{\mathrm{Hom}}

\def\Tors{\mathrm{Tors}}
\def\T{\mathrm{T}}

\def\vol{\mathrm{vol}}

\def\id{ \mathrm{id}}
\def\im{\mathrm{im}}
\def\rk{\mathrm {rk}}

\def\si{\mathrm {si}}
\def\st{\mathrm {st}}
\def\pst{\mathrm{pst}}

\def\S{\mathrm{S}}
\def\ASD{\mathrm{ASD}}
\def\HE{\mathrm{HE}}

\def\reg{\mathrm{reg}}
\def\Ext{\mathrm{Ext}}

\newcommand\smvee{{\hskip -0.25ex \raise 0.2ex\hbox{$\scriptscriptstyle\vee$}}}
\def\we{{\smvee \hskip -0.2ex \smvee}}
\newcommand{\extpw}{\mathchoice{{\textstyle\bigwedge}}%
    {{\bigwedge}}%
    {{\textstyle\wedge}}%
    {{\scriptstyle\wedge}}}
\def\extp{{\extpw}\hspace{-2pt}}
\begin{document}

\title[On the Donaldson-Uhlenbeck compactification]{On the Donaldson-Uhlenbeck compactification of  instanton moduli spaces  on  class VII surfaces}
\author{Nicholas Buchdahl \and Andrei Teleman  \and Matei Toma}
\address{Nicholas Buchdahl: 
Department of Mathematics, University of Adelaide, Adelaide, 5005 Australia, email: nicholas.buchdahl@adelaide.edu.au }
\address{Andrei Teleman: Aix Marseille Université, CNRS, Centrale Marseille, I2M, UMR 7373, 13453 Marseille, France, email: andrei.teleman@univ-amu.fr }
\address{Matei Toma:  Institut de Mathématiques Elie Cartan, Université de Lorraine, B. P. 70239, 54506 Vandoeuvre-lès-Nancy Cedex, France
email: matei.toma@univ-lorraine.fr}
\begin{abstract} 
We study the following question: Let $(X,g)$ be a compact Gauduchon surface, $E$ be a differentiable rank $r$ vector bundle on $X$ and ${\cal D}$ be a fixed holomorphic structure on $D:=\det(E)$. Does the complex space structure on ${\cal M}_a^\ASD(E)^*$ induced by the Kobayashi-Hitchin correspondence   extend to a complex space structure on the Donaldson compactification $\overline{{\cal M}}_a^\ASD(E)$? Our results   answer this question in detail for the moduli spaces of $\SU(2)$-instantons with $c_2=1$ on   general (possibly unknown) class VII surfaces.
	
\end{abstract}

\maketitle

\tableofcontents
\section{Introduction}\label{intro}

\subsection{The moduli problem for vector bundles on compact complex manifolds}\label{KHsection}

An old, classical problem in complex geometry concerns the classification of the holomorphic structures on a fixed differentiable vector bundle on a complex manifold, up to isomorphism. The moduli problem for holomorphic vector bundles  is devoted to studying the corresponding set of isomorphism classes and the geometric structures (topologies, complex space structures, Hermitian metrics, etc)  one can naturally put on this set.  

More precisely, let $X$ be a compact complex manifold of dimension $n$ and $E$ be a differentiable vector bundle of rank $r$ on $X$. A semi-connection on $E$ is a first order operator $\delta:A^0(E)\to A^{0,1}(E)$ satisfying the $\bar\partial$-Leibniz rule: $\delta( f s)=\bar\partial f s+ f \delta(s)$ for any $f\in {\cal C}^\infty(X,\C)$ and $s\in A^0(E)$. The natural de Rham-type extension $A^{0,p}(E)\to A^{0,p+1}(E)$ will be denoted by the same symbol.   A semi-connection $\delta$ is called integrable if $F_\delta=0$, where $F_\delta\in A^{0,2}(\End(E))$ is the endomorphism-valued $(0,2)$-form defined by the composition $\delta\circ\delta: A^0(E)\to A^{0,2}(E)$.

Denote by ${\cal H}ol(E)$ the set of holomorphic structures on $E$, and  by ${\cal A}^{0,1}(E)$ (${\cal A}^{0,1}(E)^{\rm int}$) the space of (integrable) semi-connections on $E$. The Newlander-Nirenberg theorem gives a bijection ${\cal A}^{0,1}(E)^{\rm int}\to {\cal H}ol(E)$.  We will denote by ${\cal E}_\delta$ the holomorphic structure on $E$ associated with  an integrable semi-connection $\delta$. For an open set $U\subset X$ the space ${\cal E}_\delta(U)$ of holomorphic sections of $\resto{{\cal E}_\delta}{U}$ coincides with the kernel of the first order operator $\delta_U:A^0(U,E)\to A^{0,1}(U,E)$. Using this bijection and the natural ${\cal C}^\infty$-topology on $A^{0,1}(E)^{\rm int}$  we obtain a natural topology on the space ${\cal H}ol(E)$ of holomorphic structures on $E$. The moduli space of holomorphic structures on $E$ is the  topological space obtained as the quotient   %
$$\qmod{{\cal H}ol(E)}{\Aut(E)}=\qmod{{\cal A}^{0,1}(E)^{\rm int}}{\Aut(E)}$$
of ${\cal H}ol(E)$ by the group $\Aut(E):=\Gamma(X,GL(E))$ of differentiable automorphisms of $E$, which acts naturally on this space.  

Since $\GL(r)$ (the structure group of a vector bundle) is not a simple group, one considers a natural refinement of our moduli problem, which has a simple ``symmetry group".  Let ${\cal D}$ be a holomorphic structure on the determinant line bundle $D:=\det(E)$, and $\lambda\in {\cal A}^{0,1}(D)^{\rm int}$ be the corresponding integrable semi-connection. A ${\cal D}$-{\it oriented} holomorphic structure on $E$ is a holomorphic structure ${\cal E}$ on $E$ with $\det({\cal E})={\cal D}$.  

We recall that an $\SL(r,\C)$-vector bundle on $X$  is a rank $r$ vector bundle $E$ endowed with a trivialization of its determinant line bundle.  Therefore,  in this case, $D$ comes with a tautological trivial holomorphic structure $\Theta$; an $\SL(r,\C)$-{\it holomorphic structure} on $E$ is just a $\Theta$-oriented holomorphic structure in $E$ in our sense. This shows that classifying ${\cal D}$-oriented holomorphic structures on a fixed differentiable vector bundle gives the natural generalization of the classification   problem for $\SL(r,\C)$-holomorphic structures on an $\SL(r,\C)$-vector bundle.

The set ${\cal H}ol_{\cal D}(E)$ of ${\cal D}$-oriented holomorphic structures on $E$ can be identified with the subspace ${\cal A}_\lambda^{0,1}(E)^{\rm int}\subset{\cal A}^{0,1}(E)^{\rm int}$ defined by the condition $\det(\delta)=\lambda$.  The gauge group ${\cal G}^\C_E:=\Gamma(X,\SL(E))$ acts naturally on this subspace. We will focus on the moduli space of ${\cal D}$-oriented holomorphic structures on $E$, which, by definition, is the quotient space 
$${\cal M}_{\cal D}(E):=\qmod{{\cal H}ol_{\cal D}(E)}{{\cal G}^\C_E}=\qmod{{\cal A}_\lambda^{0,1}(E)^{\rm int}}{{\cal G}^\C_E}.
$$
 In general this quotient is highly  non-Hausdorff and cannot be endowed with a natural complex space structure.  It can be identified with the topological space associated with a holomorphic stack, but up till now it is not clear if, in  our   non-algebraic complex geometric framework, this approach leads to effective new results. On the other hand the classical point of view (studying moduli spaces whose points correspond to equivalence classes of holomorphic structures) has been  used with effective results, for instance in making progress towards the classification of class VII surfaces \cite{Te2}, \cite{Te4}.

We recall that a holomorphic structure ${\cal E}$ on $E$ is called {\it simple} if $H^0({\cal E}nd({\cal E}))=\C\id_E$. In contrast with ${\cal M}_{\cal D}(E)$, the moduli space ${\cal M}^\si_{\cal D}(E)$ of simple ${\cal D}$-oriented holomorphic structures has a natural, in general non-Hausdorff, complex space structure. This structure can be obtained in two different ways, which, a posteriori turn out to be equivalent. The first approach \cite{FK}  uses  classical deformation theory. The second \cite{LO} uses complex gauge theory. The equivalence of the two points of view has been established by Miyajima \cite{Miy}. 

As in algebraic geometry, in order to define Hausdorff moduli spaces of holomorphic structures, one needs a stability condition, which   depends on the choice of an additional structure on $X$. Whereas in algebraic geometry this additional structure is a polarisation on $X$ (i.e. the choice of   an ample line bundle on $X$), in our general complex geometric framework we need a Gauduchon metric on $X$, i.e. a Hermitian metric $g$ on $X$ whose associated  $(1,1)$-form $\omega_g$ satisfies the Gauduchon condition $dd^c(\omega_g^{n-1})=0$ \cite{Gau}. Such a metric defines a Lie group morphism
$$\deg_g:\Pic(X)\to\R,\ \deg_g({\cal L})=\int_X c_1({\cal L},h)\wedge \omega_g^{n-1},
$$  
where $c_1({\cal L},h)$ denotes the Chern form of the Chern connection associated with a Hermitian metric $h$ on  ${\cal L}$.  Using this degree map one can introduce a slope stability condition in the same way as in the algebraic framework:

  For a coherent sheaf ${\cal F}$ on $X$ we put $\deg_g({\cal F}):=\deg_g(\det({\cal F}))$. A non-zero torsion-free sheaf ${\cal F}$  on $X$   is called  (semi-)stable (with respect to $g$) if for every subsheaf ${\cal H}\subset {\cal F}$ with $0<\rk({\cal H}) < \rk({\cal F})$ one has
$$\frac{\deg_g({\cal H})}{\rk({\cal H})} < \frac{\deg_g({\cal F})}{\rk({\cal F})} \ \hbox{ respectively }\frac{\deg_g({\cal H})}{\rk({\cal H})} \leq \frac{\deg_g({\cal F})}{\rk({\cal F})} .
$$ 
 ${\cal F}$ is called polystable (with respect to $g$) if it splits as a direct sum ${\cal F}=\oplus_{i=1}^k {\cal F}_i$ of non-zero stable subsheaves ${\cal F}_i$ such that $\frac{\deg_g({\cal F}_i)}{\rk({\cal F}_i)}=\frac{\deg_g({\cal F})}{\rk({\cal F})}$ for $1\leq i\leq k$.

We denote by ${\cal M}^\si_{\cal D}(E)$,  ${\cal M}^\st_{\cal D}(E)$, ${\cal M}^\pst_{\cal D}$  the moduli  space of simple, stable, respectively polystable ${\cal D}$-oriented holomorphic structures on  $E$. Any stable bundle is simple and polystable, so ${\cal M}^\st_{\cal D}(E)\subset {\cal M}^\si_{\cal D}(E)$, ${\cal M}^\st_{\cal D}(E)\subset {\cal M}^\pst_{\cal D}(E)$.  Stability is an open condition {\it with respect to the classical topology} \cite{LT}, hence ${\cal M}^\st_{\cal D}(E)$ is open in both ${\cal M}^\si_{\cal D}(E)$, and  ${\cal M}^\pst_{\cal D}(E)$. In particular  ${\cal M}^\st_{\cal D}(E)$ comes with a natural complex space space structure induced from ${\cal M}^\si_{\cal D}(E)$. 

Note that in general, in the non-Kählerian framework, semi-stability is not an open condition (even with respect to the classical topology).\\

By choosing a Hermitian metric $h$ on $E$, ${\cal M}^\pst_{\cal D}(E)$ can be identified with the  moduli space of projectively Hermitian-Einstein connections on $E$, defined as follows. Denote  by ${\cal A}(E)$ the space of Hermitian connections on $(E,h)$. For a fixed Hermitian connection $a$ on the Hermitian line bundle $(D,\det(h))$ we put
$${\cal A}_a(E):=\{A\in {\cal A}(E)|\ \det(A)=a\}.$$
An element of  ${\cal A}_a(E)$ will be called an $a$-{\it oriented connection} on $(E,h)$. We define the space and the {\it moduli space of projectively Hermitian-Einstein $a$-oriented connections on} $(E,h)$  by 
$${\cal A}^\HE_a(E):=\{A\in {\cal A}_a(E)|\ \Lambda_g F_A^0=0,\ (F_A^0)^{0,2}=0\}\ ,\ {\cal M}^{\HE}_a(E):=\qmod{{\cal A}^\HE_a(E)}{{\cal G}_E},
$$
where $F_A^0$ stands for the trace-free part of the curvature $F_A$ of $A$, and  ${\cal G}_E:=\Gamma(X,\SU(E))$ is the $\SU(r)$-gauge group of $(E,h)$. If $n=2$,  the conditions $\Lambda_g F_A^0=0$, $(F_A^0)^{0,2}=0$ are equivalent to the anti-selfduality condition $(F_A^0)^+=0$, so ${\cal M}^{\HE}_a(E)$ is just the moduli space ${\cal M}^\ASD_a(E)$ of projectively ASD, $a$-oriented  connections on $(E,h)$.

The Kobayashi-Hitchin correspondence \cite{Bu1}, \cite{LT}, \cite{LY} generalising Donaldson's fundamental work \cite{D1} states that 
\newtheorem*{KH*}{Theorem}
\begin{KH*}[The Kobayashi-Hitchin correspondence]\label{KH} Let ${\cal D}$ be a fixed holomorphic structure on $D=\det(E)$ and let $a$  be the Chern connection of the pair $({\cal D},\det(h))$. 	The map $A\mapsto {\cal E}_{\bar\partial_A}$ induces a homeomorphism 
$${\cal M}^\HE_a(E)\textmap{\simeq} {\cal M}^\pst_{\cal D}(E)\,,$$
 which restricts to a homeomorphism ${\cal M}^\HE_a(E)^*\textmap{\simeq} {\cal M}^\st_{\cal D}(E)$ between the moduli space of irreducible projectively Hermitian-Einstein, $a$-oriented connections on $E$,  and the moduli space of stable ${\cal D}$-oriented holomorphic structures on $E$. 	
\end{KH*}

\begin{re} In the special case when $E$ is an $\SL(r,\C)$-bundle we can choose $h$ such that the distinguished trivialisation of $\det(E)$ is unitary. With this choice $(E,h)$ becomes an $\SU(r)$-vector bundle, $a$  coincides with the trivial connection associated with this trivialisation,   ${\cal A}_a(E)$ is  the space of $\SU(r)$-connections on $(E,h)$, and ${\cal M}^\HE_a(E)$ is  the moduli space of Hermitian-Einstein $\SU(r)$-connections on $(E,h)$.   In the case of an $\SL(r,\C)$-bundle $E$ (or of  an $\SU(r)$-bundle $(E,h)$) we will omit the subscripts ${\cal D}$, $a$ in our notation, hence we will write ${\cal M}^\pst(E)$, ${\cal M}^\st(E)$  for the moduli spaces of polystable (stable) $\SL(r,\C)$-structures on $E$, and ${\cal M}^\HE(E)$, ${\cal M}^\HE(E)^*$ for the moduli spaces of (irreducible) Hermitian-Einstein $\SU(r)$-connections on $(E,h)$.
	\end{re}

The Kobayashi-Hitchin correspondence has important consequences: using standard gauge-theoretical techniques one can prove that the quotient topology of  ${\cal M}^\HE_a(E)$ is Hausdorff. The point is that, for this moduli space, the gauge group ${\cal G}_E$ is the group of sections of a locally trivial bundle with compact standard fibre. Therefore  ${\cal M}^\pst_{\cal D}(E)$, ${\cal M}^\st_{\cal D}(E)$ are Hausdorff spaces, and in particular ${\cal M}^\st_{\cal D}(E)$ is a Hausdorff complex subspace of the (possibly non-Hausdorff) moduli space ${\cal M}^\si_{\cal D}(E)$.  

\subsection{Extending the complex  structure to a compactification of ${\cal M}^\st_{\cal D}(E)$. The fundamental questions} \label{FundamQ}

The first natural question related to the Kobayashi-Hitchin correspondence is: {\it Does ${\cal M}^\pst_{\cal D}(E)$  have a natural complex space structure extending the canonical complex space structure of ${\cal M}^\st_{\cal D}(E)$?} The explicit examples described in \cite{Te2}, \cite{Te4} and the general results proved in \cite{Te5} show that in general the answer is negative. For instance, for a class VII surface $X$ with $b_2(X)=1$, a bundle $E$ with $c_2(E)=0$, $c_1(E)=c_1({\cal K}_X)$ and a suitable Gauduchon metric $g$ on $X$, the moduli space ${\cal M}^\pst_{\cal D}(E)$ can be identified with a compact disk, whose interior corresponds to  ${\cal M}^\st_{\cal D}(E)$. Moreover, on a class VII surface $X$  with $b_2(X)=2$ one obtains in a similar way  a moduli space ${\cal M}^\pst_{\cal D}(E)$ which can be identified with $S^4$, and in this case  ${\cal M}^\st_{\cal D}(E)$ corresponds to the complement of the union of two circles in $S^4$. These examples show that, in the general (possibly non-Kählerian) framework the complex space structure on ${\cal M}^\st_{\cal D}(E)$ does not extend to a complex space structure on ${\cal M}^\pst_{\cal D}(E)$. On the other hand, in general ${\cal M}^\pst_{\cal D}(E)$ is not compact, and when this is the case, even if the complex structure of ${\cal M}^\st_{\cal D}(E)$ does extend to ${\cal M}^\pst_{\cal D}(E)$, the result is not satisfactory. This motivates the following:
\begin{qu} \label{QQ1} Let $(X,g)$ be a compact Gauduchon manifold, $E$ be a differentiable rank $r$ vector bundle on $X$ and ${\cal D}$ be a fixed holomorphic structure on $D:=\det(E)$. Does the complex space structure of ${\cal M}^\st_{\cal D}(E)$ extend to a complex space structure on a natural compactification of it, which contains the space ${\cal M}^\pst_{\cal D}(E)$?	
\end{qu}

Note that, in the case $n=2$, one has a good candidate for ``a natural compactification of ${\cal M}^\st_{\cal D}(E)$ which contains ${\cal M}^\pst_{\cal D}(E)$": we identify ${\cal M}^\st_{\cal D}(E)$ with ${\cal M}_a^\ASD(E)^*$ via the Kobayashi-Hitchin isomorphism, and we embed the latter space in the Donaldson compactification of ${\cal M}_a^\ASD(E)$. Therefore, in the case $n=2$, one can ask a more precise version of Question 1: 
\begin{qu}   \label{QQ2}  Let $(X,g)$ be a compact Gauduchon surface, $E$ be a differentiable rank $r$ vector bundle on $X$ and ${\cal D}$ be a fixed holomorphic structure on $D:=\det(E)$. Does the complex space structure on ${\cal M}_a^\ASD(E)^*$ induced by the Kobayashi-Hitchin correspondence   extend to a complex space structure on the Donaldson compactification $\overline{{\cal M}}_a^\ASD(E)$?
\end{qu}

As the examples above show, in our general (possibly non-Kählerian) framework both questions have negative answer. Indeed, in our examples the space ${\cal M}_a^\ASD(E)$ is already compact, and admits no complex space structure at all.

On the other hand we believe that both questions have   positive answer when $(X,g)$ is Kähler. In other words
\begin{conj}\label{Conj1} Let $(X,g)$ be a compact Kähler manifold, $E$ be a differentiable rank $r$ vector bundle on $X$ and ${\cal D}$ be a fixed holomorphic structure on $D:=\det(E)$. Then the natural complex space structure of ${\cal M}^\st_{\cal D}(E)$ (induced from  ${\cal M}^\si_{\cal D}(E)$) extends to a natural compactification of it which contains ${\cal M}^\pst_{\cal D}(E)$. For $n=2$ the natural complex space structure of ${\cal M}^\st_{\cal D}(E)$ extends to the Donaldson compactification $\overline{{\cal M}}_a^\ASD(E)$ of ${\cal M}_a^\ASD(E)$. 
	\end{conj}

This conjecture is known to hold in the projective algebraic framework (when $\omega_g$ is the Chern form of an ample line bundle ${\cal H}$ on $X$) for $\SL(2,\C)$-bundles. In this case, for $n=2$ one proves \cite{Li} that $\overline{{\cal M}}_a^\ASD(E)$  can be identified with the image in a projective space of a regular map defined on a Zariski closed subset of the Gieseker moduli space associated with the data $c_1=0$, $c_2=c_2(E)$, $r=2$. This Gieseker moduli space is a projective variety. Taking into account Li's result, and the fact that any compact Kähler surface admits arbitrary small deformations which are projective (\cite{Kod1}, \cite{Bu3}, \cite{Bu4}), Conjecture \ref{Conj1} becomes very natural. The recent results of \cite{GT} concern the higher dimensional projective case, and give further evidence for this conjecture.  \\

In this article we will study in detail Question 2  in an interesting special case: the moduli space of $\SL(2,\C)$-structures on an $\SL(2,\C)$-bundle $E$ with  $c_2(E)=1$ on a class VII surface $X$.  One has 
\begin{equation}\label{FirstDec}
\begin{split}
\overline{{\cal M}}^\ASD(E)&= {\cal M}^\ASD(E)\cup \big({\cal M}_0\times X\big)\\
&= {\cal M}^\ASD(E)^*\cup {\cal R} \cup \big({\cal M}_0^*\times X\big) \cup \big ({\cal R}_0\times X \big),
\end{split}	
\end{equation}
where ${\cal R}\subset {\cal M}^\ASD(E)$ is the subspace of reducible $\SU(2)$-instantons with $c_2=1$, and ${\cal M}_0$ (${\cal M}_0^*$, ${\cal R}_0$) stands for the moduli space of flat, respectively flat irreducible,   flat reducible, $\SU(2)$-instantons. The last three terms  in the decomposition (\ref{FirstDec}) are compact.  Therefore in our case,   Question 2   reduces to a set of three more specific questions: 
\vspace{2mm}\\
{\it Does the complex space structure of  ${\cal M}^\ASD(E)^*={\cal M}^\st(E)$ extend across   the  compact strata (a)    ${\cal R}_0\times X$, (b) ${\cal M}_0^*\times X$, (c) ${\cal R}$ ?}
\vspace{2mm}

The factor ${\cal R}_0$ of the third summand in \ref{FirstDec} can be further decomposed as a disjoint union as follows. Denote by $\mathrm{C}(X)$ the group  of characters $\Hom(H_1(X,\Z),S^1)$. One has a natural identification
$$\Cg(X):=\qmod{\mathrm{C}(X)}{\langle\jg\rangle} \textmap{\simeq}  {\cal R}_0,
$$
 where $\jg$ is the involution induced by the conjugation $S^1\to S^1$. The group    $\mathrm{C}(X)$ fits into the  following commutative diagram with exact rows:
\begin{diagram}[h=7mm]
0&\rTo &\mathrm{C}^0(X)\simeq S^1 &\rTo & \mathrm{C}(X) &\rTo&\Tors (H^2(X,\Z))&\rTo & 0\phantom{\,, }\\
 &&\dInto  && \dInto && \parallel &\\
0&\rTo &\Pic^0(X)\simeq \C^* &\rTo & \Pic^\T(X) &\rTo^{c_1}&\Tors (H^2(X,\Z))&\rTo &0\,,	
\end{diagram}
where
\begin{align*}
\mathrm{C}^0(X)&:=\Hom(H_1(X,\Z)/\Tors,S^1)\simeq S^1,\\
\Pic^\T(X)&:=\{[{\cal L}]\in\Pic(X)|\ c_1({\cal L})\in \Tors (H^2(X,\Z))\}.	
\end{align*}
The central vertical monomorphism  maps  a character $\chi\in \mathrm{C}(X)$ to the isomorphy class of the associated holomorphic Hermitian line bundle ${\cal L}_\chi$, and  identifies $\mathrm{C}(X)$ with   $\ker(\resto{\deg_g}{\Pic^\T(X)})$, which  is independent of the Gauduchon metric $g$. More precisely the congruence class $\mathrm{C}^c(X)\in \mathrm{C}(X)/\mathrm{C}^0(X)$ associated  with a torsion class $c\in \Tors (H^2(X,\Z))$ is identified with  the vanishing circle of $\deg_g$ on the congruence class $\Pic^c(X)\in\Pic(X)/\Pic^0(X)$.

The involution $\jg:\mathrm{C}(X)\to \mathrm{C}(X)$ maps $\mathrm{C}^c(X)$ diffeomorphically onto $\mathrm{C}^{-c}(X)$, in particular leaves invariant any circle $\mathrm{C}^c(X)$ associated with a class $c$ belonging to the $\Z_2$-vector space 
$$\Tors_2(H^2(X,\Z)):=\ker \big(H^2(X,\Z)\textmap{2\cdot} H^2(X,\Z)\big).$$
Let $\mu_2$ be the multiplicative group $\{\pm1\}$. For  $c\in \Tors_2(H^2(X,\Z))$ the set $\rho^c(X)$ of fixed points of the induced involution  $\mathrm{C}^c(X)\textmap{\jg^c} \mathrm{C}^c(X)$ is a $\mu_2$-torsor (in particular has two points), and the quotient 
$$\Cg^c(X):=\qmod{\mathrm{C}^c(X)}{\langle \jg_c\rangle}$$
is a segment. More precisely, the choice of a fixed point $l\in \rho^c(X)$ gives a homeomorphism 
$$\Cg^c(X)\textmap{h_l \simeq}\qmod{S^1}{z\mapsto \bar z}\simeq[-1,1].
$$
Let $\Tg(X)$ be the quotient of $\Tors(H^2(X,\Z))$ by the involution $c\mapsto -c$, and denote by $\Tg_0(X)$, $\Tg_1(X)$ the subsets of $\Tg(X)$ which correspond respectively to $\Tors_2(H^2(X,\Z))$ and  $\Tors(H^2(X,\Z))\setminus \Tors_2(H^2(X,\Z))$. 

For a class $\cg=[c]\in\Tg(X)$ denote by $\Cg^\cg(X)$ the quotient of $\mathrm{C}^c(X)\cup \mathrm{C}^{-c}(X)$ by the involution induced by $\jg$ on this union. Note that for  $\cg=\{c\}\in \Tg_0(X)$ one has  $\Cg^\cg(X)=\Cg^c(X)$ whereas,  for $\cg\in \Tg_1(X)$, the choice of a representative $c\in\cg$ gives an identification $\Cg^\cg(X)\simeq \mathrm{C}^c(X)$. Therefore the space ${\cal R}_0$ of reducible flat instantons on $X$  decomposes as a  finite disjoint union of segments and  circles:
\begin{equation}\label{SegmCirc}
{\cal R}_0\simeq\Cg(X)=\union_{\cg\in \Tg(X)} \Cg^\cg	(X)= \bigg( \union_{\cg\in \Tg_0(X)} \Cg^\cg(X)\bigg)\union  \bigg(\union_{\cg\in \Tg_1(X)} \Cg^\cg	(X)\bigg).
\end{equation}
The set of all vertices appearing in the first term of  (\ref{SegmCirc}) can be identified with the set of fixed points of $\jg$, which coincides with 
$$\rho(X):=H^1(X,\mu_2)=\Hom(H_1(X,\Z),\mu_2)=\union_{c\in \Tors_2(H^2(X,\Z))}\rho^c(X)\,,$$
 and has an important gauge theoretical interpretation: it corresponds  to the moduli space   of reducible flat $\SU(2)$-instantons with {\it non-abelian stabiliser} $\SU(2)$.  All other points of ${\cal R}_0$ correspond to reducible flat $\SU(2)$-instantons with abelian stabiliser $S^1$. Therefore, from a gauge theoretical point of view, the first term in the decomposition (\ref{SegmCirc}) (which is always non-empty) is  substantially more complex than the second, because a segment $\Cg^c(X)$  contains both abelian and non-abelian reductions. Our main result states
\begin{thry}\label{Th1}
Let $X$ be a class VII surface, and let $E$ be  an $\SL(2,\C)$-bundle on $X$ with $c_2(E)=1$. The complex space structure of ${\cal M}^\ASD(E)^*={\cal M}^\st(E)$ extends across ${\cal R}_0\times X$, and $\overline{{\cal M}}^\ASD(E)$ is a smooth complex 4-fold at any reducible virtual point $([A],x)\in {\cal R}_0\times X$. 	
\end{thry}
This result is surprising for two reasons: first, ${\cal R}_0$ is a union of segments and circles, which are not complex geometric objects; second, since  the vertices of the segments $\Cg^\cg(X)$ ($\cg\in\Tg_0(X)$) are isolated non-abelian reductions, one expects essential singularities at   these points. 

Our second result concerns the extensibility  of the complex space structure across ${\cal M}_0^*\times X$:  
\begin{thry}\label{Th2}
Let $X$ be a class VII surface, and let $E$ be be an $\SL(2,\C)$-bundle on $X$ with $c_2(E)=1$. Then ${\cal M}_0^*$ consists of finitely many simple points, and the complex space structure of ${\cal M}^\ASD(E)^*={\cal M}^\st(E)$ extends across ${\cal M}_0^*\times X$.  For every $[A]\in {\cal M}_0^*$ the surface $\{[A]\}\times X$ has an open neighbourhood in  $\overline{\cal M}^\ASD(E)$ which is a normal complex space whose singular locus is $\{[A]\}\times X$, and the normal cone of this singular locus can be identified with   the cone bundle of degenerate elements in ${\cal K}_X^\smvee\otimes S^2({\cal E})$, where ${\cal E}$ is the holomorphic bundle associated with $A$.  	
\end{thry}
Here we denoted by $S^2({\cal E})$ the second symmetric power of ${\cal E}$; an element $\eta\otimes \sigma\in {\cal K}_X(x)^\smvee\otimes S^2({\cal E}(x))$ is degenerate  if the associated linear map ${\cal E}(x)\to {\cal E}(x)^\smvee\otimes {\cal K}_X(x)^\smvee$ has non-trivial kernel. Since $\rk({\cal E})=2$, this   is equivalent to the condition that  $\sigma$ belongs to the image of the squaring map ${\cal E}(x)\to S^2({\cal E}(x))$.

Finally, the extensibility of the    complex space structure across the subspace ${\cal R}\subset {\cal M}^\ASD(E)$ has been studied in detail in a more general framework in \cite{Te5}. In our special case the result is the following

 \begin{thry}\label{Th3} \cite{Te5}
 Let $X$ be a class VII surface endowed with a Gauduchon metric $g$ with $\deg_g({\cal K}_X)<0$, and let $E$ be  an $\SL(2,\C)$-bundle on $X$ with $c_2(E)=1$. Then ${\cal M}^\ASD(E)^*={\cal M}^\st(E)$ is a smooth 4-fold, and the reduction locus ${\cal R}\subset {\cal M}^\ASD(E)$ is	a union of $b_2(X)|\Tors(H^2(X,\Z))|$ circles. Any such circle  has a  neighbourhood which can be identified with a neighbourhood of the singular circle in a  flip passage; in particular the holomorphic structure  of ${\cal M}^\ASD(E)^*$ does not extend across any of these circles.  
 \end{thry}
 In other words, for class VII surfaces with $b_2>0$, the holomorphic structure does {\it not} extend across the circles of reductions in the moduli space, but (supposing  $\deg_g({\cal K}_X)<0$) the structure of the moduli space around such a circle is perfectly understood. Note that the condition $\deg_g({\cal K}_X)<0$ is not restrictive. Indeed, using the classification of class VII surfaces with $b_2=0$ \cite{Te1} and the results of \cite{Bu2} (see also \cite[Lemma 2.3]{Te4}),  it follows that:
 \begin{re}
 Any class VII surface whose minimal model is not an Inoue surface admits Gauduchon metrics $g$ such that $\deg_g({\cal K}_X)<0$.
 \end{re}
 
 A surprising corollary of our results is:
\begin{co}\label{PrimHopf}  Let $(X,g)$ be a primary Hopf surface endowed with a Gauduchon metric, and $E$ be an $\SL(2,\C)$-bundle on $X$ with $c_2(E)=1$. Then the natural complex structure 	on ${\cal M}^\st(E)$ is smooth and extends to a complex  structure on  $\overline{{\cal M}}^\ASD(E)$, which becomes a 4-dimensional compact complex  manifold. 
\end{co}

This study of  moduli spaces  of $\SU(2)$-instantons on class VII surfaces has several motivations. First, in recent articles the second author showed that $\PU(2)$-instanton moduli spaces can be used to make progress on the classification of class VII surfaces, more precisely to prove the existence of curves on such surfaces  \cite{Te2}, \cite{Te4}. A natural question is: can one obtain similar (or even stronger) results using  moduli spaces of $\SU(2)$-instantons? In order to follow this strategy one  needs a thorough understanding of {\it compactified} such moduli spaces. 
 
 A second motivation is related to   Corollary \ref{PrimHopf}:  according to this result, the assignment $(X,g)\mapsto \overline{{\cal M}}^\ASD(E)$ defines a functor from the category of Gauduchon primary Hopf surfaces to the class of 4-dimensional smooth compact complex  manifolds. Moreover, it is known that ${\cal M}^\ASD(E)^*={\cal M}^\st(E)$ is endowed with a canonical Hermitian metric $\g$ which is strongly KT, i.e., satisfies $\partial\bar\partial\omega_\g=0$ \cite{LT}. The class of compact strongly KT Hermitian manifolds has been intensively studied in recent years. This class of manifolds intervenes in modern physical theories (II string theory, 2-dimensional supersymmetric $\sigma$-models) and also in  Hitchin's  theory of generalised Kähler geometry. Therefore, it is natural to ask
\begin{qu}\label{KT}
In the conditions of Corollary \ref{PrimHopf} does the canonical strongly KT metric on ${{\cal M}}^\ASD(E)^*={\cal M}^\st(E)$	 extend to a smooth Hermitian metric on the complex 4-fold $\overline{{\cal M}}^\ASD(E)$?
\end{qu}

If this question has a positive answer, the resulting metric on $\overline{{\cal M}}^\ASD(E)$ will be strongly KT, giving   an interesting functor from the category of Gauduchon primary Hopf surfaces to the category of compact strongly KT 4-dimensional manifolds.  This functor would yield a large class of  examples of 4-dimensional strongly KT  compact  manifolds. We will come back to Question \ref{KT} in a future article.

A third motivation: the novelty of the methods used in the proofs, which emphasise   surprising difficulties which occur in the non-Kählerian framework.  We shall explain the ideas of proofs and these difficulties in the next subsections, and in the course of the subsequent proofs themselves, it will be apparent  that our methods will be applicable in many other situations.

\subsection{The idea of proof of Theorem \ref{Th1}}
\label{IdeaTh1}

In the first part of the proof we study  $\overline{{\cal M}}^\ASD(E)$ from the  topological point of view. We show that any reducible  flat  $\SU(2)$-instanton $A$ on a class VII surface is regular, i.e. one has $\H^2_A=0$. Using the   local model theorem for virtual instantons \cite[Theorem 8.2.4]{DK}, we will prove that   $\overline{{\cal M}}^\ASD(E)$ is a topological 8-manifold around any virtual point $([A],x)\in {\cal R}_0\times X$. Denoting by  ${\cal M}^\ASD(E)^*_{\reg}\subset {\cal M}^\ASD(E)$ the open subspace  of  regular irreducible instantons, we see that the subspace 
$$\Mg:={\cal M}^\ASD(E)^*_{\reg}\cup \big({\cal R}_0\times X\big)\subset \overline{{\cal M}}^\ASD(E)$$
 is a topological 8-manifold.

In a second step we will construct a complex manifold structure on  the topological manifold $\Mg$ using a gluing construction based on the following simple result proved in the appendix:

\newtheorem*{th-glue}{Lemma \ref{glue}} 
\begin{th-glue}
Let ${\cal X}$ be a topological $2n$-dimensional manifold, and ${\cal Y}\subset {\cal X}$ be an open subset endowed with a complex manifold structure. Let ${\cal U}$ be an $n$-dimensional complex manifold, and $f:{\cal U}\to {\cal X}$ be a continuous, injective map with the properties:
\begin{itemize}
\item ${\cal X}\setminus{\cal Y}\subset \im(f)$,
\item The restriction	$\resto{f}{f^{-1}({\cal Y})}:f^{-1}({\cal Y})\to {\cal Y}$ is holomorphic with respect to the holomorphic structure induced by the open embedding  $f^{-1}({\cal Y})\subset{\cal U}$.
\end{itemize}
Then 
\begin{enumerate}
\item $\im(f)$ is an open neighbourhood of ${\cal X}\setminus {\cal Y}$ in ${\cal X}$,
\item $f$ induces a homeomorphism ${\cal U}\to \im(f)$, 
\item $f$ induces a biholomorhism $f^{-1}({\cal Y})\to \im(f)\cap {\cal Y}$ 	 with respect to the holomorphic structures induced by the open embeddings  $f^{-1}({\cal Y})\subset{\cal U}$, $\im(f)\cap {\cal Y}\subset {\cal Y}$,
\item There exists a unique complex manifold structure on ${\cal X}$ which extends the fixed complex structure on ${\cal Y}$, and such that $f$ becomes biholomorphic on its image.

\end{enumerate}
\end{th-glue}

The hard part of the proof of Theorem \ref{Th1} is the construction of a pair $({\cal U},f:{\cal U}\to \Mg)$ such that, taking ${\cal Y}:={\cal M}^\ASD(E)^*_\reg$ with the complex structure induced from ${\cal M}^\st(E)_\reg$, the hypothesis of Lemma \ref{glue} is fulfilled. 

Note first that, surprisingly, the Gieseker (semi)stability condition for   torsion-free sheaves, can be naturally extended to arbitrary Gauduchon surfaces. The point is that, for surfaces, in these conditions only the degree of the sheaf and its  Euler-Poincaré characteristic intervene \cite[p. 97]{Fr}. We agree to call the   torsion-free sheaves on $X$, which are not locally free, {\it singular sheaves}.

In our situation we obtain  a natural homeomorphism $\varphi:{\cal S}\textmap{\simeq}{\cal R}_0\times X$ between the moduli space ${\cal S}$ of singular rank 2 Gieseker stable sheaves ${\cal F}$ on $X$ with the properties
\begin{itemize}
\item  $\det({\cal F})\simeq{\cal O}_X$, $c_2({\cal F})=1$,	
\item the double dual  ${\cal F}^\we$ is properly semi-stable.
\end{itemize}
 Taking into account this identification, a natural choice would be to take for ${\cal U}$ an open neighbourhood of  ${\cal S}$ in the moduli space of  all (locally free and singular) rank 2 Gieseker stable sheaves (with trivial determinant and $c_2=1$). Unfortunately in our non-Kählerian framework an unexpected difficulty arises: {\it Gieseker stability is not an open condition.} More precisely, in our case, any neighbourhood (in the moduli space of simple rank 2 sheaves with trivial determinant) of a point $[{\cal F}]\in {\cal S}$  contains points corresponding  to  sheaves  which are not even slope semi-stable.

Taking into account this difficulty, we will take ${\cal U}$ to be a sufficiently small open neighbourhood of ${\cal S}$ in the moduli space ${\cal M}^\si$ of simple rank 2 sheaves \cite{KO} with trivial determinant.   We will define  a map
$$f:{\cal U}\to \Mg
$$
whose restriction to ${\cal S}$ coincides with $\varphi$ and whose restriction to ${\cal U}\cap{\cal M}^\st(E)_\reg$ coincides with the identity of this set. For a point $[{\cal F}]\in {\cal U}$ with ${\cal F}$ not semi-stable  we put $f([{\cal F}]):=[{\cal E}_{\cal F}]$, where ${\cal E}_{\cal F}$ is a slope stable locally free sheaf which is determined up to isomorphy by the conditions 
$$H^0({\cal H}om({\cal F},{\cal E}_{\cal F}))\ne 0\ ,\ H^0({\cal H}om({\cal E}_{\cal F},{\cal F}))\ne 0.
$$
The proof will be completed by showing that these properties determine a well-defined map $f:{\cal U}\to \Mg$  satisfying the assumptions of Lemma \ref{glue}.

\subsection{The idea of proof of Theorem \ref{Th2}} 
\label{IdeaTh2}
According to the main result of \cite{Pl} the moduli space ${\cal M}_0^*$ consists of finitely many simple points. Let $[A]\in {\cal M}_0^*$ be a flat irreducible $\SU(2)$-instanton,  ${\cal E}$ be the associated stable holomorphic bundle, and $\pi:\P({\cal E})\to X$ its projectivisation. For a point $x\in X$ and a line  $y\in \P({\cal E}(x))$ we denote by $\eta_y:{\cal E}(x)\to q_y:={\cal E}(x)/y$ the corresponding epimorphism, and we put
$${\cal F}_y:=\ker \big({\cal E}\to {\cal E}_{\{x\}}\textmap{\eta_y}q_y\otimes{\cal O}_{\{x\}}\big)\,.$$ 
 We will construct a torsion-free sheaf $\mathscr{F}$ on $\P({\cal E})\times X$, flat over $\P({\cal E})$, such that for any point $y\in\P({\cal E})$ the   restriction $\resto{\mathscr{F}}{\{y\}\times X}$, regarded as a sheaf on $X$, is isomorphic with ${\cal F}_y$. The sheaf $\mathscr{F}$ defines an embedding  $\P({\cal E})\to {\cal M}^\si$ in the moduli space of simple, torsion-free sheaves with trivial determinant and $c_2=1$. The normal line bundle of the image ${\cal P}$ of this embedding can be computed explicitly. We will prove that  ${\cal P}$  has an open neighbourhood ${\cal U}$ such that ${\cal U}\setminus{\cal P}\subset {\cal M}^\st(E)_\reg$. On the other hand, using Fujiki's contractibility criterion \cite{Fuj}, it follows that, there exists  a modification ${\cal U}\to {\cal V}$ on a singular complex space ${\cal V}$, which contracts the projective fibres of ${\cal P}$. Using the continuity theorem \cite{BTT} we obtain a continuous map ${\cal U}\to \overline{{\cal M}}^\ASD(E)$ which induces a homeomorphism between the complex space ${\cal V}$ and an open neighbourhood of $\{[A]\}\times X$ in $\overline{{\cal M}}^\ASD(E)$.\\  
 
The article is organised as follows:  Section \ref{TopStrVirtSect} contains results on the topology $\overline{{\cal M}}^\ASD(E)$ around the virtual locus. These results are obtained using gauge theoretical methods.  Using the constructions given in  Section \ref{FamSection}, we construct in Section \ref{VSection}  a holomorphic embedding  $V_\varepsilon:\Ag_\varepsilon\times X\to {\cal M}^\si$, where $\Ag_\varepsilon$ is an open neighbourhood of $\Cg(X)$ in the quotient of $\Pic^\T(X)$ by the involution $l\mapsto l^\smvee$. This embedding plays a crucial role in  the proof of Theorem \ref{Th1}. The proofs of  Theorem \ref{Th1} and  Theorem \ref{Th2} are given in Section \ref{ExtendingSect}. The appendix groups together general results needed in the proofs; many of these results are new, and are useful in many other situations.

\section{The topological structure of $\overline{{\cal M}}^\ASD(E)$ at the virtual points}

\label{TopStrVirtSect}

\subsection{Local models at the reducible virtual points}
\label{LocModVirtRedSect}

We start with the following regularity result
\begin{pr}\label{RegFlatPolyst}
Let $X$ be a class VII surface endowed with a Gauduchon metric $g$, and ${\cal E}$  be a holomorphic, split polystable $\SL(2,\C)$-bundle on $X$ with $c_2({\cal E})=0$. Then $H^2({\cal E}nd_0({\cal E}))=0$. 	
\end{pr}
\begin{proof}
The hypothesis implies ${\cal E}={\cal L}\oplus{\cal L}^\smvee$, where ${\cal L}$ is a holomorphic line bundle on $X$ with $\deg_g({\cal L})=0$ and $c_1({\cal L})^2=0$. Since the intersection form $q_X$ is negative definite the latter condition implies $c_1({\cal L})\in\Tors(H^2(X,\Z))$ so, using the notation introduced in Section \ref{FundamQ}, $[{\cal L}]\in \mathrm{C}^c(X)$ for a class $c\in \Tors(H^2(X,\Z))$. Note that ${\cal E}nd_0({\cal E})\simeq {\cal O}_X\oplus{\cal L}^{\otimes 2}\oplus {\cal L}^{\smvee\otimes 2}$  and, since $X$ is a class VII surface, $h^2({\cal O}_X)=h^0({\cal K}_X)=0$. The result follows from Lemma \ref{H2FlatLB} below. 
\end{proof}
\begin{lm}\label{H2FlatLB} Let $X$ be a class VII surface, and $[{\cal M}]\in \mathrm{C}(X)$. Then $H^2({\cal M})=0$.	
\end{lm}
\begin{proof}
By Serre duality we have $h^2({\cal M})=h^0({\cal K}_X\otimes{\cal M}^\smvee	)$. On the other hand, since ${\cal M}$ is associated with a character $\chi:H_1(X,\Z)\to S^1$, it follows that its Chern class $c_1^{\mathrm{BC}}({\cal M})$ in Bott-Chern cohomology vanishes. Therefore it suffices to prove that, on  a class VII surface $X$, the Bott-Chern class  $c_1^{\mathrm{BC}}({\cal K}_X)\in  H^{1,1}_{\mathrm{BC}}(X,\R)$ is not represented by an effective divisor. Let $\pi:X\to X_{\min}$ be the projection of $X$ on its minimal model. If a Bott-Chern class $\cg\in H^{1,1}_{\mathrm{BC}}(X,\R)$ is represented by an effective divisor $D$, then $\pi_*(\cg)$ will be represented by the effective divisor $\pi_*(D)$, hence it suffices to prove that the class
$$\pi_*\big(c_1^{\mathrm{BC}}({\cal K}_X)\big)=c_1^{\mathrm{BC}}({\cal K}_{X_{\min}})$$
is not represented by an effective divisor on $X_{\min}$.
\vspace{2mm}\\
{\it Case 1.} $X_{\min}$ is an Inoue surface. An Inoue surface has no curve so, if $c_1^{\mathrm{BC}}({\cal K}_{X_{\min}})$ were represented by an effective divisor,  this class would vanish.  But  $c_1^{\mathrm{BC}}({\cal K}_{X_{\min}})$ is non-trivial and pseudo-effective (see \cite[Remark 4.2]{Te3}).
\vspace{2mm}\\
{\it Case 2.} $X_{\min}$ is a Hopf surface. Any primary Hopf surface $H$ contains a non-trivial anti-canonical  effective divisor \cite[p. 696]{Kod2}. In other words one has ${\cal K}_H\simeq {\cal O}_H(-D)$ where $D>0$.  It follows that for any Hopf surface $H$ the class  $-c_1^{\mathrm{BC}}({\cal K}_{X_{\min}})$ is non-trivial and pseudo-effective, so $c_1^{\mathrm{BC}}({\cal K}_{X_{\min}})$ cannot be represented by an effective divisor.
\vspace{2mm}\\
{\it Case 3.} $X_{\min}$ is a minimal class VII surface with $b_2(X_{\min})>0$.  In this case, using \cite[Lemma 1.1.3]{Na} we see that even the de Rham class $c_1^{\mathrm{DR}}({\cal K}_{X_{\min}})$ is not represented by an effective divisor.
\end{proof}

\begin{co} \label{RegFlatRed}
Let $A$ be a reducible flat $\SU(2)$-instanton on a class VII surface. Then $\H^2_A=0$.	
\end{co}
\begin{proof}
In general, if $A$  is a projectively ASD connection on a Gauduchon surface, and ${\cal E}$  is the associated polystable bundle ${\cal E}$  (see section \ref{KHsection}), the cohomology spaces of the deformation elliptic complexes associated with ${\cal E}$ and $A$ can be compared explicitly \cite[Section 1.4.4]{Te4}. By \cite[Corollary 1.21]{Te4}, if $b_1(X)$ is odd, the second cohomology spaces of the two complexes coincide. Note that that on Kähler surfaces the vanishing of $H^2({\cal E}nd_0({\cal E}))$ does not imply the vanishing of $\H^2_A$. Indeed, if $g$ is Kähler,   the harmonic space  $\H^2_A$ contains $\H^0_A\omega_g$, hence it cannot vanish when ${\cal E}$ is a split polystable bundle. The result follows from Proposition \ref{RegFlatPolyst}.
\end{proof}

The following proposition describes   the local structure around a regular, reducible virtual instanton on any Riemannian 4-manifold  with $b_1(X)=1$ and $b_+(X)=0$:
\begin{pr} \label{LocModFlatRed} 
Let $X$ be a connected, oriented, compact Riemannian 4-manifold with $b_1(X)=1$ and $b_+(X)=0$. Let $E$ be an $\SU(2)$-bundle with $c_2(E)=1$ on $X$. Let $A$ be a flat, reducible instanton on $X$ with $\H^2_A=0$. The Donaldson compactification 
$\overline{{\cal M}}^\ASD(E)$ is an 8-dimensional topological manifold at   $([A],x)$.  
\end{pr}
\begin{proof}
Let $E^0$ be the trivial $\SU(2)$-bundle on $X$. For a point $u\in X$ denote    by $\mathrm{Gl}_u$ the space of gluing data  at $u$, which is just the space $\mathrm{Isom}^+(\Lambda^+_u,\su(E^0_u))\simeq\SO(3)$ of orientation preserving linear isometries $\Lambda^+_u\to\su(E^0_u)$. Let $C_u^\varepsilon$ be the $\varepsilon$-cone over $\mathrm{Gl}_u$ in the vector space $\Hom(\Lambda^+_u,\su(E^0_u))$. In other words
$$C_u^\varepsilon:=\{t\,\gamma|\ t\in[0,\varepsilon),\ \gamma\in\mathrm{Gl}_u \}\stackrel{\rm homeo}{\simeq}\qmod{[0,\varepsilon)\times \mathrm{Gl}_u}{\{0\}\times \mathrm{Gl}_u}   .
$$
The unions 
$$\mathrm{Gl}_X:=\union_{u\in X} \mathrm{Gl}_u\ ,\ C^\varepsilon_X:=\union_{u\in X} C_u^\varepsilon$$
have   natural structures of a locally trivial fibre bundles over $X$ with standard fibres $\SO(3)$, respectively the $\varepsilon$-cone $C^\varepsilon:=\{t\gamma|\ t\in[0,\varepsilon),\ \gamma\in\SO(3) \}$ over $\SO(3)$ in $M_{3,3}(\R)$.  
For an open set $U\subset X$ we  denote by $\mathrm{Gl}_U$, $C^\varepsilon_U$ the restrictions of these bundles to $U$. The gauge group ${\cal G}_{E^0}:=\Gamma(X,\SU(E^0))$ acts naturally on $\mathrm{Gl}_X$ (hence also on $C^\varepsilon_X$) via the morphisms
$${\cal G}_{E^0}\textmap{\mathrm{ev}_u} \SU(E^0_u)\textmap{\mathrm{Ad}} \SO(\su(E^0_u))\textmap{\rm composition} \Aut\big(\mathrm{Isom}^+(\Lambda^+_u,\su(E^0))\big) ,
$$
where $\mathrm{Ad}$ is just the adjoint representation of $\SU(E^0_u)$ on its Lie algebra $\su(E^0_u)$ and the composition morphism on the right maps $\varphi\in\SO(\su(E^0_u))$ to the automorphism $\gamma\mapsto \varphi\circ\gamma$ of $\mathrm{Gl}_u$. For an open set $U\subset X$ the bundles  $\mathrm{Gl}_U$, $C^\varepsilon_U$ are obviously gauge invariant.\\

Let $A$ be a reducible flat instanton on $X$ with $\H^2_A=0$, and let $x\in X$.   In this case  \cite[Proposition 8.2.4]{DK} gives an open neighbourhood $U$ of $x$ in $X$, a positive number $\varepsilon>0$, an open neighbourhood $W$ of $([A],x)$ in $\overline{{\cal M}}^\ASD(E)$ and a homeomorphism
$$\qmod{B_\varepsilon \times C^\varepsilon_U}{{\cal G}_A}\to W ,$$
where $B_\varepsilon$ is the radius $\varepsilon$ ball of $\H^1_A$, and ${\cal G}_A$ the stabiliser of $A$.
\\ \\
1. Suppose first the ${\cal G}_{A}\simeq S^1$. In this case one has $\H^0_A=\R\sigma$, $\H^1_A= \H^1(X)\sigma$, where $\sigma\in \Gamma(X,\su(E^0))$ is an $A$-parallel unit section. The stabiliser ${\cal G}_A$ acts trivially on $\H^1_A$.   Let $\Lambda^+_U$, $S(\Lambda^+_U)$, $B_\varepsilon(\Lambda^+_U)$ be the restriction of the bundle $\Lambda^+$ to $U$, the corresponding unit sphere bundle, respectively the corresponding radius $\varepsilon$ ball bundle.    The maps $a:\mathrm{Gl}_U\to S(\Lambda^+_U)$, $b:C^\varepsilon_U\to B_\varepsilon(\Lambda^+_U)$ given by
$$a(\gamma)=\gamma^{-1}(\sigma_u)\,,\ b(t\gamma)=t \gamma^{-1}(\sigma_u)\ \forall u\in U\ \forall\gamma\in \mathrm{Gl}_u\ \forall t\in[0,\varepsilon)
$$
are surjective, and their fibres coincide with the ${\cal G}_{A}$-orbits in $\mathrm{Gl}_U$, $C^\varepsilon_U$ respectively. Therefore in this case one has a homeomorphism 
$$\qmod{B_\varepsilon \times C^\varepsilon_U}{{\cal G}_A}\simeq B_\varepsilon\times B_\varepsilon(\Lambda^+_U),
$$
which is obviously an 8-dimensional topological manifold.
\\ \\
2. Suppose now that ${\cal G}_A\simeq\SU(2)$. In this case the bundle $\su(E^0)$ is trivial, and choosing a trivialisation of this bundle, one obtains identifications $\H^0_A=\su(2)$, $\H^1_A= \H^1(X)\otimes\su(2)$. Therefore in this case the quotient ${B_\varepsilon \times C^\varepsilon_U}/{{\cal G}_A}$ is a locally trivial fibre bundle over $U$ with fibre
$$F_\varepsilon:=\qmod{B_\varepsilon(\su(2))\times C^\varepsilon}{\SU(2)}\simeq \qmod{B^3_\varepsilon \times C^\varepsilon}{\SO(3)}
$$
Identifying $C^\varepsilon$ with the topological cone 
$$\qmod{[0,\varepsilon)\times\SO(3)}{\{0\}\times\SO(3)}$$
we obtain a proper surjective map
$$\pi_\varepsilon:\qmod{B^3_\varepsilon\times [0,\varepsilon)\times\SO(3)}{\SO(3)}\to F_\varepsilon .
$$
The quotient on the left can be  identified with $B^3_\varepsilon\times [0,\varepsilon)$, hence $F_\varepsilon$ is just the quotient of $B^3_\varepsilon\times [0,\varepsilon)$ collapsing the fibres of the norm map $B^3_\varepsilon\times \{0\}\to[0,\varepsilon)$. In other words $F_\varepsilon$ is the quotient space of $B^3_\varepsilon\times [0,\varepsilon)$ by the equivalence relation  $\sim_\varepsilon$  generated by  
$$\big\{\big((v,0),(w,0)\big)\in \big (B^3_\varepsilon\times [0,\varepsilon)\big)\times \big (B^3_\varepsilon\times [0,\varepsilon)\big)\vert \ \|v\|=\|w\|\big\}. $$

The claim follows now by Lemma \ref{homeo} below, taking into account that $B^3_\varepsilon\times [0,\varepsilon)$ is a  $\sim$ saturated open neighbourhood of $(0,0)$ in $\R^3\times[0,\infty)$ and that $\sim_\varepsilon$ is just the equivalence relation induced by $\sim$ on this neighbourhood.
\end{proof}
\begin{figure}[h]
\includegraphics[scale=0.37]{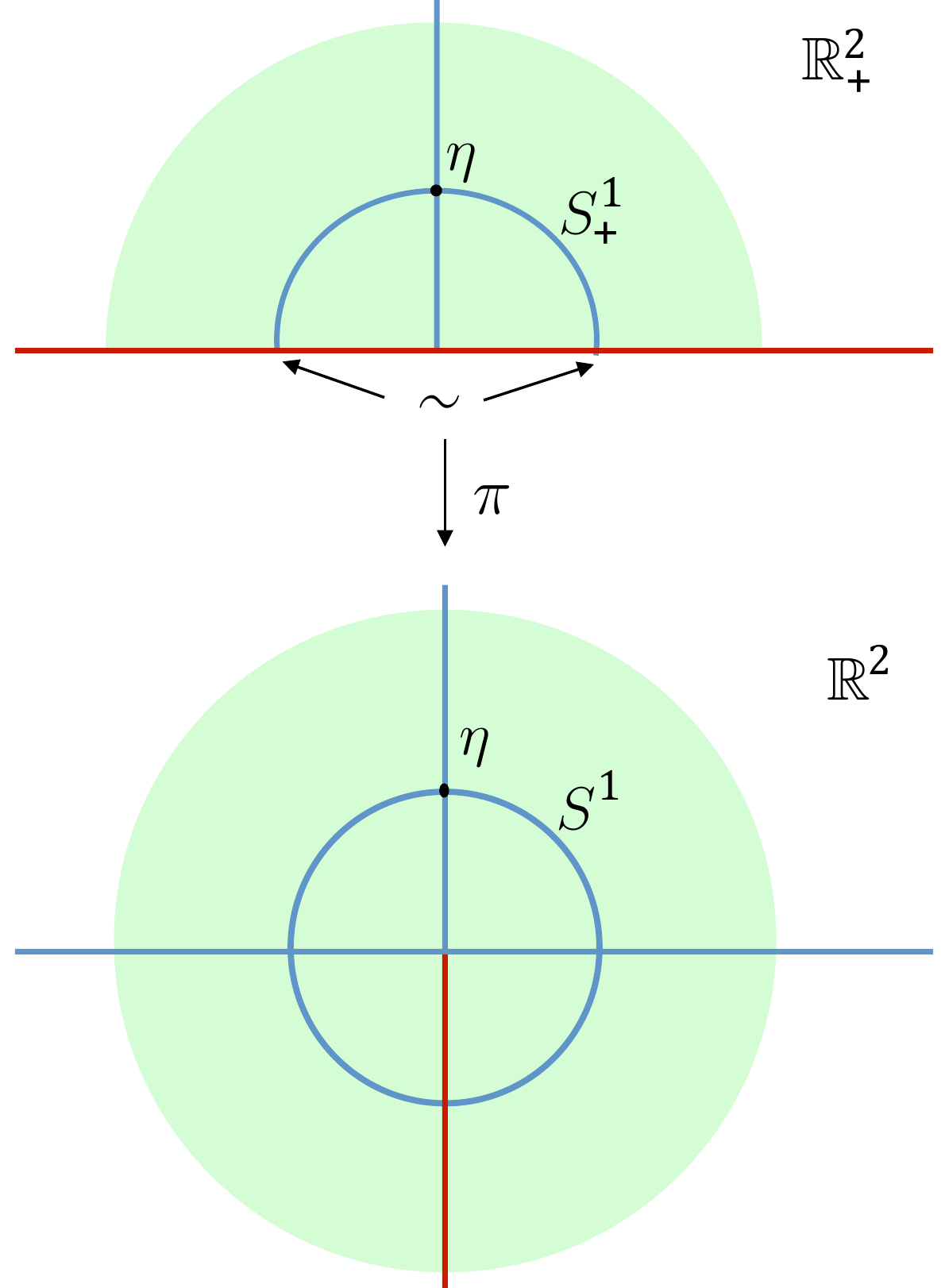}	
\caption{The map $\pi$}
\label{pi}
\end{figure}
\begin{lm}\label{homeo}
Put $\R^{n+1}_+:=\R^n\times [0,\infty)$,  $S^{n}_+:=S^n\cap \R^{n+1}_+$, $\eta:=(0_{\R^n},1)\in S^n_+$, and let $p:S^{n}_+\to S^n$ be the surjective map defined by the conditions 
\begin{enumerate}
\item the points $\eta$, $x$ and $p(x)$ belong to a geodesic of $S^n$,
\item $d(p(x),\eta)=2d(x,\eta)$, where distances are computed using   the standard Riemannian metric on $S^n$.	
\end{enumerate}
  The map $\pi: \R^{n+1}_+\to\R^{n+1}$ defined by  %
$$\pi(\rho w)=\rho p(w)\ \forall w\in S^{n}_+\ \forall \rho\in[0,\infty)$$
 is proper and identifies $\R^{n+1}$ with the topological quotient of $\R^{n+1}_+ $ by the equivalence relation $\sim$ generated by 
$$\big\{\big((v,0),(w,0)\big)\in \R^{n+1}_+\times \R^{n+1}_+\vline \ \|v\|=\|w\|\big\}. $$
\end{lm}

The map $p:S^{n}_+\to S^n$ is just the (maximal extension of the) homothety   of  coefficient 2 and center $\eta$ on the sphere $S^n$. Figure \ref{pi}  illustrates  Lemma \ref{homeo} in the case $n=2$.
\vspace{2mm}

Combining  Corollary \ref{RegFlatRed} with Proposition \ref{LocModFlatRed} we obtain
\begin{pr} \label{R0X}
Let $X$ be a class VII surface, and $E$ be an $\SU(2)$-bundle with $c_2(E)=1$ on $X$. Then $\overline{\cal M}^\ASD(E)$ is a topological 8-manifold around ${\cal R}_0\times X$. 	The union $\Mg:={\cal M}^\ASD(E)^*_{\reg}\cup \big({\cal R}_0\times X\big)$ is open in $\overline{\cal M}^\ASD(E)$ and is a topological 8-manifold.
\end{pr}
\subsection{The structure of $\overline{\cal M}^\ASD(E)$ around the strata of irreducible virtual points} \label{LocModVirtIRRedSect}

Let $X$ be a class VII surface. By \cite[Theorem 2.1]{Pl} it follows that for any stable holomorphic bundle ${\cal E}$ of rank 2  with $\det({\cal E})\simeq {\cal O}_X$ and $c_2({\cal E})=0$ one has $H^2({\cal E}nd_0({\cal E}))=0$. Using the same argument as in the proof of Corollary \ref{RegFlatRed} we obtain
\begin{pr}\label{RegFlatIrred}
Let $A$ be an irreducible  flat $\SU(2)$-instanton on a class VII surface. Then $\H^2_A=0$. 	
\end{pr}
On the other hand the  method used in the proof of Proposition \ref{LocModFlatRed} gives: 
\begin{pr} \label{ModelIrred} Let $X$ be a connected, oriented, compact Riemannian 4-manifold with $b_1(X)=1$ and $b_+(X)=0$. Let $E$ be an $\SU(2)$-bundle with $c_2(E)=1$ on $X$. Let $A$ be a flat, irreducible instanton on $X$ with $\H^2_A=0$. For a sufficiently small $\varepsilon>0$ the subspace $\{[A]\}\times X$ has an open neighbourhood  in  $\overline{{\cal M}}^\ASD(E)$ which can be identified with the cone bundle  $C^\varepsilon_X$.
 \end{pr}
 The following remark shows that an $\SL(2,\C)$-holomorphic structure ${\cal E}$ on $E^0$ defines  a normal complex space structure on the cone bundle 	$C^\varepsilon_X$.
 \begin{re} \label{HolStrCone}
 Let $(X,g)$ be Hermitian surface. The cone bundle $C^\varepsilon_X$ can be identified with   a neighbourhood of the zero section in the cone bundle   of degenerate elements in $\Lambda^{02}_X\otimes S^2(E^0)$. Therefore any holomorphic $\SL(2,\C)$-structure ${\cal E}$ on $E^0$ identifies $C^\varepsilon_X$ with  a neighbourhood of the zero section in the cone bundle of degenerate elements in ${\cal K}_X^\smvee\otimes S^2({\cal E})$, which is a normal complex space over $X$.
 \end{re}
\begin{proof}
The standard fibre of $C^\varepsilon_X$ is a cone over $\SO(3)$, hence it can be identified with a neighbourhood of the singular point in the quotient $\R^4/\mu_2$.  If we fix a $\Spin^c(4)$-structure $\sigma$ on $X$, then the $\SO(3)$-bundle $\mathrm{Gl}_X$ can be identified with the quotient $\U(\Sigma^+,E^0)/S^1$, where $\Sigma^+$ is the positive spinor bundle of $\sigma$.
This identification is obtained using the standard isomorphism $\Lambda^+(X)\textmap{\simeq}\su(\Sigma^+)$. If  $(X,g)$ is a Hermitian surface and  $\sigma$ is its canonical $\Spin^c(4)$-structure, then $\Sigma^+=\Lambda^{00}\oplus\Lambda^{02}$. Denoting by $\S(E)$ the sphere bundle of $E^0$  we obtain a bundle isomorphism $\U(\Sigma^+,E^0)\textmap{u\simeq}\S(E^0)\times_X \S( \Lambda^{02})$ given fibrewise by the formula
$$\U(\Sigma^+_x,E^0_x)\ni a\mapsto  \big(a^{-1}(j\cdot a(1_x),a(1_x)\big)\in \S( \Lambda^{02}_x)\times_X \S(E^0_x)\,,$$
where $1_x$ is the standard generator of $\Lambda^{00}_x$, and $j\cdot$ stands for the multiplication with the quaternionic unit $j$  with respect to the quaternionic structure of $E$ given by the standard identification $\SU(2)\simeq\mathrm{Sp}(1)$. The $S^1$-action on $\S( \Lambda^{02}) \times_X  \S(E^0)$ induced by the standard action of $S^1$ on $\U(\Sigma^+,E^0)$  is given by $\zeta\cdot (u,v)=(\zeta^{-2} u, \zeta v)$. This shows that the bundle $\mathrm{Gl}_X=\U(\Sigma^+,E^0)/S^1$ can be identified with the image  $\Sg$  of $\S(\Lambda^{02})\times_X\S(E^0)$ in $\Lambda^{02}\otimes S^2(E^0)$ via the $\S^1$-invariant map $(e,\nu)\mapsto (e\vee e)\otimes \nu$.  It suffices to note that  $\{t\,\sg|\ t\in[0,\varepsilon),\ \sg\in \Sg\}$ is a neighbourhood of  the zero section in the cone bundle   of degenerate elements in $\Lambda^{02}_X\otimes S^2(E^0)$.   \end{proof}

\section{Singular sheaves with trivial determinant and $c_2=1$}
\label{FamSection}

\subsection{A family of $\SL(2,\C)$-bundles with $c_2=0$}\label{c2=0}

Let $X$ be a class VII surface. For such a surface one has  $\chi({\cal O}_X)=0$, $H^1({\cal O}_X)\simeq\C$,  and $\Pic^0(X)\simeq\C^*$. Fix an isomorphism $\vartheta:H^1({\cal O}_X)\to \C$ which maps $H^1(X,\Z)$ onto $\Z$, and let $\varpi$  be the induced isomorphism $\Pic^0(X)\to \C^*$.

Since $X$ is fixed, we will omit it in the notations $\Pic^\T(X)$, $\Pic^c(X)$, $\mathrm{C}(X)$, $\Cg(X)$, $\mathrm{C}^c(X)$, $\Cg^c(X)$, $\Cg^\cg(X)$, $\rho^c(X)$, $\rho(X)$,  $\Tg(X)$, $\Tg_0(X)$, $\Tg_1(X)$ introduced in section \ref{FundamQ}.  

Let $\mathscr{L}$ be a Poincaré line bundle  on $\Pic^\T\times X$ normalised at a point $x_0\in X$.     For $l\in\Pic^\T$ we put  ${\cal L}_l:= {\mathscr{L}}_{\{l\}\times X}$, regarded as a line bundle on $X$. The universal property of the Poincaré line bundle shows that, for any $l\in\Pic^\T$, the isomorphism type of ${\cal L}_l$ is $l$. For a point $l\in\Pic^\T$ and a tangent vector $v\in T_l(\Pic^\T)=H^1(X,{\cal O}_X)$, the infinitesimal deformation $\epsilon_v(\mathscr{L})\in \mathrm{Ext}^1({\cal L}_l,{\cal L}_l)=H^1(X,{\cal O}_X)$ associated with the pair $(\mathscr{L},v)$  is given by
\begin{equation}\label{DefPoin}
\epsilon_v(\mathscr{L})=v\,.	
\end{equation}

The quotient $\Pg$ of $\Pic^\T$ by the involution $l\stackrel{\jg}\mapsto l^{\smvee}$ decomposes as a disjoint union $\Pg=\union_{\cg\in\Tg} \Pg^\cg$ where, for a class $\cg=[c]\in \Tg$, $\Pg^\cg$ stands for the quotient of $\Pic^c\cup \Pic^{-c}$ by $\jg$. $\Pg^\cg$ is biholomorphic to $\C$ when $\cg\in\Tg_0$, and can be identified with $\Pic^c\simeq\C^*$ when $\cg\in\Tg_1$.  The projection   $\pi:\Pic^\T\to \Pg$  is a branched double covering whose ramification locus is the set 
$$\rho=H^1(X,\mu_2)\subset \union_{c\in\Tors_2(H^2(X,\Z))}\Pic^c,$$
 and whose  branch locus  is $\beta:=\pi(\rho)\subset \union_{\cg\in\Tg_0}\Pg^\cg$.  Since $\Pic^0$ is an injective $\Z$-module, the short exact sequence $0\to\Pic^0\to \Pic^\T\to \Tors(H^2(X,\Z))\to 0$ splits. Fix a left splitting $\sigma:\Pic^\T\to\Pic^0$, and put
 $$\xi:=\frac{1}{2}(\varpi\circ \sigma+\varpi\circ \sigma\circ \jg)\,,\ \zeta:=\frac{1}{2}(\varpi\circ \sigma-\varpi\circ \sigma\circ \jg)\in {\cal O}(\Pic^\T)\,.
 $$
The zero locus of $\zeta$ is the ramification locus $\rho$. With these notations one has  
\begin{equation}\label{pi(O)} 
\pi_*({\cal O}_{\Pic^\T})={\cal O}_{\Pg}\xi\oplus {\cal O}_{\Pg}\zeta\,.	
\end{equation}

The push-forward $\mathscr{E}:=(\pi\times\id_X)_*(\mathscr{L})$ is a rank 2 bundle on $\Pg\times X$.  For $p\in \Pg$ we put ${\cal E}_p:= {\mathscr{E}}_{\{p\}\times X}$ regarded as  a bundle on $X$. Taking into account the definition of $(\pi\times\id_X)_*$, one obtains  canonical isomorphisms
\begin{equation}\label{Ep}
{\cal E}_{\pi(l)}=\mathscr{L}_{X_l+X_{l^\smvee}},	
\end{equation}
where $X_l$ stands for the smooth divisor $\{l\}\times X\subset \Pic^\T\times X$, where $X_l+X_{l^\smvee}$ is regarded as an effective divisor in $\Pic^\T\times X$, and where the right hand sheaf is regarded as an ${\cal O}_X$-module via the obvious projection $X_l +X_{l^\smvee}\to X$.  For a ramification point $l\in \rho$, one obtains 
\begin{equation}\label{Ep+-}
{\cal E}_{\pi(l)}=\mathscr{L}_{2X_l}  \hbox{ (viewed as an ${\cal O}_X$-module)},	
\end{equation}
where $2X_l$ is regarded as a non-reduced complex subspace of $\Pic^\T\times X$. Tensoring by $\mathscr{L}$ the  short exact sequence 
\begin{equation}\label{ExSeq2D}
0\to {\cal O}_{X_l}\textmap{\zeta} {\cal O}_{2X_l}\to {\cal O}_{X_l}\to 0	
\end{equation}
and regarding the central term as an ${\cal O}_X$-module, we obtain a canonical exact sequence
\begin{equation}\label{ExSeqEp+-}
0\to {\cal L}_l\to {\cal E}_{\pi(l)}\to {\cal L}_l\to 0\,,	
\end{equation}
whose extension class is precisely $\epsilon_{v_0}(\mathscr{L})$, where $v_0\in T_l(\Pic^\T)$ is defined by the condition $d \zeta(v_0)=1$ (see Proposition \ref{coverings}  in the Appendix). Taking into account formula (\ref{DefPoin}) we get:

\begin{lm} \label{family} In the conditions and with the notation above, let $p=\pi(l)\in \Pg$. Then
\begin{enumerate}
\item If $l\not\in \rho$ one has   canonical identifications 
$${\cal E}_p={\cal L}_l\oplus {\cal L}_{l^{\smvee}}={\cal L}_l\oplus {\cal L}_{l}^{\smvee},$$
\item If $l \in \rho$ one has a canonical short exact sequence
\begin{equation}\label{extEp} 
0\to {\cal L}_l\textmap{j_l} {\cal E}_p\textmap{r_l} {\cal L}_l\to 0,	
\end{equation}
  whose extension class is the generator of $H^1({\cal O}_X)$ induced by the fixed isomorphism $\vartheta:H^1({\cal O}_X)\to\C$.	
\end{enumerate}
\end{lm}
%
%
%
%
Using the identification (\ref{Ep}) we see that: 
\begin{re} \label{CanShESRem}
Every representative $l\in p$ yields a	canonical short exact sequence
\begin{equation}\label{CanShES}
0\to {\cal L}_l^{\smvee}\textmap{j_l} {\cal E}_p\textmap{r_l} {\cal L}_l\to 0.
\end{equation}
If $l'\ne l$ is a different representative, then $r_l\circ j_{l'}=\id_{{\cal L}_l}$, $r_{l'}\circ j_{l}=\id_{{\cal L}_{l'}}$. 
\end{re}

\begin{dt} \label{CanLines} Let $p=\pi(l)\in \Pg$ and $x\in X$. If $l\not\in \rho$ the two summands ${\cal L}_l$, ${\cal L}_l^{\smvee}$ of ${\cal E}_p$ will be called {\rm  the canonical line sub-bundles} of ${\cal E}_p$ and ${\cal L}_l(x)$, ${\cal L}_l^{\smvee}(x)$  {\rm  the canonical lines} of ${\cal E}_p(x)$. If  $l\in \rho$ the image of ${\cal L}_l$  in ${\cal E}_p$ (via the monomorphism  $j_l$ given by Lemma \ref{family}) will be called the {\rm canonical line sub-bundle} of ${\cal E}_p$, and ${\cal L}_l(x)$  the {\rm canonical line} of ${\cal E}_p(x)$.
\end{dt}
\begin{pr}\label{endom}
Fix a Gauduchon metric $g$ on $X$,  let $l\in\Pic^\T$  with $\deg_g(l)\geq 0$, and put $p:=\pi(l)$. Then
\begin{enumerate}
\item Suppose that $l\not\in\rho$.  Then 
$$\End({\cal E}_p)=\C\id_{{\cal L}_l}\oplus\C\id_{{\cal L}_l^{\smvee}}\oplus H^0({\cal L}_l^{\otimes 2}),\ \Aut({\cal E}_p)=\C^*\id_{{\cal L}_l}\times\C^*\id_{{\cal L}_l^{\smvee}}\times H^0({\cal L}_l^{\otimes 2}).$$
The multiplication on $\C\id_{{\cal L}_l}\oplus\C\id_{{\cal L}_l^{\smvee}}\oplus H^0({\cal L}_l^{\otimes 2})$ induced by the composition of endomorphisms is given by 
$$(z,\zeta,\lambda)(z',\zeta',\lambda')=(zz',\zeta\zeta', z\lambda'+\zeta'\lambda).
$$

\item 	When $l\in\rho$ then 
$$\End({\cal E}_p)=\C\id_{{\cal E}_p}\oplus\C (j_l\circ r_l),\ \Aut({\cal E}_p)=\C^*\id_{{\cal E}_p}\times\C (j_l\circ r_l).$$
The multiplication on $\C\id_{{\cal E}_p}\oplus\C (j_l\circ r_l)$ induced by the composition of endomorphisms is given by 
$$\big(z,\zeta(j_l\circ r_l)\big)\big(z',\zeta'(j_l\circ r_l)\big))=\big(zz', (z\zeta'+z'\zeta)(j_l\circ r_l)\big)\,.
$$

\end{enumerate}
	\end{pr}
\begin{proof}
 (1) follows from Lemma \ref{family} taking into account that, since $\deg_g({\cal L}_l^{\smvee\otimes 2})\leq 0$ and  ${\cal L}_l^{\smvee\otimes 2}\not\simeq{\cal O}_X$, one has $H^0({\cal L}_l^{\smvee\otimes 2})=0$.  Note that in our non-Kählerian framework $X$ might contain non-empty  effective divisors with torsion fundamental class, so  $H^0({\cal L}_l^{\otimes 2})$ might be non-zero even if $l\not\in\rho$. \vspace{2mm}
\\
For (2), let $\varphi\in \End({\cal E}_p)$. The composition $r_l\circ\varphi\circ j_l$ must vanish because, if not, it would be an isomorphism, hence $\varphi\circ j_l$ would define a splitting of the extension (\ref{extEp}). Therefore $\varphi$  maps $j_l({\cal L}_l)$ into itself, giving an endomorphism $\varphi_0$ of ${\cal L}_l$, which can be written as $\varphi_0=z\,\id_{{\cal L}_l}$. The endomorphism $\psi:=\varphi-z\,\id_{{\cal E}_p}$ will vanish on   $j_l({\cal L}_l)$, so it can be written as $\psi=\chi\circ r_l $, for a morphism $\chi:{\cal L}_l\to {\cal E}_p$. Since, by the same argument as above  one has $r_l\circ \chi=0$, it follows that $\chi=j_l\circ u$ for a morphism $u:{\cal L}_l\to {\cal L}_l$. Writing $u=\zeta\id_{{\cal L}_l}$ we obtain $\varphi=z\,\id_\varepsilon+\zeta (j_l\circ r_l)$, which proves the claim.      		
\end{proof}
\begin{dt}\label{vg} Let ${\cal D}iv(X)^{>0}$ be the set of {\it non-empty, effective}  divisors on $X$. Put
$${\cal D}iv(X)^{>0}_0:=\big\{D\in {\cal D}iv(X)^{>0}|\ c_1({\cal O}(D))\in\Tors(H^2(X,\Z))\big\}\,;$$ 
\begin{equation}\label{MinVol}
\begin{array}{c}
\vg(g)=\left\{\begin{array}{cc}
\inf\big\{\vol_g(D)|\ D\in{\cal D}iv(X)^{>0}\big\} &\hbox{ if ${\cal D}iv(X)^{>0}\ne\emptyset$}\,,\\
\infty  & \hbox{ if ${\cal D}iv(X)^{>0}=\emptyset$\,;}
\end{array}\right.
\\ \\
\vg_0(g)=\left\{\begin{array}{cc}
\inf\big\{\vol_g(D)|\ D\in{\cal D}iv(X)^{>0}_0\big\} &\hbox{ if ${\cal D}iv(X)^{>0}_0\ne\emptyset$}\,,\\
\infty  & \hbox{ if ${\cal D}iv(X)^{>0}_0=\emptyset$}\,.
\end{array}\right.

\end{array}		
\end{equation}
\end{dt}
One has obviously $\vg_0(g)\geq \vg(g)$, and using Bishop's compactness theorem (see for instance \cite{Ch}), it follows that $\vg(g)>0$. 
\begin{re}\label{VanSmallDeg} Let ${\cal L}$ be a non-trivial holomorphic line bundle on $X$. Suppose that  $\deg_g({\cal L})<\vg(g)$ or $c_1^\R({\cal L})=0$ and   $\deg_g({\cal L})<\vg_0(g)$. Then $H^0({\cal L})=0$.\end{re}

With these notations Proposition \ref{endom} gives
\begin{co}\label{CoroEndo}
Let  $p=\pi(l)\in\Pg\setminus\beta$ with $0\leq |\deg_g(l)|<\frac{1}{2}\vg_0(g)$. Then 
$$\End({\cal E}_p)=\C\id_{{\cal L}_l}\oplus\C\id_{{\cal L}_l^{\smvee}},\ \Aut({\cal E}_p)=\C^*\id_{{\cal L}_l}\times\C^*\id_{{\cal L}_l^{\smvee}}.$$
\end{co}
Let $A_\varepsilon$ be the tubular neighbourhood of $\mathrm{C}=\Hom(H_1(X,\Z),S^1)$ in $\Pic^T$ given by
$$A_\varepsilon:=\{l\in\Pic^\T|\ \deg_g(l)\in(-\varepsilon,\varepsilon) \},
$$
and let $\Ag_\varepsilon$ be the quotient $\Ag_\varepsilon:=A_\varepsilon/\langle\jg\rangle$.
$A_\varepsilon$ is a disjoint union of annuli, and $\Ag_\varepsilon$ is  a disjoint union of annuli and disks.
Denoting by $|\deg_g|:\Pg\to [0,\infty)$ the map given by $|\deg_g|(\pi_0(l)):=|\deg_g(l)|$, we see that  $\Cg$ is just the zero set of  $|\deg_g|$, and $\Ag_\varepsilon$ is the open subset of $\Pg$ defined by the inequality $|\deg_g|(p)<\varepsilon$.

\begin{co}\label{versal} Let $(X,g)$ be a class VII surface endowed with a Gauduchon metric. For sufficiently small $\varepsilon>0$ the following holds:
\begin{enumerate}
\item 	$h^0({\cal E}nd_0({\cal E}_p))=1$, $h^2({\cal E}nd({\cal E}_p))=h^2({\cal E}nd_0({\cal E}_p))=0$ for any  $p\in\Ag_\varepsilon$,
\item   the family $\mathscr{E}$ is versal at any point $p\in\Ag_\varepsilon$.
\end{enumerate}

\end{co}
\begin{proof}
(1) By Corollary \ref{CoroEndo} it follows that $h^0({\cal E}nd_0({\cal E}_p))=1$ if $|\deg|_g(p)$ is sufficiently small. Write $p=\pi(l)$ with $l\in A_\varepsilon$. By Lemma \ref{H2FlatLB} we know that $h^0({\cal K}_X\otimes {\cal L}^{\otimes 2})=h^0({\cal K}_X\otimes {\cal L}^{\smvee \otimes 2})=0$ if $l\in \mathrm{C}$. Using Remark \ref{CanShES} and the known property $h^0({\cal K}_X)=0$, we see that $h^0({\cal K}_X\otimes{\cal E}nd({\cal E}_p))=h^0({\cal K}_X\otimes{\cal E}nd_0({\cal E}_p))=0$  when $p\in \Cg$. By Grauert's semicontinuity theorem it follows  that, for sufficiently small $\varepsilon>0$, one has $h^2({\cal E}nd({\cal E}_p))=h^2({\cal E}nd_0({\cal E}_p))=0$  for any $p\in \Ag_\varepsilon$. 

(2) Using the Riemann-Roch theorem we get $h^1({\cal E}nd_0({\cal E}_p))=1$ for any $p\in\Ag_\varepsilon$. It suffices to note that the infinitesimal deformation $\epsilon_v(\mathscr{E})$ is non-trivial for any non-trivial tangent vector $v$.  For $p\not\in\beta$ the claim follows from the similar property of the Poincaré line bundle $\mathscr{L}$. For $p\in \beta$ this follows from Proposition \ref{coverings} (\ref{DefoCOmp}).
\end{proof}

\subsection{Singular torsion-free sheaves with trivial determinant and $c_2=1$}
\label{SingSh}
We start with the following simple classification result:
\begin{pr}\label{singF}
Let ${\cal F}$ be a singular torsion-free rank 2 sheaf on a non-algebraic complex surface $X$ with  $\det({\cal F})\simeq {\cal O}_X$ and $c_2({\cal F})=1$. Then ${\cal F}$ fits in an exact sequence
\begin{equation}\label{SingF} 
0\to {\cal F}\to {\cal E}\to  {\cal O}_{\{x\}}\to 0,	
\end{equation}
where ${\cal E}={\cal F}^{\we}$ is a rank 2-bundle on  $X$ with $\det({\cal E})\simeq{\cal O}_X$,  $c_2({\cal E})=0$, and ${\cal O}_{\{x\}}:={\cal O}/{\cal I}_x$.
\end{pr}
\begin{proof} Since ${\cal F}$ is torsion-free on a surface, its singularity set $S\subset X$ is 0-dimensional,  ${\cal F}$  embeds injectively in its bidual ${\cal E}:={\cal F}^{\we}$ (which is locally free), and the support of the quotient sheaf  ${\cal Q}:={\cal E} /{{\cal F}}$ is $S$. Using the exact sequence
$$0\to {\cal F}\to {\cal E}\to {\cal Q}\to 0
$$
 we get $\det({\cal E})\simeq\det({\cal F})\simeq {\cal O}_X$, and $c_2({\cal E})=c_2({\cal F})-h^0({\cal Q})$. On the other hand, since $X$ is non-algebraic, we get \cite[Théorème 0.3]{BLP} $4c_2({\cal E})-c_1({\cal E})^2\geq 0$, hence, in our case, $c_2({\cal E})\geq 0$. This gives $h^0({\cal Q})\leq c_2({\cal F})=1$. Since we suppose that ${\cal F}$ is singular, ${\cal Q}$ cannot vanish, hence $h^0({\cal Q})=1$, which shows that ${\cal Q}$ is just the structure sheaf of a simple point.
\end{proof}
Let ${\cal E}$ be rank 2 locally free sheaf  on a surface $X$, $x\in X$ and $\eta:{\cal E}(x)\to\C$ be a linear epimorphism. We will  denote by ${\cal F}({\cal E},x,\eta)$ the torsion-free sheaf $\ker(\tilde \eta)$, where $\tilde \eta:{\cal E}\to {\cal O}_{\{x\}}$ is the epimorphism induced by $\eta$.
\begin{pr}\label{morphisms}
If ${\cal F}({\cal E},x,\eta)\simeq {\cal F}({\cal E}',x',\eta')$ then $x=x'$ and any isomorphism $f:{\cal F}({\cal E},x,\eta)\to {\cal F}({\cal E}',x,\eta')$ is induced by a bundle isomorphism $\varphi:{\cal E}\to {\cal E}'$ such that $\eta'\circ \varphi(x)\in\C^* \eta$.
\end{pr}
\begin{proof} Indeed, if ${\cal F}({\cal E},x,\eta)\simeq {\cal F}({\cal E}',x',\eta')$ then comparing the singularity sets of the two sheaves we get $x=x'$. An isomorphism $f:{\cal F}({\cal E},x,\eta)\to  {\cal F}({\cal E}',x,\eta')$ induces an isomorphism 
$$\varphi: {\cal F}({\cal E},x,\eta)^{\we}={\cal E}\to {\cal E}'={\cal F}({\cal E}',x,\eta')^{\we}
$$
between the respective biduals which makes the diagram
$$
\begin{diagram}[h=7.5mm]
{\cal F}({\cal E},x,\eta)& \rTo^{f}_{\simeq} &{\cal F}({\cal E}',x,\eta')\\
\dInto  & & \dInto \\
{\cal E}&\rTo^{\varphi}_{\simeq} & {\cal E}'
\end{diagram} 
$$
commutative. This shows that $\varphi$ maps ${\cal F}({\cal E},x,\eta)$ onto ${\cal F}({\cal E}',x,\eta')$ and $f$ is just the restriction of $\varphi$ to ${\cal F}({\cal E},x,\eta)$. The former property is equivalent to $\varphi(\ker(\tilde \eta))=\ker(\tilde \eta')$, which, since  ${\cal O}_{\{x\}}$ is a skyscraper sheaf with 1-dimensional stalk, is obviously equivalent to $\tilde \eta'\circ \varphi\in\C^* \tilde \eta$. But this is equivalent to  $  \eta'\circ \varphi(x)\in\C^*   \eta$.   
	\end{proof}
\begin{co} \label{IsoCrit} Let ${\cal E}$ be rank 2 locally free sheaf  on a surface $X$,  $x$, $x'\in X$, and $\eta$, $\eta':{\cal E}(x)\to\C$ be epimorphisms. The following conditions are equivalent:
\begin{enumerate}
 \item ${\cal F}({\cal E},x,\eta)\simeq  {\cal F}({\cal E},x',\eta')$,
 \item $x=x'$ and there exists an automorphism $\varphi\in\Aut({\cal E})$ such that $\eta'\circ \varphi(x)=\eta$.	
\end{enumerate}
\end{co}

We return to   class VII surfaces  and the objects constructed in the previous sections.

\begin{dt} Let $X$ be a class VII surface, let $p\in\Pg$ and $x\in X$. An epimorphism $\eta:{\cal E}_p(x)\to\C$ will be called {\rm admissible} if its kernel does not coincide with a canonical line of ${\cal E}_p(x)$ (see Definition \ref{CanLines}).
\end{dt}

Using Proposition  \ref{endom} and Corollary \ref{CoroEndo} we get
\begin{co}\label{F-Bundles}
Let  $p=\pi(l)\in\Pg$ with $0\leq |\deg_g(l)|<\frac{1}{2}\vg_0(g)$.
\begin{enumerate}
\item For any admissible epimorphism	 $\eta:{\cal E}_p(x)\to\C$, the sheaf ${\cal F}({\cal E}_p,x,\eta)$ is simple,
\item If $\eta$, $\eta':{\cal E}_p(x)\to\C$ are admissible epimorphisms, then ${\cal F}({\cal E}_p,x,\eta)\simeq {\cal F}({\cal E}_p,x,\eta')$.
\end{enumerate}
\end{co}
	
The sheaves of the form ${\cal F}({\cal E}_p,x,\eta)$ will play a crucial role in the proof of our main result. Using Remark \ref{CanShESRem} we see that
\begin{re} \label{CanShESFRem} Let $\eta:{\cal E}_p(x)\to\C$ be an admissible epimorphism. A representative $l\in p$ gives a short exact sequence
\begin{equation}\label{CanShESF} 0\to {\cal L}_l^{\smvee}\otimes{\cal I}_x\textmap{j_{l,x,\eta}} {\cal F}({\cal E}_p,x,\eta)\textmap{r_{l,x,\eta}} {\cal L}_l\to 0 ,
\end{equation}   
where $j_{l,x,\eta}$, $r_{l,x,\eta}$ are induced by $j_l$, $r_l$. \end{re}

\section{The map $V_\varepsilon$ and its properties}
\label{VSection}

We will need holomorphic families of singular sheaves of the form ${\cal F}({\cal E}_p,x,\eta)$ parameterised by   products $\Pg\times U$, for sufficiently small open subsets $U\subset X$. Therefore we will need families of admissible epimorphisms 
$$(\eta_{p,u}:{\cal E}_p(u)\to\C)_{(p,u)\in\Pg\times U}$$
defined for sufficiently small open subsets $U\subset X$. Let  $U\subset X$ be a contractible, Stein open set. The restriction $\mathscr{L}_U:=\resto{\mathscr{L}}{\Pic^\T\times U}$ is trivial. Choose a trivialisation morphism $\sigma:\mathscr{L}_U\textmap{\simeq} {\cal O}_{\Pic^\T\times U}$. Taking push-forward via the double cover  
$$\pi\times\id_U:\Pic^\T\times U\to \Pg\times U\,,$$
and using (\ref{pi(O)}), we get a bundle isomorphism 
\begin{equation}\label{Dec}
(\pi\times\id_U)(\sigma):\mathscr{E}_U\textmap{\simeq} (\pi\times\id_U)_*({\cal O}_{\Pic^\T\times U})={\cal O}_{\Pg\times U}\,\xi\oplus {\cal O}_{\Pg\times U}\,\zeta,	
\end{equation}
where $\mathscr{E}_U:=\resto{\mathscr{E}}{{\Pg\times U}}$. The projection on the second summand gives an epimorphism $\eta^\sigma:\mathscr{E}_U\to {\cal O}_{\Pg\times U}$.  
\begin{lm}\label{eta-admiss} For   any  $(p,u)\in\Pg\times U$, $\eta^{\sigma}_{p,u}:{\cal E}_p(u)\to\C$ is an admissible epimorphism. \end{lm}
\begin{proof}
 	  Via the identifications 
$$\mathscr{E}_U(\Pg\times U)=\mathscr{L}_U(\Pic^\T\times U)\textmap{\simeq \sigma}{\cal O}(\Pic^\T\times U),
$$
 	    the morphism $\eta^\sigma$ is given by $\eta^\sigma(\alpha)=\frac{1}{2\zeta}\big(\alpha -\alpha\circ(\jg\times\id_U)\big)$. If $l\ne l^\smvee$ this shows that $\ker(\eta^{\sigma}_{p,u})$ is identified with the diagonal line of the sum ${\cal L}_l(u)\oplus {\cal L}_{l^\smvee}(u)$.  If  $l= l^\smvee$ the same formula shows that $\eta^\sigma$ defines a splitting of the restriction of the exact sequence (\ref{ExSeqEp+-}) to $U$, so $\ker(\eta^{\sigma}_{p,u})$ is a complement of ${\cal L}_l(u)$ in ${\cal E}_p(u)$.	
\end{proof}

\subsection{The embedding  $V_\varepsilon$}\label{embeddV}

Let ${\cal M}^\si$ be the moduli space of simple torsion-free sheaves on $X$ with trivial determinant line bundle and $c_2=1$, and let
$${\cal M}^\si_\reg:=\{[{\cal F}]\in {\cal M}^\si|\ \mathrm{Ext}^2_0({\cal F},{\cal F})=0\}=\{[{\cal F}]\in {\cal M}^\si|\ \mathrm{Ext}^2({\cal F},{\cal F})=0\}
$$
be its regular part.  
\begin{pr}\label{embedding} Let $(X,g)$ be a class VII surface endowed with a Gauduchon metric. For sufficiently small $\varepsilon>0$ the following hold:
\begin{enumerate}
\item 	 $[{\cal F}({\cal E}_p,x,\eta)]\in {\cal M}^\si_\reg$ for any $p\in\Ag_\varepsilon$,  $x\in X$ and admissible epimorphism $\eta:{\cal E}_p(x)\to\C$.
\item \label{HolEmbed} The map $V_\varepsilon:\Ag_\varepsilon\times X\to {\cal M}^\si_\reg$ given by 
$$V(p,x):=[{\cal F}({\cal E}_p,x,\eta)]
$$
(for an admissible epimorphism $\eta:{\cal E}_p(x)\to\C$) is  well-defined, injective, holomorphic,  immersive,  and homeomorphic on its image.
\end{enumerate}
\end{pr}
In other words, for sufficiently small $\varepsilon>0$, $V_\varepsilon$ is a holomorphic embedding.
\begin{proof}
(1) Put ${\cal F}:={\cal F}({\cal E}_p,x,\eta)$, ${\cal E}:={\cal E}_p$ to save on notation. The fact that ${\cal F}$ is simple for sufficiently small $|\deg_g|(p)$  follows from Corollary \ref{F-Bundles}.

To prove that $\mathrm{Ext}^2({\cal F},{\cal F})=0$, use the local-global spectral sequence. Since the singularity set of ${\cal F}$ is 0-dimensional, one has 
$$H^1({\cal E}xt^1({\cal F},{\cal F}))=0,$$
hence  it suffices to show that, when $|\deg|_g(p)$ is sufficiently small, one has
\begin{enumerate}[(a)]
\item 	$H^2({\cal H}om({\cal F} ,{\cal F}))=0$,
\item  ${\cal E}xt^2({\cal F} ,{\cal F})=0$.
\end{enumerate}
To prove (a) note that the sheaf ${\cal H}om({\cal F} ,{\cal F})$ fits in a short exact sequence
\begin{equation}\label{ExSeqHom} 
0\to {\cal H}om({\cal F} ,{\cal F} )\to {\cal H}om( {\cal F}, {\cal E})={\cal H}om( {\cal E}, {\cal E})\textmap{\psi_\eta} {\cal H}om\big(\qmod{{\cal F}}{{\cal I}_{x}{\cal E}}\,, {\cal O}_{\{x\}}\big)\to 0\,,	
\end{equation}
where $\psi_\eta$ is the composition 
$${\cal H}om( {\cal E}, {\cal E})\to  {\cal H}om\big( \qmod{{\cal E}}{{\cal I}_{x}{\cal E}}, \qmod{{\cal E}}{{\cal I}_{x}{\cal E}}\big)\to {\cal H}om\big(\qmod{{\cal F}}{{\cal I}_{x}{\cal E}}, \qmod{{\cal E}}{{\cal I}_{x}{\cal E}}\big)\textmap{\eta\circ }{\cal H}om\big(\qmod{{\cal F}}{{\cal I}_{x}{\cal E}}, {\cal O}_{\{x\}}\big) .$$

Since ${\cal H}om\big(\qmod{{\cal F}}{{\cal I}_{x}{\cal E}},{\cal O}_{\{x\}}\big)$ is a torsion sheaf supported at $\{x\}$, we have
$$H^2({\cal H}om({\cal F},{\cal F}))\simeq  H^2({\cal H}om( {\cal E}, {\cal E}))\,,
$$	
which vanishes if $\varepsilon$ is sufficiently small by Corollary \ref{versal}.
\vspace{2mm} \\
For (b) note that the stalk ${\cal F}_y$ is a free ${\cal O}_{y}$-module for any point $y\ne x$, whereas
$${\cal F}_x\simeq {\cal O}_x\oplus \big\{{\cal I}_x\big\}_x
$$
as ${\cal O}_x$-modules. Therefore we have to prove that
$$\mathrm{Ext}^2_{{\cal O}_x}\big(\big\{{\cal I}_x\big\}_x,{\cal O}_x\big)=\mathrm{Ext}^2_{{\cal O}_x}\big(\big\{{\cal I}_x\big\}_x,\big\{{\cal I}_x\big\}_x\big)=0\,.
$$
Using the exact sequence
\begin{equation}\label{ExSeq} 0\to {\cal O}_x\to {\cal O}_x^{\oplus 2}\to \big\{{\cal I}_x\big\}_x\to 0
\end{equation}
we see that the homological dimension  of $\big\{{\cal I}_x\big\}_x$ is 1, hence    
$$\mathrm{Ext}^i_{{\cal O}_x}\big(\big\{{\cal I}_x\big\}_x,{\cal O}_x\big)=0$$
 for $i\geq 2$ \cite[section V.3, formula 3.20]{Ko}. Using   (\ref{ExSeq}) again, we obtain an exact sequence
$$\dots\to \mathrm{Ext}^2_{{\cal O}_x}\big(\big\{{\cal I}_x\big\}_x,{\cal O}_x^{\oplus 2}\big)\to \mathrm{Ext}^2_{{\cal O}_x}\big(\big\{{\cal I}_x\big\}_x,\big\{{\cal I}_x\big\}_x\big)\to \mathrm{Ext}^3_{{\cal O}_x}\big(\big\{{\cal I}_x\big\}_x,{\cal O}_x\big)\to\dots
$$
which shows that $\mathrm{Ext}^2_{{\cal O}_x}\big(\big\{{\cal I}_x\big\}_x,\big\{{\cal I}_x\big\}_x\big)=0$, too.
\\ \\
(2) Corollary \ref{F-Bundles} shows that $V_\varepsilon$ is well-defined. The injectivity of $V_\varepsilon$ follows from Proposition \ref{morphisms}, taking into account that ${\cal E}_p$, ${\cal E}_{p'}$ are non-isomorphic when $p\ne p'$ and $|\deg_g|(p)$, $|\deg_g|(p')$ are sufficiently small.\\

We prove  that   $V_\varepsilon$ is holomorphic.
Let  $U\subset X$  be a contractible, Stein open subset of $X$, $\sigma$ be a trivialization of $\mathscr{L}_U:=\resto{\mathscr{L}}{\Pic^\T\times U}$, and let $\eta^\sigma: \mathscr{E}_U\to {\cal O}_{\Pg\times U}$ be the associated  epimorphism.  On the product $\Pg\times U \times X$ we consider the   following sheaves:

 \begin{itemize}
 \item 	$\EXUX:=p_{\Pg X}^*(\mathscr{E})$,
 \item 	$\EUUX:=p_{\Pg U}^*(\mathscr{E}_U)$.
  \end{itemize}
 
 Denoting by $\Delta_{\sUX}\simeq U$ the graph of the inclusion map $U\hookrightarrow X$, we get an obvious identification 
 $$\resto{\EXUX}{\Pg\times \Delta_{\ssUX}}=\resto{\EUUX}{\Pg\times \Delta_{\ssUX}}\,,$$
 so an obvious epimorphism
 $$w^\sigma:\resto{\EXUX}{\Pg\times \Delta_{\ssUX}}\to {\cal O}_{\Pg\times\Delta_{\ssUX}}
 $$
 which corresponds to $\eta^\sigma$ via the canonical identifications.   Put
\begin{equation}\label{Fsigma}
\mathscr{F}^\sigma:=\ker\big(\EXUX\to \EXUX_{\Pg\times\Delta_{\ssUX}}\textmap{w^\sigma}  {\cal O}_{\Pg\times\Delta_{\ssUX}}\big) .	
\end{equation}
Since  $\EXUX$ and ${\cal O}_{\Pg\times\Delta_{\ssUX}}$ are flat over $\Pg\times U$, it follows by Lemma \ref{claimLemma} proved below that    $\mathscr{F}^\sigma$ is also flat over $\Pg\times U$, and that  the restriction ${\cal F}_{p,u}^\sigma$ of $\mathscr{F}^\sigma$ to a fibre $\{(p,u)\}\times X$ (regarded as a sheaf on $X$) is just the kernel of the composition 
 $${\cal E}_p\to  {\cal E}_p(u)\textmap{\eta^\sigma_{p,u}} \C.
 $$ 
 Therefore one  has
 $$[{\cal F}_{p,u}^\sigma]= V_\varepsilon(p,u)\ \forall (p,u)\in\Ag_\varepsilon\times U\, .
 $$
 Since $\mathscr{F}^\sigma$ is flat over $\Pg\times U$, this proves that $V_\varepsilon$ is holomorphic on $\Ag_\varepsilon\times U$. Using a covering of $X$ by Stein, contractible open subsets, we see that $V_\varepsilon$ is holomorphic on $\Ag_\varepsilon\times X$.\\
 
 We will show now that $V_\varepsilon$ is an immersion. The tangent space $T_{[{\cal F}]}{\cal M}^\si$ can be identified with the kernel $\Ext^1_0({\cal F},{\cal F})$ of the trace map $\Ext^1({\cal F},{\cal F})\to H^1(X,{\cal O}_X)$ \cite[section 3]{To}. The local-global spectral sequence gives the exact sequence
 $$0\to H^1({\cal H}om_0({\cal F},{\cal F}))\textmap{J} \Ext^1_0({\cal F},{\cal F})\textmap{R} H^0({\cal E}xt^1({\cal F},{\cal F}))\to 0\,,
 $$
 where, on the right, we took into account that $H^2({\cal H}om({\cal F},{\cal F}))=0$. The image of $J$ is the subspace of $\Ext^1_0({\cal F},{\cal F})$ consisting of trace-free extension classes of ${\cal F}$ by ${\cal F}$, which are locally split. The morphism $R$ has a simple geometric interpretation:  If
 $$0\to {\cal F}\textmap{a} {\cal F}'\textmap{b} {\cal F}\to 0
 $$
 represents an extension class $\varepsilon\in \Ext^1_0({\cal F},{\cal F})$ then, for any $x\in X$, $R(\varepsilon)(x)\in \Ext^1_{{\cal O}_x}({\cal F}_x,{\cal F}_x)$ is the extension class of the ${\cal O}_x$-module exact sequence
 $$0\to {\cal F}_x\textmap{a_x} {\cal F}'_x\textmap{b_x} {\cal F}_x\to 0\,.
 $$
 
 The injectivity of the tangent map $D_{p,x}V_\varepsilon: T_{(p,x)}(\Pg\times X)\to T_{[{\cal F}]}{\cal M}^\si=\Ext^1_0({\cal F},{\cal F})$ follows from Lemma \ref{PartDer} below.
 \\
 
 Finally we prove that, for sufficiently small $\varepsilon>0$ the map $V_\varepsilon$ is homeomorphic on its image. Put
 $$\bar A_\varepsilon:=\big\{l\in\Pic^\T|\ \deg_g(l)\in[-\varepsilon,\varepsilon] \big\},\  \bar\Ag_\varepsilon:=\qmod{\bar A_\varepsilon}{l\mapsto l^{\smvee}},
$$
and let $\bar V_\varepsilon:\bar\Ag_\varepsilon\times X\to {\cal M}^\si$ be the map defined again by  
$$\bar V_\varepsilon(p,x):=[{\cal F}({\cal E}_p,x,\eta)],
$$
for an admissible epimorphism $\eta:{\cal E}_p(x)\to\C$.   It is easy to see that, for sufficiently small $\varepsilon>0$, the map $\bar V_\varepsilon$ is continuous, and injective. Moreover, its image $\bar V_\varepsilon(\bar\Ag_\varepsilon\times X)$ is Hausdorff. This follows using the non-separability criterion explained in section \ref{ComSub}, and noting that, for two distinct pairs $(p,x)\ne (p',x')$ with $p$, $p'\in \bar\Ag_\varepsilon$, one has  $\Hom\big({\cal F}({\cal E}_p,x,\eta),{\cal F}({\cal E}_{p'},x',\eta')\big)=0$. Therefore $\bar V_\varepsilon$ induces a continuous bijection $\bar\Ag_\varepsilon\times X\to \bar V_\varepsilon\big(\bar\Ag_\varepsilon\times X\big)$ from a compact space to a Hausdorff space. By a well-known result in topology, this bijection is a homeomorphism. It suffices to note that $V_\varepsilon$ is the restriction of $\bar V_\varepsilon$ to $\Ag_\varepsilon\times X$.

\end{proof}
\begin{lm}\label{PartDer} \begin{enumerate}
\item The partial derivative $\frac{\partial V_\varepsilon}{\partial p}:T_p\Pg\to\Ext^1_0({\cal F},{\cal F})$ factorises as $\frac{\partial V_\varepsilon}{\partial p}=J\circ h$, where $h:T_p\Pg\to H^1({\cal H}om_0({\cal F},{\cal F}))$ is an isomorphism.  
\item The composition $R\circ \frac{\partial V_\varepsilon}{\partial x}:T_xX\to H^0({\cal E}xt^1({\cal F},{\cal F}))$ is injective.
\end{enumerate}	
\end{lm}

\begin{proof}   The definition of $\mathscr{F}^\sigma$ (\ref{Fsigma}) gives the exact sequence 
\begin{equation}\label{ExSeqF} 0\to \mathscr{F}^\sigma\hookrightarrow \EXUX \textmap{w_\sigma} {\cal O}_{\Pg\times\Delta_{\ssUX}}\to 0\,,
\end{equation}
where $w_\sigma$ stands for the composition 
$\EXUX\to \EXUX_{\Pg\times\Delta_{\ssUX}}\textmap{w^\sigma}  {\cal O}_{\Pg\times\Delta_{\ssUX}}$. The decomposition (\ref{Dec}) of  the restriction $\mathscr{E}_U=\resto{\mathscr{E}}{\Pg\times U}$ and the definition of the sheaf $\mathscr{F}^\sigma$  gives an obvious isomorphism
\begin{equation}\label{DecFU}
\resto{\mathscr{F}^\sigma}{\Pg\times U\times U}={\cal O}_{\Pg\times U\times U}\oplus {\cal I}_{\Pg\times\Delta_{\ssUU}}\,, 
\end{equation}	
where $\Delta_{UU}$ is the graph of $\id_U$.\\
\\
(1)  Let $(p,x)\in\Pg\times U$, and put ${\cal E}:={\cal E}_p$,  ${\cal F}:={\cal F}_{(p,x)}^\sigma={\cal F}({\cal E}_p,x,\eta^{\sigma}_{p,x})$.   The direct sum decomposition (\ref{DecFU}) shows that 
$$\resto{\mathscr{F}^\sigma}{\Pg\times \{x\}\times U}=p_U^*({\cal O}_U\oplus {\cal I}_x)\,,$$ 
hence this restriction can be regarded as a {\it constant} family of sheaves on $U$ parameterized by $\Pg\times \{x\}$. Taking into account the geometric interpretation of the morphism $R$, this shows that, for every $p\in\Pg$ and for every $\xi\in T_p(\Pg)$ one has
$$R\big(\frac{\partial V_\varepsilon}{\partial p}(\xi,0)\big)=R\epsilon_{(\xi,0)}\big(\mathscr{F}^\sigma\big)=0\,,
$$
so $\frac{\partial V_\varepsilon}{\partial p}$ factorizes as $\frac{\partial V_\varepsilon}{\partial p}=J\circ h$, for a morphism $h:T_p\Pg\to H^1({\cal H}om_0({\cal F},{\cal F}))$. We will prove  that $h$ is an isomorphism. Note first that the obvious morphisms
$$\theta:\Ext^1({\cal F},{\cal F})\to \Ext^1({\cal F},{\cal E})\,,\ \tau: \Ext^1({\cal E},{\cal E})=H^1({\cal H}om({\cal E},{\cal E}))\to \Ext^1({\cal F},{\cal E})$$
associated with the embedding ${\cal F}\subset {\cal E}$ induce morphisms 
$$\theta_0:H^1({\cal H}om({\cal F},{\cal F}))\to  H^1({\cal H}om({\cal F},{\cal E}))\,,\ \tau_0:H^1({\cal H}om({\cal E},{\cal E}))\to  H^1({\cal H}om({\cal F},{\cal E}))\,,$$
which (taking into account the exact sequence (\ref{ExSeqHom})) are isomorphisms.
Let now $v\in T_p(\Pg)$, and $w=(v,0)\in T_{(p,x)}(\Pg\times X)$. By Lemma \ref{claimLemma}  below, we obtain
$$\theta_0(h(v))=\tau_0(\epsilon_{(v,0)}(\EXUX))=\tau_0(\epsilon_{v}(\mathscr{E})).
$$
On the other hand, by Proposition \ref{coverings} in the appendix, it follows that $\epsilon_{v}(\mathscr{E})\ne 0$ if $v\ne 0$. This proves that $h$ is injective. Taking into account that $\dim(T_p(\Pg))=\dim(H^1({\cal H}om_0({\cal F},{\cal F})))=1$, the statement (1) follows.
\\ \\
(2) Using (\ref{DecFU}) we obtain an identification 
$\resto{\mathscr{F}^\sigma}{\{p\}\times U\times U}={\cal O}_{\{p\}\times U\times U}\oplus {\cal I}_{\{p\}\times\Delta_U}$. By Proposition \ref{IdealFam} proved in the appendix we obtain for a tangent vector $\eta\in T_xX$
$$R\big(\frac{\partial V_\varepsilon}{\partial x}(\eta)\big)=R\big( \epsilon_{(0,\eta)}(\mathscr{F}^\sigma)\big)=\epsilon_\eta({\cal I}_{\Delta_U})=\eta\,,
$$
which proves (2).
\end{proof}
\begin{lm}\label{claimLemma} Let $S$, $X$ be complex manifolds with $X$ compact, and let
\begin{equation}\label{SES}
0\to \mathscr{F}\textmap{j} \mathscr{E}\textmap{q} \mathscr{Q}\to 0	\end{equation}
be a short exact sequence of coherent sheaves on $S\times X$, where $\mathscr{E}$ and $\mathscr{Q}$ are flat over $S$.   We denote by ${\cal E}_s$, ${\cal F}_s$ the sheaves on $X$ given by  $\resto{\mathscr{E}}{\{s\}\times X}$, $\resto{\mathscr{F}}{\{s\}\times X}$. Then 
\begin{enumerate}
\item $\mathscr{F}$ is flat over $S$, and for any $s\in S$ the morphism $j_s:{\cal F}_s\to {\cal E}_s$ induced by $j$ is injective,
\item For any $s\in S$ and any tangent vector $w\in T_sS$ one has 
$$\theta(\epsilon_w(\mathscr{F}))=\tau(\epsilon_w(\mathscr{E}))\,,
$$
where $\theta:\Ext^1({\cal F}_s,{\cal F}_s)\to \Ext^1({\cal F}_s,{\cal E}_s)\,,\ \tau: \Ext^1({\cal E}_s,{\cal E}_s)\to \Ext^1({\cal F}_s,{\cal E}_s)$ are the morphisms induced by the monomorphism $j_s:{\cal F}_s\to {\cal E}_s$.
\end{enumerate}
\end{lm}
\begin{proof}
 (1) Using the short exact sequence (\ref{SES}) and the flatness of $\mathscr{E}$,  $\mathscr{Q}$ over $S$, it follows that for any $(s,x)\in S\times X$ and any ${\cal O}_s$-module ${\cal A}$, one has $\mathrm{Tor}_k(\mathscr{F}_{(s,x)},{\cal A})=0$ for $k>0$. Therefore $\mathscr{F}$ is flat over $S$. The injectivity of $j_s$ follows from a well-known flatness criterion \cite[p. 150]{Mat}.
 
 (2) We may suppose that $S\subset\C$ is a disk centered at 0, $s=0$ and $w=\frac{d}{dz}$. The fibre $X_0$ is a divisor in $Y:=S\times X$.  Tensoring  the short exact sequence
 $$0\to {\cal O}_{X_0}\textmap{z\cdot} {\cal O}_{2X_0}\to {\cal O}_{X_0}\to 0
 $$
with the ${\cal O}_S$-flat sheaves $\mathscr{F}$, $\mathscr{E}$,  $\mathscr{Q}$ we obtain the following
the commutative  diagram with exact rows and columns
\begin{diagram}[h=6mm]
 &  & 0 &  &  	0&  &0 &  \\ 
&& \dTo & & \dTo &&\dTo&&\\
0&\rTo & \mathscr{F}_{X_0} &\rTo^{z'} &  	 \mathscr{F}_{2X_0}&\rTo^{p'} &\mathscr{F}_{X_0} &\rTo& 0\phantom{\,.}\\ 
&& \dInto^{j_0} & & \dInto^{J_0} && \dInto^{j_0}&&\\
0&\rTo & \mathscr{E}_{X_0} &\rTo^{z''} &  	 \mathscr{E}_{2X_0}&\rTo^{p''} &\mathscr{E}_{X_0} &\rTo& 0\phantom{\,.}\\
&& \dTo^{q_0} & & \dTo^{Q_0}&& \dTo^{q_0} &&\\
0&\rTo & \mathscr{Q}_{X_0} &\rTo^{\tilde z} &  	 \mathscr{Q}_{2X_0}&\rTo^{\tilde p} &\mathscr{Q}_{X_0} &\rTo& 0\,.\\
&& \dTo & & \dTo &&\dTo&&\\
 &  & 0 &  &  	0&  &0 &  
\end{diagram}
The morphisms $z'$, $z''$, $\tilde z$ are induced by multiplication with $z$, and $p'$, $p''$ , $\tilde p$ are restriction morphisms.
 The infinitesimal deformations $\epsilon_w(\mathscr{F})$, $\epsilon_w(\mathscr{E})$ are the extension classes of the first two horizontal rows. By definition, $\tau(\epsilon_w(\mathscr{E}))$ is the extension class of
\begin{equation}\label{SES1}
0\to  \mathscr{E}_{X_0} \textmap{\varphi''}   { p''}^{-1}(\mathscr{F}_{X_0} )\textmap{p''}  \mathscr{F}_{X_0}\to  0\,,	
\end{equation}
 and $\theta(\epsilon_w(\mathscr{F}))$ is the extension class of
\begin{equation}\label{SES2}
0\to  \mathscr{E}_{X_0}\to  \mathscr{F}_{2X_s}\oplus_{\mathscr{F}_{X_0}} \mathscr{E}_{X_0}\textmap{p'\circ p_1} \mathscr{F}_{X_0}\to 0\,.	
\end{equation}
Let  $k: \mathscr{F}_{2X_s}\oplus \mathscr{E}_{X_0}\to \mathscr{E}_{2X_0}$ be the morphism given by   $k(x,y)= J_0(x)-z''(y)$. A simple diagram chasing shows that  $\ker(k)=(z'\oplus j_0)(\mathscr{F}_{X_0})$ and $\im(k)= { p''}^{-1}(\mathscr{F}_{X_0} )$. Therefore $k$ induces an isomorphism 
$$\mathscr{F}_{2X_s}\oplus_{\mathscr{F}_{X_s}}\mathscr{E}_{X_s}\to  { p''}^{-1}(\mathscr{F}_{X_s})\,.$$
It also defines an isomorphism between the   extensions (\ref{SES1}), (\ref{SES2}).
  \end{proof}

\subsection{Topological properties of the map $V_\varepsilon$}

We will need the following notations:
$${\cal M}^\si_{\rm sing}:=\big\{[{\cal F}]\in {\cal M}^\si|\ {\cal F}\hbox{ is singular}\big\}\,,\ {\cal M}^\si_{\rm lf}:=\big\{[{\cal F}]\in {\cal M}^\si|\ {\cal F}\hbox{ is locally free}\big\}\,, $$
$${\cal M}^\st_{\rm lf}=\big\{[{\cal F}]\in {\cal M}^\si_{\rm lf}|\ {\cal F}\hbox{ is stable}\big\}\,,\ \{{\cal M}^\st_{\rm lf}\}_{\rm reg}=\big\{[{\cal E}]\in {\cal M}^\st_{\rm lf}|\ H^2({\cal E}nd_0({\cal E}))=0\big\}\,.
$$
Note that ${\cal M}^\si_{\rm sing}$ is an analytic set of ${\cal M}^\si$,  ${\cal M}^\si_{\rm lf}$ is its (Zariski open) complement, and ${\cal M}^\st_{\rm lf}$ ($\{{\cal M}^\st_{\rm lf}\}_{\rm reg}$) is a Hausdorff open subset of ${\cal M}^\si_{\rm lf}$ which can be identified with the moduli space ${\cal M}^\st(E)$ (respectively ${\cal M}^\st(E)_\reg$) introduced in the section \ref{intro}.

The first result of this section shows that the map $V_\varepsilon$  has an important geometric interpretation:  for a point 
$p_0\in\Ag_\varepsilon$  (with $\varepsilon$ sufficiently small) any small {\it singular} simple deformation of ${\cal F}(p_0,x_0,\eta_0)$ can be obtained by deforming the triple $(p_0,x_0,\eta_0)$, so any such deformation remains in the image of $V_\varepsilon$.  

\begin{pr} \label{Open1} Let $(X,g)$ be a class VII surface endowed with a Gauduchon metric. For sufficiently small $\varepsilon>0$ the  map  $\Ag_\varepsilon\times X \to {\cal M}^\si_{\rm sing}$ induced by $V_\varepsilon$ is open.\end{pr}
\begin{proof}   Let $(p_0,x_0)\in  \Ag_\varepsilon\times X$, $\eta_0:{\cal E}_p(x_0)\to\C$ be an admissible epimorphism, and let $\mathscr{H}$ be a sheaf on $B\times X$, flat over $B$ (where $B\subset\C^4$ is the standard ball), with an identification ${\cal H}_0={\cal F}(p_0,x_0,\eta_0)$, which is a universal deformation of ${\cal F}(p_0,x_0,\eta_0)$  in the category of simple, torsion-free sheaves with trivial determinant and $c_2=1$. Note that ${\cal M}^{\si}$ is smooth of dimension 4 at $[{\cal F}(p_0,x_0,\eta_0)]$ by regularity and the Riemann-Roch theorem \cite{To}. Put 
$$S:=\{s\in B|\ {\cal H}_s\hbox{ is singular}\}, 
$$
let $\mathscr{S}$ be the restriction of $\mathscr{H}$ to $S\times X$ and, for $s\in S$, denote by ${\cal S}_s$ the restriction of $\mathscr{S}$ to $\{s\}\times X$ (which coincides with ${\cal H}_s$). It suffices to prove that  there exists an open neighbourhood $U$ of $0$ in $S$, and a holomorphic map $f:U\to \Ag_\varepsilon\times X$ such that $f(0)=(p_0,x_0)$, and for any $s\in U$ the isomorphism type of ${\cal S}_s$ coincides with $V_\varepsilon(f(s))$. By   Proposition \ref{singF} and \cite[Lemma 9.6.1]{HL} it follows that $\mathscr{S}^{\we}$ is locally free, and for any $s\in S$ the induced morphism 
$\resto{\mathscr{S}}{\{s\}\times X}\to \resto{\mathscr{S}^\we}{\{s\}\times X}
$
coincides with the canonical embedding ${\cal S}_s\to{\cal S}_s^\we$, in particular it is a monomorphism. Using the short exact sequence
$$0\to \mathscr{S}\hookrightarrow \mathscr{S}^{\we}\to \qmod{\mathscr{S}^{\we}}{\mathscr{S}}\to 0
$$
and a well-known flatness criterion \cite[p. 150]{Mat}, we see that $\mathscr{Q}:= {\mathscr{S}^{\we}}/{\mathscr{S}}$ is flat over $S$, and, for any $s\in S$, the sequence
$$0\to {\cal S}_s\to ({\cal S}^{\we})_s={\cal S}_s^{\we}\to {\cal Q}_s\to 0
$$
is exact. Using again Proposition  \ref{singF} it follows that for any $s\in S$
\begin{enumerate}
\item there exists a unique point $\psi(s)\in X$ such that ${\cal Q}_s\simeq{\cal O}_{\{\psi(s)\}}$, 
\item $\mathscr{S}^{\we}$ is a flat family of locally free sheaves on $X$ with trivial determinant and $c_2=0$.
\end{enumerate}

But the map $x\mapsto {\cal O}_{\{x\}}$ defines a biholomorphism between $X$ and the Douady moduli space of length 1 quotient sheaves of ${\cal O}_X$. Therefore the flatness of $\mathscr{Q}$ over $S$ shows that the map $\psi:S\to X$ is holomorphic. On the other hand (supposing $\varepsilon$ is sufficiently small)  the family $\mathscr{E}$ is versal at $p_0$  by Corollary \ref{versal}; it follows that there exists  an open neighbourhood $U$ of $0$ in $S$,  a holomorphic map $\varphi:U\to \Ag_\varepsilon$, and an isomorphism  
$$\resto{\mathscr{S}^{\we}}{U\times X}\simeq(\varphi\times\id_X)^*(\mathscr{E}).$$
Therefore, with these notations, we see that for any $s\in S$, the sheaf ${\cal S}_s$ fits in an exact sequence
$$0\to  {\cal S}_s\to {\cal E}_{\varphi(s)}\textmap{r_s} {\cal O}_{\{\psi(s)\}}\to 0.
$$
Since ${\cal S}_s$ is simple, it follows easily that $r_s$ is induced by an admissible epimorphism ${\cal E}_{\varphi(s)}(\psi(s))\to\C$. Therefore the isomorphism type of ${\cal S}_s$ is $V(\varphi(s),\psi(s))$.
\end{proof}
Proposition \ref{Open1} can be reformulated as follows: 
\begin{co}\label{OpenCo} In the conditions of Proposition \ref{Open1} the following holds: For any pair $(p_0,x_0)\in \Ag_\varepsilon\times X$, and for any neighbourhood $N$ of $(p_0,x_0)$,  the point $V_\varepsilon(p_0,x_0)$ has a neighbourhood ${\cal U}$ such that ${\cal U}\setminus V_\varepsilon(N) \subset {\cal M}^\si_{\rm lf}$.
\end{co}

We will need the following stronger version of this statement:

\begin{pr} \label{Open2} In the conditions of Proposition \ref{Open1} the following holds: For any pair $(p_0,x_0)\in \Ag_\varepsilon\times X$, and for any neighbourhood $N$ of $(p_0,x_0)$,  the point $V_\varepsilon(p_0,x_0)$ has a neighbourhood ${\cal U}$ such that ${\cal U}\setminus V_\varepsilon(N) \subset {\cal M}^\st_{\rm lf}$.
\end{pr}

For the proof we will need the following simple general results concerning filtrable $\SL(2,\C)$-bundles with $c_2=1$ on class VII surfaces.  Let $X$ be a class VII surface, and put $b:=b_2(X)$. The intersection form $q_X$ of $X$ is negative definite. Using the first Donaldson theorem \cite{D2} and the fact that $c_1({\cal K}_X)$ is a characteristic element for $q_X$, one obtains (see \cite{Te4}, \cite{Te5}) a basis $(e_1,\dots,e_b)$ of $H^2(X,\Z)/\Tors(H^2(X,\Z))$ such that 
$$q_X(e_i,e_j)=-\delta_{ij}\,,\ c_1({\cal K}_X)+\Tors=\sum_{i=1}^b e_i.$$  
Using these classes, the filtrable $\SL(2,\C)$-bundles with $c_2=1$ on   $X$ can be easily classified as follows.   

\begin{lm}\label{filtrables}
Let ${\cal E}$ be a	filtrable $\SL(2,\C)$-bundle with $c_2=1$ on $X$, and $q:{\cal E}\to {\cal M}$ be an epimorphism on a  torsion-free coherent sheaf of rank 1. Then one of the following holds
\begin{enumerate}
\item ${\cal M}$ is  locally free, there exists $i\in \{1,\dots,b\}$ such that  $c_1({\cal M})+\Tors\in \{\pm e_i\}$, and $\ker(q)\simeq {\cal M}^{\smvee}$. Therefore in this case ${\cal E}$ fits in a short exact sequence	
$$0\to {\cal M}^{\smvee}\to {\cal E}\to {\cal M}\to 0.
$$
with $c_1({\cal M})+\Tors\in \{\pm e_i\}$.
\item There exists a line bundle ${\cal L}$ on $X$ with torsion Chern class, and a point $x\in X$ such that 
 \begin{enumerate}
\item  The pair $(x,{\cal L}^{\otimes 2}\otimes {\cal K}_X)$ satisfies the Cayley-Bacharach condition: any   section of ${\cal L}^{\otimes 2}\otimes {\cal K}_X$ vanishes at $x$.
\item ${\cal M}\simeq {\cal L}\otimes {\cal I}_x$, and $\ker(q)\simeq {\cal L}^{\smvee}$.
\end{enumerate}
Therefore in this case ${\cal E}$ fits in a short exact sequence	
$$0\to {\cal L}^{\smvee}\to {\cal E}\to {\cal L}\otimes {\cal I}_x\to 0.$$
\end{enumerate}

\end{lm}
 
 Recall that any filtrable locally free sheaf of rank 2 on a complex surface admits an epimorphism onto a torsion-free coherent sheaf of rank 1, hence Lemma \ref{filtrables} classifies all filtrable $\SL(2,\C)$-bundles with $c_2=1$ on class VII surfaces. Lemma \ref{filtrables} shows that 
 \begin{lm}\label{filtrablesLm} Let $X$ be a class VII surface. The set 
\begin{multline*}
\big\{c_1({\cal M}^{\we})|\    \hbox{${\cal M}$ is a rank 1 torsion-free quotient }\\ \hbox{ 	of an  $\SL(2,\C)$-bundle ${\cal E}$ with $c_2=1$ on $X$}\big\}	
\end{multline*}
is finite.
 \end{lm}

\begin{proof}  (of Proposition \ref{Open2})
As in the proof of Proposition \ref{Open1} let $(p_0,x_0)\in  \Ag_\varepsilon\times X$, $\eta_0:{\cal E}_p(x_0)\to\C$ be an admissible epimorphism, and let $\mathscr{H}$ be a sheaf on $B\times X$, flat over $B$ (where $B\subset\C^4$ is the standard ball), with an identification ${\cal H}_0={\cal F}(p_0,x_0,\eta_0)$, which is a universal deformation of ${\cal F}(p_0,x_0,\eta_0)$  in the category of simple, torsion-free sheaves with trivial determinant and $c_2=1$. Taking into account Corollary \ref{OpenCo}, the conclusion of Proposition \ref{Open2} reduces to the 
\vspace{2mm}\\ 
{\it Claim:} For sufficiently small $\alpha>0$ one has
$\big\{\{[{\cal H}_z]|\ z\in B_\alpha\}\setminus V_\varepsilon(N)\big\}\subset  {\cal M}^\st_{\rm lf}$.
\vspace{2mm}\\
Suppose by reductio ad absurdum that this does not hold. Then there would exist  a sequence $(z_n)_n$ in $B$ such that $\lim_{n\to\infty} z_n=0$, and for any $n\in\N$ one has
\begin{enumerate}
\item 	$[{\cal H}_{z_n}]\not\in V_\varepsilon(N)$,
\item ${\cal H}_n$ is not stable.
\end{enumerate}

By Corollary \ref{OpenCo} the sheaf ${\cal H}_{z_n}$ will be locally free for any sufficiently large $n$. We may assume that this is the case for any $n$. Since ${\cal H}_{z_n}$ is not stable, there exists a line bundle ${\cal L}_n$ with $\deg({\cal L}_n)\leq 0$ and a non-trivial morphism $u_n:{\cal H}_{z_n}\to {\cal L}_n$.  We may of course assume that ${\cal L}_n$ is ``minimal", i.e. it coincides with the bidual $\im(u_n)^{\we}$.

By Lemma \ref{proper} below it follows that the sequence $\big(\deg_g({\cal L}_{n})\big)_n $ is also bounded from below. Using Lemma \ref{filtrablesLm} (which applies, because ${\cal H}_{z_n}$ are locally free) it follows that $\big([{\cal L}_{n}]\big)_n$ is a sequence of a compact subset of $\Pic(X)$, hence, passing again to a subsequence if necessary, we may assume that this sequence converges to a point $[{\cal L}_\infty]\in\Pic(X)$.  We will have $\deg_g({\cal L}_\infty)\leq 0$, and, using Flenner's semicontinuity theorem (see Proposition \ref{SemiCont}), it follows $\Hom({\cal H}_0,{\cal L}_\infty)\ne 0$. Since ${\cal H}_0={\cal F}(p_0,x_0,\eta_0)$, choosing a representative $l_0\in p_0$ and using Remark \ref{CanShESFRem},  we obtain $\Hom({\cal L}_{l_0},{\cal L}_\infty)\ne 0$, or $\Hom({\cal L}^\smvee_{l_0},{\cal L}_\infty)\ne 0$. But $\deg_g({\cal L}_{l_0}^\smvee\otimes{\cal L}_\infty)\leq \varepsilon$, $\deg_g({\cal L}_{l_0}\otimes{\cal L}_\infty)\leq \varepsilon$. Choosing $\varepsilon<\vg(g)$  it will follow by Remark \ref{VanSmallDeg} that ${\cal L}_\infty\simeq {\cal L}_{l_0}$, or ${\cal L}_\infty\simeq {\cal L}_{l_0}^\smvee$, so in particular $[{\cal L}_\infty]\in\Pic^\T$. But this implies that for any sufficiently large $n\in\N$ one has $[{\cal L}_{n}]\in\Pic^\T$.  Since ${\cal H}_{n}$ is locally free, Lemma \ref{filtrables} shows that for any (sufficiently large $n$) there exists $x_n\in X$ such that $\Hom({\cal H}_{n},{\cal L}_{n}\otimes {\cal I}_{x_n})\ne 0$. Using the compactness of $X$, we get a point $x_\infty\in X$ such that $\Hom({\cal H}_{0},{\cal L}_{\infty}\otimes {\cal I}_{x_\infty})\ne 0$. But it is easy to see that (supposing $|\deg_g(l_0)|$ is sufficiently small) one has 
 $$\Hom({\cal F}(p_0,x_0,\eta_0),{\cal L}_{l_0}\otimes {\cal I}_{x})= \Hom({\cal F}(p_0,x_0,\eta_0),{\cal L}_{l_0}^\smvee\otimes {\cal I}_{x})=0\ \forall x\in X.$$
Therefore we obtained a contradiction, which completes the   proof.\end{proof}
\begin{lm}\label{proper}
Let $(X,g)$ be a compact Gauduchon surface, $M$ be a complex manifold and $\mathscr{E}$ be a sheaf on $M\times X$, flat over $M$. For any compact subset $K\subset M$ the set
$$\{\deg_g({\cal L})|\ \exists u\in K \hbox{ such that }\Hom({\cal E}_u,{\cal L})\ne 0\}
$$ 
is bounded from below.	
\end{lm}
\begin{proof}
Using \cite[Proposition  3.1]{To} it follows that there exists a sheaf epimorphism $q: \mathscr{F}\to \mathscr{E}$, where $\mathscr{F}$  is a locally free sheaf on $M\times X$. Therefore, for any $u\in M$ we get an epimorphism $q_u:{\cal F}_u\to {\cal E}_u$, where ${\cal F}_u$ is locally free. Let $h$ be a Hermitian metric on $\mathscr{F}$, $h_u$ be the induced metric on ${\cal F}_u$, and for any line bundle ${\cal L}$ on $X$, let $h_{\cal L}$ be a Hermitian-Einstein metric on ${\cal L}$. We denote by $A_u$, $A_{{\cal L}}$ the Chern connections associated with $h_u$, $h_{\cal L}$. One has
$$F_{A_u^\smvee\otimes A_{\cal L}}=F_{A_u^\smvee}\otimes\id_{{\cal L}}+\id_{{\cal F}_u^\smvee}\otimes F_{A_{\cal L}}=F_{A_u^\smvee}\otimes\id_{{\cal L}}+F_{A_{\cal L}}\id_{{\cal F}_u^\smvee\otimes{\cal L}}  
$$
Since $K$ is compact, and $i\Lambda F_{A_{\cal L}}=c_{\cal L}$, where the Einstein constant $c_{\cal L}$ is proportional to $\deg_g({\cal L})$, it follows that there exists $d_K\in\R$ such that $i\Lambda_g(F_{A_u^\smvee\otimes A_{\cal L}})$ is negative definite for any $u\in K$ and any line bundle ${\cal L}$ with $\deg({\cal L})<d_K$. Using a well-known vanishing theorem \cite{Ko}, we obtain $\Hom({\cal F}_u,{\cal L})=0$ for any $u\in K$ and line bundle ${\cal L}$ with $\deg({\cal L})<d_K$.  It suffices to note that $\Hom({\cal F}_u,{\cal L})=0$ implies $\Hom({\cal E}_u,{\cal L})=0$.
\end{proof}

Put ${\cal V}_\varepsilon:=V_\varepsilon(\Ag_\varepsilon\times X)$.
\begin{co}\label{MainCo} Let $(X,g)$ be a class VII surface endowed with a Gauduchon metric. For sufficiently small $\varepsilon>0$ there exists a Hausdorff open neighborhood ${\cal U}$ of ${\cal V}_\varepsilon$ in ${\cal M}^\si$ such that
\begin{enumerate}
\item ${\cal V}_\varepsilon$ is closed in ${\cal U}$,
\item ${\cal U}\setminus {\cal V}_\varepsilon\subset \{{\cal M}^\st_{\rm lf}\}_\reg$.
\end{enumerate}
\end{co}
\begin{proof}
Using 	Propositions \ref{embedding}  and \ref{Open2}, every point $e\in {\cal V}_\varepsilon$ has a neighbourhood ${\cal U}_e$ in ${\cal M}^\si_\reg$ such that ${\cal V}_\varepsilon\cap {\cal U}_e$ is closed in ${\cal U}_e$ and ${\cal U}_e\setminus {\cal V}_\varepsilon\subset {\cal M}^\st_{\rm lf}$.  It suffices to put 
$${\cal U}:=\bigg(\union_{e\in {\cal V}_\varepsilon} {\cal U}_e\bigg)\cap {\cal N},$$
where ${\cal N}$ is a Hausdorff neighbourhood of $V_\varepsilon(\bar \Ag_\varepsilon\times X)$ (see Proposition \ref{criterion}).
\end{proof}

\section{Extending the complex structure}
\label{ExtendingSect}
 
Let $(X,g)$ be  a class VII surface endowed with a Gauduchon metric.  We have seen that, for a pair $(p,x)\in\big(\Pg\setminus \Cg\big)\times X$ and an admissible epimorphism $\eta:{\cal E}(x)\to \C$, the sheaf ${\cal F}(p,x,\eta)$  is not even slope semistable, because it is destabilised by its subsheaf ${\cal L}^\smvee\otimes {\cal I}_x$, where $l=[{\cal L}]$ is the unique representative of $p$ with negative degree. We will extend the holomorphic structure of ${\cal M}^\st(E)$ to its Donaldson compactification using the gluing Lemma \ref{glue}. The parametrisation $f$ intervening  in this theorem will be defined on an open neighbourhood of ${\cal V}_\varepsilon$ in  ${\cal M}^\si$, and will map an element $[{\cal F}(p,x,\eta)]\in {\cal V}_\varepsilon$ with $|\deg_g|(p)>0$ to the isomorphism class of a stable, locally free Serre extension. This  isomorphism class is non-separable (in the Hausdorff sense)  from $[{\cal F}(p,x,\eta)]$ in ${\cal M}^\si$. 

\subsection{Families of Serre extensions}\label{Serre}

The family  of Serre extensions we need is  defined as follows: \\

Let $U$ be an open set of $X$. Denote by $\mathscr{L}_\varepsilon$, $\mathscr{L}_{\varepsilon U}$  the restriction of the Poincaré line bundle $\mathscr{L}$ to $A_\varepsilon\times X$, respectively $A_\varepsilon\times U$ . On the product $\Pi_\varepsilon^U:=A_\varepsilon\times U\times X$  consider the line bundle
$$\mathscr{M}:=p_{A_\varepsilon X}^*(\mathscr{L}_\varepsilon)\otimes p_{A_\varepsilon U}^*(\mathscr{L}^{\smvee}_{\varepsilon U}),
$$
Put $Z:=A_\varepsilon\times\Delta_{UX}$, where $\Delta_{UX}$ is the graph of the embedding   $U\hookrightarrow X$.  We have canonical isomorphisms
$${\cal E}xt^1\big({\cal I}_{Z}\otimes \mathscr{M}^\smvee,\, p_U^*({\cal K}_U)\otimes\mathscr{M} \big)={\cal E}xt^1\big({\cal I}_{Z}\,,\, p_U^*({\cal K}_U)\otimes\mathscr{M}^{\otimes 2}\big) =$$
$$=\omega_{Z/\Pi^U_\varepsilon}\otimes_{{\cal O}_Z} \big(p_U^*({\cal K}_U)\otimes \mathscr{M}^{\otimes 2}\big)  \simeq {\cal O}_Z\,.
$$
For the second isomorphism we used formula (\ref{can}) proved in the appendix; for the third isomorphism we used the obvious isomorphism $\resto{\mathscr{M}}{Z}={\cal O}_Z$, and the obvious identification between the normal bundle of $Z$ and the pull-back of the tangent bundle $T_U$. The local-global spectral  sequence gives an exact sequence
\begin{equation}\label{LGSS} 
\begin{gathered} 0\to H^1\big(p_U^*({\cal K}_U)\otimes\mathscr{M}^{\otimes 2}\big)\to  \Ext^1\big({\cal I}_{Z}\otimes \mathscr{M}^\smvee,\, p_U^*({\cal K}_U)\otimes\mathscr{M}\big)\textmap{\ag_U}\\
H^0\big({\cal E}xt^1\big({\cal I}_{Z}\otimes \mathscr{M}^\smvee,\, p_U^*({\cal K}_U)\otimes\mathscr{M} \big)\big)=H^0({\cal O}_Z)\to  H^2\big(p_U^*({\cal K}_U)\otimes\mathscr{M}^{\otimes 2}\big)\,. 
\end{gathered}
\end{equation}

The cohomology spaces $H^k\big(p_U^*({\cal K}_U)\otimes\mathscr{M}^{\otimes 2}\big)$ can be computed using the Leray spectral sequence associated with the projection $p_{A_\varepsilon U}:\Pi^U_\varepsilon\to A_\varepsilon\times U$.  Using the projection formula and denoting by $q_{A_\varepsilon}:A_\varepsilon\times U \to A_\varepsilon$, $q_U:A_\varepsilon\times U \to U$, $\pi_{A_\varepsilon}:A_\varepsilon\times X\to A_\varepsilon$ the obvious projections, we get
\begin{equation}\label{Rk} R^k(p_{A_\varepsilon U*})\big(p_U^*({\cal K}_U)\otimes\mathscr{M}^{\otimes 2}\big)=   \mathscr{L}_{\varepsilon U}^{\smvee\otimes 2}\otimes q_U^*({\cal K}_U)\otimes q_{A_\varepsilon}^*\big(R^k \pi_{A_\varepsilon*}(\mathscr{L}^{\otimes 2}_\varepsilon)\big)\,.
\end{equation}
Since $H^0({\cal L}_l^{\otimes 2})=0$ for generic $l$,  formula  (\ref{Rk}) shows   that $p_{A_\varepsilon U*}\big(p_U^*({\cal K}_U)\otimes\mathscr{M}^{\otimes 2}\big)=0$. Let now $\varepsilon>0$ be sufficiently small such that 
\begin{enumerate}[(C1)]
\item $H^2({\cal L}_l^{\otimes 2})=0$ for  any $l\in A_\varepsilon$ (see Lemma \ref{H2FlatLB}), 
\item $H^0({\cal L}_l^{\otimes 2})=0$ for any $l\in A_\varepsilon\setminus\rho$ (see Remark \ref{VanSmallDeg}).
\end{enumerate}

The first condition implies $R^2(p_{A_\varepsilon U*})\big(p_U^*({\cal K}_U)\otimes\mathscr{M}^{\otimes 2}\big)=0$, and the second implies that the sheaf $R^1(p_{A_\varepsilon U*})\big(p_U^*({\cal K}_U)\otimes\mathscr{M}^{\otimes 2}\big)$ is supported on $\rho\times U$.  Therefore, if $U$ is Stein, one obtains $H^i\big(R^k(p_{A_\varepsilon U*})\big(p_U^*({\cal K}_U)\otimes\mathscr{M}^{\otimes 2}\big)\big)=0$ for $i+k=2$, so $H^2\big(p_{U*}({\cal K}_U)\otimes\mathscr{M}^{\otimes 2}\big)=0$, implying that the morphism $\ag_U$ in (\ref{LGSS}) is surjective.
For any $\nu\in \ag_U^{-1}(1)$ we obtain a short exact sequence
$$0\to p_U^*({\cal K}_U)\otimes\mathscr{M}\to \mathscr{S}^\nu\to {\cal I}_{Z}\otimes \mathscr{M}^\smvee\to 0
$$
on $\Pi_\varepsilon^U$  with locally free central term (by Serre's Lemma, see \cite[Lemma 5.1.2]{OSS}). The 2-bundle ${\cal S}^\nu_{(l,u)}$ on $X$ defined by the restriction of $\mathscr{S}^\nu$ to $\{(l,u)\}\times X$ fits in a short exact sequence
$$0\to {\cal L}_l\to {\cal S}^\nu_{(l,u)}\to {\cal I}_u\otimes {\cal L}_l^\smvee\to 0\,, 
$$ 
so it is a Serre extension of ${\cal I}_u\otimes {\cal L}_l^\smvee$ by ${\cal L}_l$. On the other hand

\begin{pr} \label{SSheaves} Let $\varepsilon>0$ be  sufficiently small such that conditions (C1), (C2) above hold. Then  
\begin{enumerate}	
\item For every $(l,x)\in   \big(A_\varepsilon\setminus\rho\big) \times X$  one has $\dim(\Ext^1({\cal L}_l^\smvee\otimes {\cal I}_x,{\cal L}_l))=1$, and there exists a  locally free sheaf ${\cal S}(l,x)$, unique up to isomorphism, which fits in an exact   sequence of the form
$$0\to {\cal L}_l\to {\cal S}(l,x)\to {\cal L}^\smvee_l\otimes {\cal I}_x\to 0.
$$
\item  Suppose that   $\varepsilon<\vg(g)$. Then ${\cal S}(l,x)$ is  (semi)stable if and only of $\deg_g(l)<0$ ($\deg_g(l)\leq 0$) . 
\end{enumerate}	
	\end{pr}

\begin{proof}
(1) The first statement follows easily using the exact sequence
$$0\to H^1({\cal L}_l^{\otimes 2})\to \Ext^1({\cal L}_l^\smvee\otimes {\cal I}_x,{\cal L}_l)\to H^0({\cal E}xt^1({\cal L}_l^\smvee\otimes {\cal I}_x,{\cal L}_l)\to H^2({\cal L}_l^{\otimes 2})\to 0
$$
and the conditions (C1), (C2).
\\ \\
(2) If $\deg_g(l)\geq 0$ ($\deg_g(l) > 0$) then ${\cal S}(l,x)$ is obviously non-stable (non semi-stable). Conversely, suppose that $\deg_g(l)< 0$ ($\deg_g(l)\leq  0$). If ${\cal L}\hookrightarrow {\cal S}(l,x)$ is a destabilising locally free subsheaf of rank 1 of ${\cal S}(l,x)$, then $H^0({\cal L}^\smvee\otimes {\cal L}_l^\smvee\otimes {\cal I}_x)\ne 0$ and $\deg_g({\cal L})\geq 0$ (respectively $\deg_g({\cal L})> 0$). This implies that the line bundle ${\cal L}^\smvee\otimes {\cal L}_l^\smvee$ has a non-trivial section whose zero divisor $D$ contains $x$, in particular is non-empty. We get
$$\vol_g(D)=|\deg_g(l)|-\deg_g({\cal L})\leq |\deg_g(l)|,
$$
which contradicts the assumption $\varepsilon< \vg(g)$ (see Definition \ref{vg}).\end{proof}
For $l\in \rho$ one has ${\cal L}_l^\smvee\simeq {\cal L}_l$, hence $H^1({\cal L}_l^{\otimes 2})\simeq H^1({\cal O}_X)\simeq \C$.  Therefore in this case the situation is different:
\begin{re}
For any line bundle ${\cal L}$ on $X$ one has $\dim(\Ext^1({\cal L}\otimes {\cal I}_x,{\cal L}))=2$, and the set of isomorphism classes of locally free sheaves ${\cal S}$ which fit into an exact sequence of the form	
$$0\to {\cal L}\to {\cal S}\to {\cal L}\otimes {\cal I}_x\to 0.
$$
can be identified with the complement of a singleton in  $\P(\Ext^1({\cal L}\otimes {\cal I}_x,{\cal L}))$.
\end{re}

\begin{co}
Let $\varepsilon$ be sufficiently small such that (C1), (C2) hold, and $(l,x)\in (A_\varepsilon\setminus\rho)\times X$. Then ${\cal S}^\nu_{(l,x)}\simeq {\cal S}(l,x)$ for any Stein open set $U\ni x$ and  any $\nu\in\ag_U^{-1}(1)$. 	
\end{co}

\subsection{The proof of Theorem \ref{Th1}}

For a point $p=\{l,l^\smvee\}\in\Cg$, we denote by $\ag_p$ the gauge class of the flat $\SU(2)$ connection $a_l\oplus a_{l^\smvee}$, where, for $l\in \mathrm{C}$, $a_l$ stands for the  flat $\U(1)$-connection associated with $l$. For $p\in\Ag_\varepsilon\setminus\Cg$ we put  ${\cal S}(p,x):={\cal S}(l,x)$, where $l$ is the representative of $p$ of negative degree. Let $\varepsilon>0$ be sufficiently small such that the properties of Proposition \ref{SSheaves} hold, and let ${\cal U}$ be an open neighbourhood of ${\cal V}_\varepsilon$ in ${\cal M}^\si$ satisfying the properties of Corollary \ref{MainCo}. Put 
\begin{align*}
{\cal V}_0&:=\{[{\cal F}({\cal E}_p,x,\eta)]|\ p\in\Cg,\ x\in X,\ \eta:{\cal E}_p(x)\textmap{{\rm adm.\  epim.}}\C\}\\
&\ =V_\varepsilon(\Cg\times X)\,.
\end{align*}

Recall that, by Proposition \ref{R0X}, the union  $\Mg:={\cal M}^\ASD(E)^*_{\reg}\cup \big({\cal R}_0\times X\big)$ is open in $\overline{{\cal M}}^\ASD(E)$ and is a topological 8-manifold. Define $f:{\cal U} \to \Mg$  by
\begin{equation}\label{f}
f([{\cal F}]):=\left\{
\begin{array}{ccc}
[{\cal F}] &\rm if & [{\cal F}]\in {\cal U}\setminus {\cal V}_\varepsilon\,,\\
\ [{\cal S}(p,x)] &\rm if & [{\cal F}]=V_\varepsilon(p,x)\hbox{ with }(p,x)\in{\cal V}_\varepsilon\setminus {\cal V}_0\,,\\
(\ag_p,x) &\rm if &[{\cal F}]=V_\varepsilon(p,x)\hbox{ with }(p,x)\in\  {\cal V}_0\,.
\end{array}
\right.	
\end{equation}
In the  first two lines of the  above formula we have used the Kobayashi-Hitchin identification ${\cal M}^\ASD_a(E)^*\simeq {\cal M}^\st(E)={\cal M}^\st_{\rm lf}$.
\begin{pr} Let $(X,g)$ be a class VII surface endowed with a Gauduchon metric. 
If $\varepsilon$ and ${\cal U}$ are sufficiently small, the following holds:
\begin{enumerate}
\item $f$ is injective.
 \item $f^{-1}({\cal M}^\st(E)_\reg)={\cal U}\setminus {\cal V}_0$, and the map ${\cal U}\setminus {\cal V}_0\to {\cal M}^\st(E)_\reg$ induced by $f$ is   holomorphic.
 \item $f$ is continuous.

\end{enumerate}
	
\end{pr}

\begin{proof} (1) We prove  that (assuming $\varepsilon$ and ${\cal U}$  sufficiently small) $f$ is injective. The restrictions $\resto{f}{{\cal U}\setminus{\cal V}_\varepsilon}$, $\resto{f}{{\cal V}_\varepsilon}$ are obviously injective, so it suffices to prove  that  one has $\im\big(\resto{f}{{\cal U}\setminus{\cal V}_\varepsilon})\cap\im\big(\resto{f}{{\cal V}_\varepsilon}\big)=\emptyset$.  Since $f({\cal U}\setminus{\cal V}_\varepsilon)\subset {\cal U}$ it suffices to prove that (assuming $\varepsilon$ and ${\cal U}$   sufficiently small) one has
\begin{equation}\label{EmptInt} {\cal U} \cap f({\cal V}_\varepsilon)=\emptyset\,.
	\end{equation}
The set 
$$Y:=\{(l,[{\cal F}])\in \Pic^\T\times {\cal U}|\ H^0({\cal L}_l\otimes {\cal F})\ne 0\}
$$
is closed in $\Pic^\T\times {\cal U}$ by Grauert's semicontinuity theorem. Let $\Pic^\T_{[0,\varepsilon]}\subset \Pic^\T$  be the union of compact annuli defined by  the condition $0\leq \deg_g(l)\leq \varepsilon$. The projection 
$$Z:=\mathrm{p}_{\cal U}\big(Y\cap (\Pic^\T_{[0,\varepsilon]}\times {\cal U})\big)$$
  is a compact subset of ${\cal U}$  which is disjoint from ${\cal V}_\varepsilon$, because (for sufficiently small $\varepsilon>0$) one has 
$$H^0({\cal L}_l\otimes {\cal F}({\cal E}_p,x,\eta))=0$$
for any $(l,p)\in \Pic^\T_{[0,\varepsilon]}\times \Ag_\varepsilon$ and any admissible epimorphism $\eta:{\cal E}_p(x)\to\C$.
On the other hand one has obviously $f({\cal V}_\varepsilon)\subset Z$, hence, replacing ${\cal U}$ by ${\cal U}\setminus Z$ if necessary, formula $(\ref{EmptInt})$ will hold.\\
\\
(2) The equality $f^{-1}({\cal M}^\st(E)_\reg)={\cal U}\setminus {\cal V}_0$ is obvious taking into account the definition of $f$ and the way in which $\varepsilon$ and ${\cal U}$ have been chosen.  We prove  that the map ${\cal U}\setminus {\cal V}_0\to {\cal M}^\st(E)_\reg$ induced by $f$ is holomorphic.   The holomorphy on  ${\cal U}\setminus {\cal V}_\varepsilon$ is obvious. It remains to show  that $f$ is holomorphic at any point $[{\cal F}(p_0,x_0,\eta_0)]\in {\cal V}_\varepsilon\setminus {\cal V}_0$. Using Proposition \ref{embedding} (\ref{HolEmbed}) we can find an open neighbourhood $N$ of $(p_0,x_0)$ in $(\Ag_\varepsilon\setminus\Cg)\times X$ and an open embedding $\theta:N\times D\to {\cal U}$ such that $\theta(p,x,0)=V_\varepsilon(p,x)$ for any $(p,x)\in N$, where $D$ is the unit disk. It suffices to prove that $f\circ\theta$ is holomorphic on $N\times D$ or, equivalently, that $f\circ\theta$ is separately holomorphic. 

The holomorphy in the $N$ direction is obvious. For the holomorphy in the $D$-direction we proceed as follows. Fix $(p,x)\in N$. The map $\theta(p,x,\cdot):D\to {\cal U}$  is defined by a sheaf ${\cal F}$ on $D\times X$, flat over $D$ with an identification ${\cal F}_0={\cal F}(p,x,\eta)$. Let $l$ be the representative  of $p$ of negative degree, and 
$$r_{l,x,\eta}:{\cal F}(p,x,\eta)={\cal F}_0\to {\cal L}_l$$
 the corresponding destabilising epimorphism. The elementary transformation of the pair $({\cal F},r_{l,x,\eta})$	 is a sheaf ${\cal F}'$ on $D\times X$, flat over $D$, with ${\cal F}'_z={\cal F}_z$ for $z\ne 0$ and whose restriction to $\{0\}\times X$ fits into a short exact sequence
 $$0\to {\cal L}_l\to {\cal F}'_0\to {\cal L}_l^{\smvee}\otimes {\cal I}_x\to 0. 
 $$
 (see Section \ref{ElTrDef}). Using Corollary \ref{MainComparisonCo} it follows that the extension class associated to this exact sequence  is non-trivial, hence ${\cal F}'_0\simeq {\cal S}(p,x)$. On the other hand, since ${\cal F}'$ is flat over $D$, this proves that $(f\circ\theta)(p,x,\cdot):D\to {\cal U}$ is holomorphic on $D$. 
  \\ \\
 (3) Using (2) it follows that $f$ is continuous on ${\cal U}\setminus {\cal V}_0$. The critical continuity at the points of ${\cal V}_0$ follows using the continuity theorem \cite{BTT} as follows. Let $([{\cal F}_n])_n$ be a sequence in ${\cal U}$ with $\lim_{n\to\infty} [{\cal F}_n]=[{\cal F}_\infty]=V_\varepsilon(p_\infty,x)\in {\cal V}_0$. We have to prove that 
\begin{equation}\label{lim}
\lim_{n\to\infty} f([{\cal F}_n])=([\ag_{p_\infty}],x).	
\end{equation}
 Using a standard argument, we can reduce the problem to three separate cases:
 \begin{enumerate}[1.]
\item $([{\cal F}_n])_n$ is a sequence of ${\cal V}_0$.
\item  $([{\cal F}_n])_n$ is a sequence of ${\cal V}_\varepsilon\setminus {\cal V}_0$.

\item  	$([{\cal F}_n])_n$ is a sequence of ${\cal U}\setminus {\cal V}_\varepsilon$.
\end{enumerate}

In first case the claim  (\ref{lim}) follows easily using the fact the map $V_\varepsilon$ is homeomorphism on its image (see Proposition \ref{embedding}).  In the second case one applies the continuity theorem to a family of Serre extensions $\mathscr{S}^\nu$ associated with a Stein open set $U\ni x$ and a lift $\nu\in\ag_U^{-1}(1)$ (see Section \ref{Serre}). In the third case one applies the continuity theorem \cite{BTT} to  a versal deformation of the simple, torsion-free sheaf ${\cal F}(p_\infty,x,\eta)$, where $\eta:{\cal E}_{p_\infty}\to\C$ is an admissible epimorphism.
 \end{proof}

 Theorem \ref{Th1} stated in Section  \ref{FundamQ} follows now directly from the gluing Lemma \ref{glue}.
 
\subsection{The proof of Theorem \ref{Th2}} 

Let $X$ be a complex surface, ${\cal E}$ be a holomorphic bundle on $X$,  $\pi:\P({\cal E})\to X$  be the projectivisation of ${\cal E}$,  ${\cal L}\subset \pi^*({\cal E})$ be the tautological line bundle on $\P({\cal E})$, and ${\cal M}$ the quotient $\pi^*({\cal E})/{\cal L}$.  On the product $\P({\cal E})\times X$ consider the bundles
$$\mathscr{E}:=p_2^*({\cal E})\,,\ \mathscr{L}:=p_1^*({\cal L})\,,\ \mathscr{M}:=p_1^*({\cal M})\,.$$
Let $Z  \subset \P({\cal E})\times X$ be the graph of $\pi$, and note that, via the obvious isomorphism $Z\simeq \P({\cal E})$ the restrictions $\resto{\mathscr{E}}{Z}$, $\resto{\mathscr{L}}{Z}$, $\resto{\mathscr{M}}{Z}$ are identified with $\pi^*({\cal E})$,  ${\cal L}$ and ${\cal M}$ respectively. Therefore one obtains a short exact sequence
$$0\to \resto{\mathscr{L}}{Z}\textmap{u} \resto{\mathscr{E}}{Z}\textmap{\mu} \resto{\mathscr{M}}{Z}\to 0\,.
$$
Let $\tilde\mu$ be the composition $\mathscr{E}\to \resto{\mathscr{E}}{Z}\textmap{\mu}\resto{\mathscr{M}}{Z}$, and put $\mathscr{F}:=\ker(\tilde\mu)$.   For a point $x\in X$ and a line $y\in \P({\cal E}(x))$ put  $q_y:={\cal E}(x)/y$, and let $\eta_y:{\cal E}(x)\to q_y$ be the canonical epimorphism. The sheaf ${\cal F}_y$ on $X$ given by the restriction $\resto{\mathscr{F}}{\{y\}\times X}$ fits into the exact sequence
\begin{equation}\label{SESFy}
0\to {\cal F}_y\hookrightarrow {\cal E}\textmap{\tilde\eta_y}q_y\otimes{\cal O}_{\{x\}} \to 0\,,	
\end{equation}
where $\tilde\eta_y$ is induced by $\eta_y$. As in the proof of Proposition \ref{embedding} we obtain for any $y\in \P({\cal E})$  a canonical exact sequence
\begin{equation}\label{HomFy} 0\to {\cal H}om({\cal F}_y ,{\cal F}_y )\to {\cal H}om( {\cal E}, {\cal E})\textmap{\psi_y} ( y^\smvee \otimes  q_y)\otimes {\cal O}_{\{x\}} \to 0\;.
\end{equation}
\begin{lm} For any $x\in X$ and $y\in \P({\cal E}(x))$ one has
\begin{enumerate}
\item ${\cal E}xt^k({\cal F}_y, {\cal F}_y)=0\ \forall k\geq 2$\,.
\item  A canonical isomorphism $ {\cal E}xt^1({\cal F}_y,{\cal F}_y)  \textmap{\delta_y}   {\cal E}xt^2(q_y\otimes{\cal O}_{\{x\}},{\cal F}_y)$\,.

\item  Canonical isomorphisms 
$$H^2({\cal E}nd({\cal F}_y))\textmap{\simeq} H^2({\cal E}nd({\cal E}))\,,\ H^2({\cal E}nd_0({\cal F}_y))\textmap{\simeq} H^2({\cal E}nd_0({\cal E}))\,.$$
\item A canonical short exact sequence 
\begin{equation}\label{CanShExSeqTx}
0\to {\cal T}(x)\to H^0({\cal E}xt^1({\cal F}_y,{\cal F}_y))\textmap{h} \Omega^2(x)^\smvee\otimes q_y^\smvee \otimes y \to  0\,.	
\end{equation}
\end{enumerate}

\begin{proof}
The first claim follows as in the proof of 	Proposition \ref{embedding};  (2)  follows directly from (\ref{SESFy}) taking into account that ${\cal E}$ is locally free, and  (3) follows from (\ref{HomFy}) taking into account that $H^k({\cal O}_{\{x\}})=0$ for $k>0$. For (4)  use (2),  the  exact sequence 
$$0\to  {\cal E}xt^1(q_y\otimes{\cal O}_{\{x\}}, q_y\otimes{\cal O}_{\{x\}})   \to     {\cal E}xt^2(q_y\otimes{\cal O}_{\{x\}},{\cal F}_y) \to  {\cal E}xt^2(q_y\otimes{\cal O}_{\{x\}},{\cal E})  \textmap{a} $$
$$\to   {\cal E}xt^2(q_y\otimes{\cal O}_{\{x\}},q_y\otimes{\cal O}_{\{x\}})\to 0\,,
$$
and the canonical isomorphisms (see Section \ref{FamIdealSh}):
$${\cal E}xt^1(q_y\otimes{\cal O}_{\{x\}}, q_y\otimes{\cal O}_{\{x\}})={\cal E}xt^1( {\cal O}_{\{x\}},  {\cal O}_{\{x\}})={\cal H}om({\cal I}_x, {\cal O}_{\{x\}})={\cal T}(x)\,,
$$
$${\cal E}xt^2(q_y\otimes{\cal O}_{\{x\}},{\cal E}))={\cal E}xt^1({\cal I}_{x},q_y^\vee\otimes{\cal E})=\Omega^2(x)^\smvee\otimes q_y^\vee\otimes  {\cal E}(x)\,,
$$
$$ {\cal E}xt^2(q_y\otimes{\cal O}_{\{x\}},q_y\otimes{\cal O}_{\{x\}})={\cal E}xt^1( {\cal I}_{x},  {\cal O}_{\{x\}})=\Omega^2(x)^\smvee.$$
The morphism $\Omega^2(x)^\smvee\otimes q_y^\vee\otimes {\cal E}\to \Omega^2(x)^\smvee$  which corresponds to $a$ via the latter identifications is $\id_{\Omega^2(x)^\smvee}\otimes \id_{q_y^\vee}\otimes \eta_y$, so $\ker(a)=\Omega^2(x)^\smvee\otimes q_y^\vee\otimes y$.

\end{proof}

\end{lm}

\begin{lm}\label{Esimple}
If ${\cal E}$ is simple, then for any $y\in\P({\cal E})$ the sheaf ${\cal F}_y$ is simple, and one has a canonical short exact sequence
$$0\to y^\smvee \otimes  q_y\to  H^1({\cal H}om_0({\cal F}_y ,{\cal F}_y ))\to H^1({\cal H}om_0({\cal E} ,{\cal E} ))\to 0\,.
$$

\end{lm}

If $H^2({\cal E}nd_0({\cal E}))=0$, then $H^2({\cal E}nd_0({\cal F}))=0$, and the local-global spectral sequence yields    a canonical short exact sequence
\begin{equation}\label{NewLG}
0\to H^1({\cal H}om_0({\cal F}_y,{\cal F}_y))\to \Ext^1_0({\cal F}_y,{\cal F}_y)\textmap{m} H^0({\cal E}xt^1({\cal F}_y,{\cal F}_y))\to 0\,.	
\end{equation}

Suppose now that ${\cal E}$ is simple and 	$H^i({\cal E}nd_0({\cal E}))=0$ for $i\in\{1,2\}$. Combining (\ref{NewLG}), (\ref{CanShExSeqTx}) with Lemma \ref{Esimple}, and putting $\Theta_y:=m^{-1}({\cal T}(x))$, we obtain two short exact sequences
\begin{equation}
\begin{gathered}
0\to \Theta_y\hookrightarrow	 \Ext^1_0({\cal F}_y,{\cal F}_y)\textmap{h\circ m} \Omega^2(x)^\smvee\otimes q_y^\smvee \otimes y \to 0\,,\\
0\to  y^\smvee \otimes q_y \to  \Theta_y \textmap{\resto{m}{\Theta_y}} {\cal T}(x)\to 0\,.
\end{gathered}
\end{equation}

  The sheaf $\mathscr{F}$ is flat over $\P({\cal E})$, so it defines a holomorphic map $\Phi^{\mathscr{F}}:\P({\cal E})\to {\cal M}^\si$. Denote by  ${\cal T}_{\P({\cal E})}^V$ the vertical tangent subbundle of the tangent bundle ${\cal T}_{\P({\cal E})}$. It is easy to see that
\begin{lm} Suppose  that ${\cal E}$ is simple and 	$H^i({\cal E}nd_0({\cal E}))=0$ for $i\in\{1,2\}$. Then $\Phi^{\mathscr{F}}$ takes values in ${\cal M}^\si_\reg$, and  for any $y\in \P({\cal E})$ the differential  $\Phi^{\mathscr{F}}_{*y}:{\cal T}_{\P({\cal E})}(y)\to \Ext^1_0({\cal F}_y,{\cal F}_y)$ at $y$   has the following properties:
\begin{enumerate}
\item It maps isomorphically ${\cal T}_{\P({\cal E})}(y)$ on $\Theta_y$.
\item It maps  $y^\smvee \otimes q_y$ isomorphically onto the vertical tangent line 	${\cal T}_{\P({\cal E})}^V(y)$, and induces the canonical isomorphism $y^\smvee \otimes q_y\textmap{\simeq} T_y(\P({\cal E}(x))$.
\end{enumerate}	
\end{lm}
This shows that
\begin{pr} \label{NormalComp} Suppose  that ${\cal E}$ is simple and 	$H^i({\cal E}nd_0({\cal E}))=0$ for $i\in\{1,2\}$. The map $\Phi^{\mathscr{F}}:\P({\cal E})\to {\cal M}^\si$ is a codimension one holomorphic embedding whose normal line bundle  ${\cal N}$  is isomorphic to $\pi^*({\cal K}_X^\smvee)\otimes{\cal M}^\smvee\otimes {\cal L} =\pi^*({\cal K}_X^\smvee\otimes\det({\cal E})^\smvee)\otimes{\cal O}_{{\cal E}}(-2)$.
	
\end{pr}
We can now prove Theorem \ref{Th2} stated in Section \ref{FundamQ}:
\begin{proof} (of Theorem \ref{Th2}) Let ${\cal E}$ be the holomorphic  $\SL(2,\C)$-bundle associated with the irreducible, flat $\SU(2)$-connection $A$. Identify $\P({\cal E})$ with its image in ${\cal M}^\si$ via $\Phi^{\mathscr{F}}$.  
Using the same methods as in the proof of Proposition \ref{Open2} and Corollary \ref{MainCo}, we obtain an open, Hausdorff  neighbourhood $\Ug$ of $\P({\cal E})$ in ${\cal M}^\si_\reg$ such that $\Ug\setminus \P({\cal E})\subset ({\cal M}^\st_{\rm lf})_\reg$.  The continuity theorem \cite{BTT} for flat families gives a continuous map $\tau:\Ug\to \overline{\cal M}^\ASD(E)$ which agrees with the Kobayashi-Hitchin correspondence  on $\Ug\setminus\P({\cal E})$ and with $\pi$ on $\P({\cal E})$. Choose ${\cal U}\Subset \Ug$ such that $\overline {\cal U}$ is a compact tubular neighbourhood of $\P({\cal E})$ in $\Ug$ with smooth boundary ${\cal B}$.  By Proposition \ref{ModelIrred} we may identify an open neighbourhood of   $\{[A]\}\times X$ with the cone bundle $C_X^\varepsilon$. Using the continuity of $\tau$ we may suppose (taking  $\overline {\cal U}$ sufficiently small) that $\tau ( \overline {\cal U})\subset  C_X^\varepsilon$. By Lemma \ref{finalLm} proved below (which is a special case of \cite[Lemma 4.4]{To}, the image   ${\cal W}:=\tau({\cal U})$ is   open in $C_X^\varepsilon$.  Since $\tau(\overline {\cal U})$ is compact, it is closed, so the  obvious inclusion $\tau(\overline {\cal U})\subset \overline{\tau(  {\cal U})}=\overline {\cal W}$ is an equality.

By Proposition \ref{NormalComp} the pair $(\Ug,\pi:\P({\cal E})\to X)$ satisfies the hypothesis of \cite[Theorem $2'$]{Fuj}, so there  exists a normal complex space $\Vg$ with an embedding $X\hookrightarrow \Vg$, and a modification $c:\Ug\to \Vg$, such that  the pair $(\Vg,c)$ is the blowing down of $\Ug$ along $\pi$. Moreover, by \cite[Theorem 3.7]{Tom} it follows that $c$ coincides with the monoidal transformation of $\Vg$ with centre   $c(\P({\cal E}))=X$.

Put  ${\cal V}:=c({\cal U})$. Since $\tau$ is constant on the fibres of $c$, it induces a continuous map $\overline{\cal V}=c(\overline{\cal U})\to \overline {\cal W}$, which is obviously bijective. But $\overline{\cal V}$ is compact and $\overline {\cal W}$ is Hausdorff, so this map is a homeomorphism. Endowing ${\cal W}$, which is an open neighbourhood of $\{[A]\}\times X$ in $\overline{\cal M}^\ASD(E)$, with the complex space structure induced from ${\cal V}$ via  the homeomorphism ${\cal V}\to   {\cal W}$, we obtain the desired normal complex  space structure around $\{[A]\}\times X$, and this   structure obviously extends the natural complex space structure of  ${\cal M}^\ASD(E)^*$.  

The identification between the normal cone of $X$ in ${\cal V}$ and the cone of degenerate elements in ${\cal K}_X^\smvee\otimes S^2({\cal E})$ follows from the isomorphism ${\cal N}\simeq \pi^*({\cal K}_X^\smvee)\otimes {\cal O}_{\cal E}(-2)$ given by Proposition \ref{NormalComp} using \cite[Section B.6]{Ful}.
\end{proof}
The following result used in the proof  is a special case of \cite[Lemma 4.4]{To}. We give a short proof for completeness. 
\begin{lm}\label{finalLm} With the notation introduced in the proof of Theorem \ref{Th2}, the set  ${\cal W}:=\tau({\cal U})$ is   open in $C_X^\varepsilon$.	
\end{lm}
\begin{proof}

The set ${\cal W}\setminus X=\tau(\overline{\cal U}\setminus(\P({\cal E})\cup {\cal B}))$  is obviously contained in  $C^\varepsilon_X\setminus(X\cup\tau({\cal B}))$. Since the restriction of $\tau$ to $\Ug\setminus \P({\cal E})$ is an open embedding,  ${\cal W}\setminus X$ is open in  the complement $C^\varepsilon_X\setminus(X\cup\tau({\cal B}))$; it also closed in this complement, because it can be written as $\tau(\overline{\cal U})\cap \big(C^\varepsilon_X\setminus(X\cup\tau({\cal B}))\big)$, and $\tau(\overline{\cal U})$ is compact. Since ${\cal W}\setminus X$ is connected, it follows that it coincides with a connected component  of  $C^\varepsilon_X\setminus(X\cup\tau({\cal B}))$.

Let  $\eta\in(0,\varepsilon)$ be such that $C^\eta_X\cap \tau({\cal B})=\emptyset$. The set $C^\eta_X\setminus X$ is contained in $C^\varepsilon_X\setminus(X\cup\tau({\cal B}))$, is connected, and intersects ${\cal W}\setminus X$, so it is contained in this connected component of $C^\varepsilon_X\setminus(X\cup\tau({\cal B}))$. Therefore $C^\eta_X\subset {\cal W}$,  so ${\cal W}$ contains a neighbourhood of $X$ in $C^\varepsilon_X$. Recalling that $\tau$ is an open embedding on $\Ug\setminus \P({\cal E})$, it follows that ${\cal W}$ is open.

\end{proof}

\section{Appendix}
\subsection{Elementary transformations of sheaves and sheaf deformations}

\subsubsection{Elementary transformations of coherent sheaves}

Let $Y$ be a  complex manifold, and $\xi:X\hookrightarrow  Y$ be the embedding map of an effective divisor $X$ of $Y$. For a coherent sheaf ${\cal S}$ on $Y$ we denote by $\resto{{\cal S}}{X}$ the restriction $\xi^*({\cal S})$, which is a coherent sheaf on $X$, and by ${\cal S}_X$ the sheaf $\xi_*\xi^*({\cal S})={\cal S}/{\cal I}_X{\cal S}$, which is a torsion coherent sheaf on $Y$.

 Let $p:{\cal F}\to {\cal H}$ be   an epimorphism of coherent  sheaves on $Y$, where  ${\cal H}$ has the property 
$${\cal I}_X{\cal H}=0\,,$$
i.e. ${\cal H}$ is isomorphic to the direct image of a coherent sheaf on $X$. Let  $p_X:{\cal F}_X\to {\cal H}$ be the morphism induced  by $p$, and put
$${\cal F}':=\ker(p)\ ,\ {\cal H}':=\ker(p_X)\,.$$
The exact sequence
$$0\to{\cal F}'\hookrightarrow {\cal F}\textmap{p}{\cal H}\to 0
$$
gives an exact sequence
 \begin{equation}\label{FirstES}
 {\cal F}'_X\to {\cal F}_X\textmap{p_X} {\cal H}\to 0\,,	
 \end{equation}
which proves that the image of ${\cal F}'_X$ in ${\cal F}_X$ is ${\cal H}'$. Thus ${\cal F}'$ comes with an epimorphism $p':{\cal F}'\to {\cal H}'$, where ${\cal H}'$ again has the property ${\cal I}_X {\cal H}'=0$\,.
\begin{dt} Let $p:{\cal F}\to {\cal H}$ be   an epimorphism of coherent  sheaves on $Y$ where ${\cal I}_X {\cal H}=0$.
The elementary transformation of the pair $({\cal F},p)$ is the pair $({\cal F}',p')$, where ${\cal F}':=\ker(p)$, ${\cal H}'=\ker(p_X)$ and $p'$ is induced by the composition ${\cal F}'\to {\cal F}'_X\to {\cal F}_X$ (whose image is ${\cal H}'$)\,. 
\end{dt}
Note that
\begin{equation}\label{ker(p'X)}
{\cal H}'':=\ker(p'_X:{\cal F}'_X\to {\cal H}'_X)= \ker({\cal F}'_X\to {\cal F}_X)=\qmod{{\cal I}_X{\cal F}}{{\cal I}_X{\cal F}'}\,.	
\end{equation}
Denoting by $({\cal F}'',p'')$ the elementary transformation of $({\cal F}',p')$ one has
$${\cal F}''=\ker(p')=\ker({\cal F}'\to {\cal F}'_X\to {\cal F}_X)=\ker({\cal F}'\hookrightarrow {\cal F}\to {\cal F}_X)={\cal F}'\cap({\cal I}_X {\cal F})={\cal I}_X{\cal F}\,,
$$
where for the last equality we have used the obvious inclusion ${\cal I}_X{\cal F}\subset {\cal F}'$. The epimorphism  $p'':{\cal F}''\to {\cal H}''$ is just the canonical epimorphism ${\cal I}_X{\cal F}\to  {{\cal I}_X{\cal F}}/{{\cal I}_X{\cal F}'}$.
Using our assumption that $X$ is a divisor in $Y$ we obtain:
\begin{re}\label{Tor} Regarding   $0\to {\cal O}(-X)\to {\cal O}\to {\cal O}_X\to 0$
as a resolution of ${\cal O}_X$ by locally free ${\cal O}_Y$-modules, we get for any coherent sheaf ${\cal A}$ on $Y$
$${\cal T}or_1({\cal A},{\cal O}_X)=\ker ({\cal A}(-X)\to {\cal A})\ ,\ {\cal T}or_k({\cal A},{\cal O}_X)=0\hbox{ for } k\geq 2\,.
$$	
Therefore the kernel of the canonical epimorphism ${\cal A}(-X)\to {\cal I}_X{\cal A}$ is ${\cal T}or_1({\cal A},{\cal O}_X)$\,.
\end{re}
The  exact sequence (\ref{FirstES}) can be continued to the left as follows
$$0\to {\cal T}or_1({\cal F}',{\cal O}_X)\to {\cal T}or_1({\cal F},{\cal O}_X)\to {\cal H}(-X)\to  {\cal F}'_X\to {\cal F}_X\textmap{p_X} {\cal H}\to 0\,.
$$
We thus obtain 
\begin{pr}\label{PropElTr}
With the notations and the assumptions above, suppose that  ${\cal T}or_1({\cal F},{\cal O}_X)=0$, and let $({\cal F}',p')$ be the elementary transformation of $({\cal F},p)$\,.  Then 
\begin{enumerate}
\item ${\cal T}or_1({\cal F}',{\cal O}_X)=0$\,,
\item 	One has a short exact sequence
\begin{equation}\label{ElTrShExSeq}
0\to {\cal H}(-X)\textmap{\rho} {\cal F}'_X\textmap{p'_X} {\cal H}'\to 0\,,	
\end{equation}
where $\rho$ identifies ${\cal H}(-X)$  with  $\ker(p'_X)={\cal I}_X{\cal F}/{\cal I}_X{\cal F}'$ via the identifications ${\cal I}_X{\cal F}={\cal F}(-X)$,  ${\cal I}_X{\cal F}'={\cal F}'(-X)$, ${\cal F}/{\cal F}'={\cal H}$.
\item The second elementary transformation $({\cal F}'',p'':{\cal F}''\to {\cal H}'')$ can be identified with  the pair 
 $({\cal F}(-X),p\otimes\id: {\cal F}(-X)\to {\cal H}(-X))$\,.
\end{enumerate}

\end{pr}
Supposing again  ${\cal T}or_1({\cal F},{\cal O}_X)=0$, the short exact sequence of ${\cal O}_{2X}$-modules
\begin{equation}\label{2X}
0\to {\cal O}_X(-X)\textmap{\iota_X} {\cal O}_{2X}\textmap{\pi_X} {\cal O}_X\to 0	
\end{equation}
associated with the decomposition $2X=X+X$ gives a short exact sequence %
\begin{equation}\label{2XF}
0\to {\cal F}_X(-X)\textmap{\iota_X^{\cal F}} {\cal F}_{2X}\textmap{\pi_X^{\cal F}} {\cal F}_X\to 0	
\end{equation}
of ${\cal O}_{2X}$-modules. Let $\varepsilon_X({\cal F})\in \Ext^1_{{\cal O}_{2X}}({\cal F}_X,{\cal F}_X(-X))$ be the  extension class of (\ref{2XF}), and  let $\varepsilon_X({\cal F},p)\in \Ext^1_{{\cal O}_X}({\cal H}',{\cal H}(-X))$ be the extension class of (\ref{ElTrShExSeq}).  The two extensions can be compared using the morphisms
$$j:{\cal H}'\hookrightarrow {\cal F}_X\,,\  p_X\otimes\id: {\cal F}_X(-X)\to {\cal H}(-X)\,.
$$
Recall that the functors $\Ext$ are contravariant with respect to the first, and covariant with respect to the second argument. For an element $u\in \Ext_A(M,N)$ and morphisms $\mu:M'\to M$, $\nu:N\to N'$ we denote by $\nu u\mu$ the corresponding element of $\Ext_A(M',N')$\,.  
 
\begin{thry}\label{MainFTh} Let $p:{\cal F}\to {\cal H}$ be   an epimorphism of coherent  sheaves on $Y$ where ${\cal I}_X {\cal H}=0$, and ${\cal T}or_1({\cal F},{\cal O}_X)=0$. With the notations   above we have
\begin{equation}\label{MainF}
A_{\pi_X}(\varepsilon_X({\cal F},p))=(p_X\otimes\id)(\varepsilon_X({\cal F}))j\,,	
\end{equation}
where 
$$
\begin{diagram}[h=8mm] 
{\cal C}oh(X)\times{\cal C}oh(X) &= &	{\cal C}oh(X)\times{\cal C}oh(X)\\
\dTo^{\Ext^1_{{\cal O}_X}}_{\phantom{H}}  & \rTo^{A_{\pi_X}} & \dTo_{\Ext^1_{{\cal O}_{2X}}}^{\phantom{H}} \\
{\cal V}ect_\C & & {\cal V}ect_\C 
\end{diagram}
$$
is the  natural transformation  induced by the epimorphism $\pi_X:{\cal O}_{2X}\to {\cal O}_X$\,.
\end{thry}
\begin{proof}
Put $r:=\iota_X^{\cal F}$, $q:=\pi_X^{\cal F}$ to simplify the notation.	
Taking into account that $j$ is a  monomorphism, and $p_X\otimes\id$ is an epimorphism, it follows that  $(p_X\otimes\id)(\varepsilon({\cal F}))j$ is the isomorphism class of the extension corresponding to the upper line in the diagram
$$
\begin{diagram}[h=7mm]
 &&0&&&&0 \\
&&\uTo &&&&\dTo &&  \\
0&\rTo &{\cal H}(-X) &\rTo^{\bar r} &\qmod{q^{-1}({\cal H}')}{r(\ker(p_X\otimes\id))}&\rTo^{\bar q} &{\cal H}'&\rTo 0\\
  & & \uOnto^{p_X\otimes\id} &&&& \dInto_j \\
0&\rTo & {\cal F}_X(-X) &\rTo^{r} &{\cal F}_{2X} &\rTo^{q} & {\cal F}_X &\rTo 0\\
&&\uTo^{i_X\otimes\id} &&&&\dTo_{p_X} && \\
&&{\cal F}'_X(-X) &&&& {\cal H} &&\\
&&  &&&&\dTo &&  \\
 && &&&&0
\end{diagram}\,,   
$$
where the morphisms $\bar r$, $\bar q$ are induced by $r$ and $q$ respectively.  The morphism   $i_X:{\cal F}'_X\to {\cal F}_X$ is induced by the inclusion  $i:{\cal F}'\hookrightarrow {\cal F}$. 
The right-hand vertical exact sequence in the diagram and the exact sequence
$$    {\cal F}'_{2X}\to {\cal F}_{2X}\textmap{p_X\circ q} {\cal H}\to 0$$
show that 
\begin{equation}\label{q-1}
q^{-1}({\cal H}')=\ker (p_X\circ q)=\qmod{{\cal F}'}{{\cal I}_{2X}{\cal F}}\,.	
\end{equation}
On the other hand, the definition of $r$ and the left-hand vertical exact sequence give
\begin{equation}\label{r(kernel)}r({\cal F}_X(-X))=\qmod{{\cal I}_X{\cal F}}{{\cal I}_{2X}{\cal F}}\ ,\ r(\ker(p_X\otimes\id))=r\big(\im(i_X\otimes\id)\big)=\qmod{{\cal I}_X{\cal F}'}{{\cal I}_{2X}{\cal F}}\,.
\end{equation}
Using (\ref{q-1}), (\ref{r(kernel)}) we get an obvious isomorphism of ${\cal O}_{2X}$-modules.
$$\qmod{q^{-1}({\cal H}')}{r(\ker(p_X\otimes\id))}\textmap{\simeq} {\cal F}'_X\,.
$$
It is easy to check that $\rho$, $p'_X$ correspond to $\bar r$, $\bar q$ via this isomorphism. 
\end{proof}
Denote by $\pmb{o}$ the double origin in $\C$, regarded as a non-reduced complex space.

\begin{re}
The extension (\ref{2X}) has a  right splitting which is multiplicative (makes ${\cal O}_{\pmb{X}}$ a sheaf of ${\cal O}_X$-algebras) if and only if the embedding $X\hookrightarrow \pmb{X}$ of $X$ in its   second order infinitesimal neighbourhood  $\pmb{X}$ has a left inverse. Equivalently, this means that $\pmb{X}$ has the structure of an $\pmb{o}$-fibre bundle over $X$ such that   $X\hookrightarrow \pmb{X}$ becomes a section in this bundle.
\end{re}

\begin{co}\label{WithSection}
Let $p:{\cal F}\to {\cal H}$ be   an epimorphism of coherent  sheaves on $Y$ where ${\cal I}_X {\cal H}=0$, and ${\cal T}or_1({\cal F},{\cal O}_X)=0$. Let $\sigma:{\cal O}_X\to {\cal O}_{2X}$ be a multiplicative right splitting of (\ref{2X}), and $\varepsilon_X^\sigma({\cal F})\in\Ext^1_{{\cal O}_X}({\cal F}_X,{\cal F}_X(-X))$ be the extension class of (\ref{2X}), regarded as an extension of ${\cal O}_X$-modules via $\sigma$. Then
\begin{equation}\label{CompExt} \varepsilon_X({\cal F},p)=(p_X\otimes\id)(\varepsilon_X^\sigma({\cal F}))j\,.
\end{equation}
\end{co}
\begin{proof}
We apply the natural transformation $A_\sigma$ in (\ref{MainF})	 taking into account that $A_\sigma\circ A_{\pi_X}=\id$, $\varepsilon_X^\sigma({\cal F})=A_\sigma(\varepsilon_X({\cal F}))$, and that $p_X\otimes\id$, $j$ are morphisms of ${\cal O}_X$-modules.
\end{proof}

\subsubsection{Elementary transformations of sheaf deformations}
\label{ElTrDef}

Suppose now that $Y=D\times X$, where $X$ is a compact complex manifold, and $D\subset \C$ is the standard disk. We will identify $X$ with the fibre $X_0=\{0\}\times X$ over the origin $0\in D$. Let ${\cal V}$ be a coherent sheaf on $X$, and ${\cal F}$ be a coherent sheaf on $Y$, flat over $D$, endowed with a fixed isomorphism $\resto{{\cal F}}{X}={\cal V}$. In other words ${\cal F}$ is a deformation of ${\cal V}$ parameterised by $D$. The corresponding infinitesimal deformation is an element $\epsilon_{\frac{\partial}{\partial z}}({\cal F})\in\Ext^1({\cal V},{\cal V})$\,.
On $X$ we fix a short exact sequence
$$0\to {\cal T}'\textmap{j_0}  {\cal V}\textmap{p_0} {\cal T}\to 0
$$
where ${\cal T}':=\ker(p_0)$ and $j_0:{\cal T}'\hookrightarrow {\cal V}$ is the inclusion morphism.  We  denote by ${\cal T}^Y$, ${\cal T}'^Y$ the direct images of ${\cal T}$, ${\cal T}$ to $Y$, and  by $p:{\cal F}\to {\cal T}^Y$ the epimorphism induced by $p_0$. Taking into account that in this special situation ${\cal O}_X(-X)$ is trivial on $X$, Proposition \ref{PropElTr} shows that the elementary transformation   of the pair $({\cal F},p)$ gives a subsheaf ${\cal F}'\subset {\cal F}$, which comes with a short exact sequence
$$0\to {\cal T}^Y\to {\cal F}'_X\to {\cal T}'^Y\to 0\,.
$$
 Restricting to $X$ we obtain a short exact sequence
 $$0\to {\cal T}\textmap{j_0'} {\cal V}'\textmap{p_0'} {\cal T}'\to 0\,,
 $$
 whose extension class is an element $\varepsilon({\cal F},p_0)\in \Ext^1({\cal T}',{\cal T})$. The following result shows that the extension class $\varepsilon({\cal F},p_0)$ can be computed explicitly in terms of the infinitesimal deformation $\epsilon_{\frac{\partial}{\partial z}}({\cal F})$\,.  Put $D^*:=D\setminus\{0\}$.

 \begin{co}\label{MainComparisonCo} With the notations and under the assumptions above one has
 \begin{enumerate}
 \item The elementary transformation ${\cal F}':=\ker(p)$ is flat over $D$, hence it is a deformation of ${\cal V}'$ parameterised by $D$, which coincides with ${\cal F}$ on $D^*\times X$.
 \item \begin{equation}\label{MainComparisonForm}
 \varepsilon({\cal F},p_0)=p_0(\epsilon_{\frac{\partial}{\partial z}}({\cal F}))j_0\,.
 \end{equation}	
 \end{enumerate}

  \end{co}
 \begin{proof} For (1) note that the stalks ${\cal O}_{D,z}$ are principal ideal domains, hence flatness over $D$ is equivalent to torsion-freeness as sheaf of ${\cal O}_D$-modules. It suffices to note that ${\cal F}'$ is a subsheaf of ${\cal F}$.\\
  
 (2)  follows directly from Corollary 	\ref{WithSection}, taking into account that, by definition, $\epsilon_{\frac{\partial}{\partial z}}({\cal F})$ is just the isomorphism class   of the canonical extension
 $$0\to {\cal F}_{X}(-X)={\cal F}_{X}\to {\cal F}_{2X}\to {\cal F}_X\to 0\,,
 $$
 where ${\cal F}_{2X}$ is regarded as a sheaf of ${\cal O}_X$-modules via the obvious multiplicative splitting $\sigma:{\cal O}_X\to {\cal O}_{2X}$ induced by the composition $2X\subset Y=D\times X\to X$\,.
 \end{proof}
 
 \subsection{The push-forward of a family under a branched double cover}
 \label{BranchedCoverSect}
 
 Let $\pi: B\to \Bg$ be  a branched  double covering of Riemann surfaces, $b\in B$ be a ramification point,  $\bg=\pi(b)$  be the corresponding branch point. Let $z$, $\zg$ be local coordinates of $B$, $\Bg$ around $b$, $\bg$ such that $z(b)=\zg(\bg)=0$,  $\zg\circ \pi=z^2$ and let $v\in T_b(B)$, $\vg\in T_\bg(\Bg)$ be tangent vectors such that $dz(v)= d\zg(\vg)=1$. 
 
 \begin{pr}\label{coverings}
Let $X$ be a compact complex manifold,  $\mathscr{V}$ be a sheaf on $B\times X$, flat over $B$, and put $\mathscr{E}:=(\pi\times\id_X)_*(\mathscr{V})$. Let $\epsilon_v(\mathscr{V})\in \mathrm{Ext}^1({\cal V}_\bg,{\cal V}_\bg)$, $\epsilon_\vg(\mathscr{E})\in \mathrm{Ext}^1({\cal E}_\bg,{\cal E}_\bg)$ be the infinitesimal deformations of ${\cal V}_b:=  {\mathscr{V}}_{\{b\}\times X}$, ${\cal E}_\bg:=  \mathscr{E}_{\{\bg\}\times X}$ (regarded as sheaves on $X$) corresponding to $(v,\mathscr{V})$, $(\vg,\mathscr{E})$ respectively.  Then
\begin{enumerate}
\item ${\cal E}_\bg$ fits in an exact sequence
$$0\to {\cal V}_b\textmap{J_z} {\cal E}_\bg\textmap{R} {\cal V}_b\to 0
$$
whose extension class is $\epsilon_v(\mathscr{V})$.
\item \label{DefoCOmp} One has  	$\epsilon_v(\mathscr{V})=R\big(\epsilon_\vg(\mathscr{E})\big)J_z$.
\end{enumerate}
\end{pr}
\begin{proof}
(1) Denote by  $\pmb{b}\subset B$, $\pmb{X}\subset B\times X$ the non-reduced complex subspaces associated with the effective divisors $2\{b\}$, $2(\{b\}\times X)$ respectively. One has obviously
$$\pmb{X}=\pmb{b}\times X=(\pi\times\id_X)^{-1}(\{\bg\}\times X)\,.
$$
Identify $X$ with the divisor $\{b\}\times X$ of $B\times X$ to save on notations. Multiplication with $z$ defines a sheaf monomorphism $j_z:{\cal O}_{X}\to {\cal O}_{\pmb{X}}$, and a short exact sequence  
 $$0\to {\cal O}_{X}\textmap{j_z} {\cal O}_{\pmb{X}}\textmap{r} {\cal O}_{ X}\to 0
 $$
 on $B\times X$.  Taking tensor product with $\mathscr{V}$ and taking into account the flatness of  $\mathscr{V}$ over $B$, we get a short exact sequence
 \begin{equation}\label{extV}
 0\to \mathscr{V}_{X}\textmap{J_z} \mathscr{V}_{\pmb{X}}\textmap{R} \mathscr{V}_{X}\to 0\,.
 \end{equation}
By definition, $\epsilon_v(\mathscr{V})$ is the extension class of (\ref{extV}) when $\mathscr{V}_{\pmb{X}}$ is regarded as an ${\cal O}_X$-module via the projection $Q:\pmb{X}=\pmb{b}\times X\to X$.  

On the other hand, since $\pi\times\id_X$ is a finite map, $(\pi\times\id_X)_*(\mathscr{V})$ commutes with base change, hence  
 $${\cal E}_\bg=\big\{{(\pi\times\id_X)_*(\mathscr{V})}\big\}_{\{\bg\}\times X}=Q_*({{\cal V}}_{\pmb{X}})\,.
 $$
It suffices to note that the sheaf $Q_*({{\cal V}}_{\pmb{X}})$ coincides with  $ {\cal V}_{\pmb{X}}$ itself regarded as an ${\cal O}_X$-module, via the  ring sheaf morphism ${\cal O}_X\to {\cal O}_{\pmb{X}}$ induced by $Q$\,.
\\ \\
(2) Denote by ${\cal X}$ the complex subspace of $B\times X$ corresponding to the effective divisor $4X$. The infinitesimal deformation $\epsilon_\vg(\mathscr{E})$ is the extension class of the lower line in the diagram below, where $\mathscr{V}_{\pmb{X}}$,  $\mathscr{V}_{{\cal X}}$ are regarded as ${\cal O}_X$-modules via the obvious projections $\pmb{X}\to X$, ${\cal X}\to X$. As in the proof of Theorem \ref{MainFTh} we see that $R\big(\epsilon_\vg(\mathscr{E})\big)J_z\in\Ext^1({\cal V}_b,{\cal V}_b)$ is the extension class of the upper line in the diagram below, where $\bar J_{z^2}$, $\bar\Rg$ are  induced by $ J_{z^2}$, $\Rg$ respectively.  On the other hand one has
$$\qmod{\Rg^{-1}(J_z(\mathscr{V}_{{\cal X}}))}{J_{z^2}(\ker(R))}= \qmod{{\cal I}_X\mathscr{V}_{{\cal X}}}{{\cal I}_X^3 \mathscr{V}_{{\cal X}}}=\qmod{{\cal I}_X\mathscr{V}}{{\cal I}_X^3 \mathscr{V}}\,.
$$
Using again the flatness of $\mathscr{V}$ over $B$ it follows that the morphism 
$$F_z:\mathscr{V}_{\pmb{X}}=\qmod{\mathscr{V}}{{\cal I}_X^2 \mathscr{V}}\to \qmod{{\cal I}_X\mathscr{V}}{{\cal I}_X^3 \mathscr{V}}=\qmod{\Rg^{-1}(J_z(\mathscr{V}_{{\cal X}}))}{J_{z^2}(\ker(R))}$$
 given by multiplication with $z$ is an isomorphism. It suffices to note that the upper triangles are commutative.

$$
\begin{diagram}[h=7mm]
 &&0&&\mathscr{V}_{\pmb{X}}&&0 \\
&&\uTo &\ruTo^{J_z}&\dTo^{F_z}_\simeq&\rdTo^{R}&\dTo &&  \\
0&\rTo &\mathscr{V}_{X} &\rTo_{\bar J_{z^2}} &\qmod{\Rg^{-1}(J_z(\mathscr{V}_{{\cal X}}))}{J_{z^2}(\ker(R))}&\rTo_{\bar \Rg} &\mathscr{V}_{X}&\rTo 0\\
  & & \uTo^{R} &&&& \dTo_{J_z} \\
0&\rTo & \mathscr{V}_{\pmb{X}} &\rTo^{J_{z^2}} &\mathscr{V}_{{\cal X}} &\rTo^{\Rg} &  \mathscr{V}_{\pmb{X}} &\rTo 0\\
&&\uTo^{J_z} &&&&\dTo_{R} && \\
&&\mathscr{V}_{ X} &&&&\mathscr{V}_{ X} &&\\
&&\uTo  &&&&\dTo &&  \\
 && 0&&&&0
\end{diagram}\,.   
$$
\end{proof}

\subsection{Extensions and families of ideal sheaves}
  \label{FamIdealSh}

  Let  $A$ be a commutative ring, and $M$, $N$ be $A$-modules. A free resolution 
$$\dots\to F_2\textmap{\delta_2} F_1\textmap{\delta_1} F_0\textmap{q} N\to 0$$
of $N$ gives an isomorphism 
$$\Ext^1(N,M)\simeq \qmod{\ker(\Hom(F_1,M)\to \Hom(F_2,M))}{\im(\Hom(F_0,M)\to \Hom(F_1,M))}.
$$
Let
$$0\to M\textmap{a} E\textmap{b} N\to 0
$$
be an extension of $N$ by $M$. The element 
$$\varepsilon(a,b)\in  \qmod{\ker(\Hom(F_1,M)\to \Hom(F_2,M))}{\im(\Hom(F_0,M)\to \Hom(F_1,M)) }$$
corresponding to the extension $(a,b)$ is given by the formula
\begin{equation}\label{e-formula}
\varepsilon(a,b)=[a^{-1}\circ p_1\circ s\circ \delta_1],
\end{equation}
where $s:F_0\to E\times_N F_0$ is a right splitting of the pull-back extension
$$0\to M \textmap{(a,0)}  E\times_N F_0\textmap{p_2} F_0\to 0\,.
$$

  Let $X$ be a complex manifold, $S\subset X$ be a codimension 2 locally complete intersection, and ${\cal I}_S$ be the ideal sheaf of $S$. Fix $x\in S$ and a system $(\xi_1,\xi_2)\in {\cal O}_x^{\oplus 2}$ of local equations of $S$. The Koszul resolution of ${\cal I}_{S,x}$ associated with $(\xi_0,\xi_1)$ reads
$$
0\to   {\cal O}_x\textmap{\delta_1}   {\cal O}_x^{\oplus 2}\textmap{q} {\cal I}_{S,x}\to 0\,, \eqno{(K_{\xi_1,\xi_2})}	
$$
where the morphisms $\delta_1$, $q$ are defined by
 $$\delta_1(\varphi)=(-\xi_2\varphi,\xi_1\varphi)\,,\ q(\varphi_1,\varphi_2)=\xi_1\varphi_1+\xi_2\varphi_2\,.
 $$
Using this resolution  we get, for any coherent sheaf ${\cal M}$ on $X$, an isomorphism 
 $$a_{\xi_1,\xi_2}^{\cal M}:{\cal E}xt^1_{{\cal O}_x}({\cal I}_{S,x},{\cal M}_x)\to {\cal M}_x/{\cal I}_S{\cal M}_x={\cal M}_{S,x}\,.$$
 The   isomorphisms $a_{\xi_1,\xi_2}^{\cal M}$ associated with local equation systems $(\xi_1,\xi_2)$ give a {\it canonical, global} identification \cite[Proposition 7.2, p. 179]{Ha}
\begin{equation}\label{can}
\begin{gathered} 
a^{\cal M}_S:{\cal E}xt^1_{{\cal O}_X}
({\cal I}_S,{\cal M})\textmap{\simeq} \omega_{S/X}\otimes_{{\cal O}_X}{\cal M}={\cal H}om_{{\cal O}_S}(\extp^2 ({\cal I}_S/{\cal I}_S^2),{\cal M}_S) \\
\langle a^{\cal M}_S(\varepsilon),d\xi_1\wedge d\xi_2\rangle =a_{\xi_1,\xi_2}^{\cal M}\,,
\end{gathered} 
 \end{equation}
where $\omega_{S/X}:={\cal H}om_{{\cal O}_S}(\extp^2 ({\cal I}_S/{\cal I}_S^2),{\cal O}_S)$ \cite[p. 141]{Ha}. 

Let now
 $$0\to {\cal M} \textmap{a} {\cal E}\textmap{b} {\cal I}_S \to 0\
 $$
 be an extension of ${\cal M}$ by ${\cal I}_S$, $\varepsilon(a,b)\in\Ext^1({\cal I}_S,{\cal M})$ be the corresponding extension class, and $\tilde \varepsilon(a,b)\in H^0({\cal E}xt^1({\cal I}_S,{\cal M}))=H^0(\omega_{S/X}\otimes_{{\cal O}_X} {\cal M})$ be its image  via the canonical morphism. Formula (\ref{e-formula}) gives for a local equation system   $(\xi_1,\xi_2)$ of $S$ at $x\in X$:
\begin{equation}\label{ext-formula}
\langle\tilde \varepsilon(a,b),d\xi_1\wedge d\xi_2\rangle=a_x^{-1}(-\xi_2\eta_1+\xi_1\eta_2)+{\cal I}_{S,x}{\cal M}_x\,,	
\end{equation}
 where $\eta_1$, $\eta_2\in {\cal E}_x$ are lifts of $\xi_1$, $\xi_2$ via $b$. 
 \begin{pr}\label{IdealFam} Let $X$ be a complex surface, ${\cal T}$ be its tangent sheaf, and  $x_0\in X$. Then
\begin{enumerate}
\item One has natural  isomorphisms
\begin{equation} 
{\cal E}xt^1_{{\cal O}_X}({\cal I}_{x_0},{\cal I}_{x_0})\textmap{\simeq}  \qmod{{\cal T}}{{\cal I}_{x_0}{\cal T}}\ ,	\ H^0\big({\cal E}xt^1_{{\cal O}_X}({\cal I}_{x_0},{\cal I}_{x_0})\big)={\cal T}(x_0)\,.
\end{equation}
\item Let $D\subset\C$ be the standard disk, $\varphi:D\to X$ be a holomorphic map with $\varphi(0)=x_0$, and $\Phi\subset D\times X$ be its graph. Then 
\begin{enumerate}
\item ${\cal I}_\Phi$ is flat over $D$,
\item The image $\tilde\epsilon_{\frac{\partial}{\partial z}(0)} ({\cal I}_\Phi)$ of   $\epsilon_{\frac{\partial}{\partial z}(0)} ({\cal I}_\Phi)\in \Ext^1({\cal I}_{x_0},{\cal I}_{x_0})$ via the canonical map  $\Ext^1({\cal I}_{x_0},{\cal I}_{x_0})\to  H^0\big({\cal E}xt^1_{{\cal O}_X}({\cal I}_{x_0},{\cal I}_{x_0})\big)$
is given by
\begin{equation}\label{speed}\tilde\epsilon_{\frac{\partial}{\partial z}(0)} ({\cal I}_\Phi)=\varphi'(0)\,.
\end{equation}
\end{enumerate}
\end{enumerate}

\begin{proof}
(1)  Using (\ref{can}) we get a canonical isomorphism
$${\cal E}xt^1({\cal I}_{x_0},{\cal I}_{x_0})\to \Omega^2(x_0)^\smvee\otimes_{{\cal O}_X} {\cal I}_{x_0}=\Omega^2(x_0)^\smvee\otimes_{{\cal O}_X} ({\cal I}_{x_0}/{\cal I}_{x_0}^2) $$
$$=\Omega^2(x_0)^\smvee\otimes \Omega^1(x_0)\simeq {\cal T}(x_0)\,,
$$
where the isomorphism ${\cal T}(x_0)\to\Hom \big(\Omega^2(x_0),  \Omega^1(x_0)\big)$ is given by
$$\langle v,\alpha_1\wedge \alpha_2\rangle=\alpha_1(v)\alpha_2-\alpha_2(v)\alpha_1\,.$$
(2) Assertion (a) is clear. Identify $X$ with the divisor $\{0\}\times X$ of $D\times X$ and denote by $\pmb{X}$ the non-reduced complex subspace corresponding to the effective divisor $2X$. The infinitesimal deformation $\epsilon:=\epsilon_{\frac{\partial}{\partial z}(0)} ({\cal I}_\Phi)$ is the extension class of the extension 
$$0\to \{{\cal I}_{\Phi}\}_X\textmap{a} \{{\cal I}_{\Phi}\}_{\pmb{X}}\textmap{b} \{{\cal I}_{\Phi}\}_X\to 0\,,
$$
where $\{{\cal I}_{\Phi}\}_{\pmb{X}}$ is regarded as an ${\cal O}_X$-module via the obvious morphism ${\cal O}_X\to {\cal O}_{\pmb{X}}$. Let $(\zeta_1,\zeta_2)$ be a chart of $X$ around $x_0$ with $\zeta_i(x_0)=0$. Put $\varphi_i(z)=\zeta_i(\varphi(z))$. Putting $\eta_i(z,x)=\zeta_i(x)-\varphi_i(z)$, the pair $(\eta_1,\eta_2)$ is a system of local equations of $\Phi$ around $(0,x_0)$, and the corresponding elements $\bar \eta_i\in \{{\cal I}_{\Phi}\}_{\pmb{X},(0,x_0)}$ are lifts of $\zeta_i$ via $b$. Putting $\psi_i:=\frac{1}{z}\varphi_i$ and  using (\ref{ext-formula}) we get
$$\langle\tilde\epsilon_{\frac{\partial}{\partial z}(0)} ({\cal I}_\Phi), d\zeta_1\wedge d\zeta_2\rangle=a^{-1}\big(-\zeta_2 \bar\eta_1+ \zeta_1 \bar\eta_2\big)+{\cal I}_{x_0}^2=\frac{1}{z}\big(-\zeta_2  \eta_1+ \zeta_1  \eta_2\big)+{\cal I}_{x_0}^2$$
$$=\zeta_2\psi_1-\zeta_1\psi_2 +{\cal I}_{x_0}^2= d\zeta_1(\varphi'(0)) d\zeta_2 - d\zeta_2(\varphi'(0)) d\zeta_1=\langle \varphi'(0), d\zeta_1\wedge d\zeta_2\rangle \,.$$
\end{proof}	
 \end{pr}
\subsection{Compact subsets of the moduli space of simple sheaves}
\label{ComSub}

Let $f:{\cal X}\to S$ be a proper morphism of complex spaces, and ${\cal A}$, ${\cal B}$ be coherent sheaves on ${\cal X}$  such that ${\cal A}$ is flat over $S$. We denote by $\Ag\ng_S$ the category of complex spaces over $S$, and by $H:=\underline{\Hom}({\cal A},{\cal B})$ the functor $\Ag\ng_S\to \sg\eg\tg\sg$ given by 
$$H(T):=\Hom_{{\cal X}_T}({\cal A}_T,{\cal B}_T)\,,
$$
where ${\cal X}_T:=T\times_S {\cal X}$, and ${\cal A}_T$, ${\cal B}_T$ are the inverse image of ${\cal A}$, ${\cal B}$ via the projection ${\cal X}_T\to {\cal X}$. A fundamental result of Flenner \cite[Section 3.2]{Fl} states that $H$ is represented by a linear space over $S$. In other words, there exists a coherent sheaf ${\cal H}$ on $S$ and, for any complex space $T\to S$ over $S$, a functorial bijection between $\Hom_{{\cal X}_T}({\cal A}_T,{\cal B}_T)$ and the set of holomorphic maps $T\to \mathbb{V}({\cal H})$  over $S$, where $\mathbb{V}({\cal H})$ denotes the linear space associated with ${\cal H}$ \cite[Section 1.6]{Fi}.   Using this result and  the isomorphisms $\mathbb{V}({\cal H})_s\simeq{\cal H}(s)^\smvee$ \cite[Section 1.8]{Fi}, one obtains
\def\supp{\mathrm{supp}}
\begin{equation}\label{support}
\{s\in S|\ \Hom({\cal A}_s,{\cal B}_s)\ne 0\}=\supp(\mathbb{V}({\cal H}))=\supp({\cal H})\,,
\end{equation}
where ${\cal A}_s$, ${\cal B}_s$ denote the restrictions of ${\cal A}$, ${\cal B}$ to the fibre $X_s:=f^{-1}(s)$. This proves the following semi-continuity result: 
\begin{pr}\cite{Fl}\label{SemiCont} The set $\{s\in S|\ \Hom({\cal A}_s,{\cal B}_s)\ne 0\}$ is Zariski closed in $S$.	
\end{pr}

Note that this statement is not a consequence of Grauert's semi-continuity theorem. Let $X$ be a compact complex space, and ${\cal M}^\si$ be the moduli space of simple sheaves on $X$. Recall that ${\cal M}^\si$ is a (possibly non-Hausdorff) complex space. For the following corollary see also \cite{KO}:
\begin{co}  \label{criterion}
Let $([{\cal F}_1],[{\cal F}_2])\in {\cal M}^\si\times {\cal M}^\si$ be a non-separable pair. Then 
$$\Hom({\cal F}_1,{\cal F}_2)\ne 0,\ \Hom({\cal F}_2,{\cal F}_1)\ne 0\,.$$
 \end{co}
\begin{proof}
For $i\in\{1,2\}$ let $(U_i,p_i)$ be  pointed Hausdorff complex spaces and $\mathscr{F}_i$ be a sheaf on $U_i\times X$, flat over $U_i$, such that $\resto{\mathscr{F}_i}{\{p_i\}\times X}\simeq {\cal F}_i$ and the map $U_i\to {\cal M}^\si$ induced by $\mathscr{F}_i$ is an open embedding.  Taking  pull-backs we obtain sheaves ${\cal A}_i$ over $(U_1\times U_2)\times X$ which are flat over $U_1\times U_2$.	Since $([{\cal F}_1],[{\cal F}_2])$ is a non-separable pair, for any open neighbourhood $W$ of $(p_1,p_2)$ in $U_1\times U_2$, there exists  $(u_1,u_2)\in W$ such that $\resto{\mathscr{F}_1}{\{u_1\}\times X}\simeq \resto{\mathscr{F}_2}{\{u_2\}\times X}$  (regarded as sheaves on $X$). Since $\resto{{\cal A}_i}{\{(u_1,u_2)\}\times X}=\resto{\mathscr{F}_i}{\{u_i\}\times X}$, the claim follows now from Proposition \ref{SemiCont}. \end{proof}

\begin{pr}\label{HausdNeighb}
Let ${\cal N}\subset {\cal M}^\si$ be a compact, locally closed	 subset,  such that for any pair $([{\cal F}_1],[{\cal F}_2])\in {\cal N}\times {\cal N}$  with $[{\cal F}_1]\ne [{\cal F}_2]$ one has $\Hom({\cal F}_1,{\cal F}_2)= 0$, or $\Hom({\cal F}_2,{\cal F}_1)= 0$. Then ${\cal N}$ has an open Hausdorff  neighbourhood. 
\end{pr}

This follows from Corollary \ref{criterion}  and Proposition \ref{MateiProp}  explained below, which is  a general existence criterion for Hausdorff open neighbourhoods of compact subspaces in locally compact spaces.   
Let $X$ be a topological space.  A subset $A\subset X$ will be called {\it separated in $X$} if any two distinct points of $A$ can be separated by disjoint neighbourhoods in $X$. For a point $x\in X$ we will denote by ${\cal V}_x$ the set of {\it open} neighbourhoods of $x$, and we put 
$$M_x:=\{y\in X|\ \forall U\in {\cal V}_x\ \forall V\in {\cal V}_y,\ U\cap V\ne\emptyset\}\setminus\{x\}\,,$$
and, for a set $L\subset X$, we put $M_L:=\union_{x\in L} M_x$.
We  adopt Bourbaki's definition of compactness, so compactness requires separateness.  
\begin{lm}\label{ML}
Let $X$ be locally Hausdorff topological  space, and $L\subset X$ be a compact subspace, which is separated in $X$. Then 
\begin{enumerate}
\item 	$L\cap \bar M_L=\emptyset$\,.
\item $M_L$ is closed in $X$.
\end{enumerate}
	
\end{lm}
\begin{proof}
(1) Let $y\in L$ and $V$ be a separated open neighbourhood of $y$ in $X$. Since $L$ is separated in $X$, the point $y$ can be separated in $X$ from any point of $L\setminus V$. Using the compactness of $L\setminus V$, it follows that there exist  open neighbourhoods $W$ of $y$ and $U$ of $L\setminus V$ with $W\cap U=\emptyset$. We may assume $W\subset V$. If, by reductio ad absurdum, $y$ belonged to $\bar M_L$, there would exist $x\in L$ such that $W\cap M_x\ne\emptyset$. The point $x$ cannot belong to $V$, because $V$ is separated and $W\subset V$; it cannot belong to $L\setminus V$ either, because in this case, for any point $z\in W$ the pair $(z,x)$ will be separated by the pair of open sets $(W,U)$\,.
\vspace{2mm}\\
(2) Suppose, by reductio ad absurdum, that there exists $y\in \bar M_L\setminus M_L$. Using (1) we obtain $y\not \in L\cup M_L$, so for every $x\in L$ there exists $V_x\in {\cal V}_x$ and $V^x_y\in {\cal V}_y$ such that $V_x\cap V^x_y=\emptyset$. Let $\{x_1,\dots,x_k\}$ be a finite subset of $L$ such that $L\subset \union_{i=1}^k V_{x_i}$, and put $V_y:=\bigcap_{i=1}^k V_y^{x_i}$. Thus $V_{x_i}\cap V_y=\emptyset$ for $1\leq i\leq k$.   Since $y\in \bar M_L$, there exists $x\in L$ and $p\in V_y\cap M_x$. Choosing $i\in\{1,\dots,k\}$ such that $x\in V_{x_i}$, we see that  the pair $(x,p)$ is separated in $X$ by the pair of open sets $(V_{x_i},V_y)$, which contradicts $p\in M_x$\,. 
\end{proof}
Lemma	\ref{ML} shows that
\begin{co}\label{SL}
In the conditions of Lemma	\ref{ML}, the set $S_L:=X\setminus M_L$ is an open neighbourhood of $L$, and  any pair $(x,y)\in L\times S_L$ with $x\ne y$ can be separated in $X$ by open sets.
\end{co}
With this preparation we can prove
\begin{pr}\label{MateiProp} Let $X$ be a topological space with the property that any point has a fundamental system of compact neighbourhoods. Let $L\subset X$ be a compact subspace, which is separated in $X$. Then $L$ admits a Hausdorff open neighbourhood in $X$.	
\end{pr}
\begin{proof}
For each $x\in L$ let $C_x$ be a compact (hence Hausdorff) neighbourhood of $x$, and  $K_x$ be a compact neighbourhood of $x$ which is contained in $S_L\cap \cringle{C}_x$. Note that, for every $x\in L$, the union $L\cup K_x$ is compact and separated in $X$, so Corollary \ref{SL} applies to this set. Let $\{x_1,\dots,x_k\}$ be a finite subset of $L$ such that $L\subset \union_{i=1}^k \cringle{K}_{x_i}$. The set 
$$V:=\big(\union_{i=1}^k \cringle{K}_{x_i}\big)\bigcap\big(\bigcap_{i=1}^k S_{L\cup K_{x_i}}\big)  
$$
is an open neighbourhood of $L$. Moreover one has $V\times V\subset \union_{i=1}^k (K_{x_i}\times S_{L\cup K_{x_i}})$, so any pair $(u,v)\in V\times V$ can be separated in $X$ by open sets.
\end{proof}

\subsection{The gluing lemma}\label{GluingThSect}

In many interesting gauge theoretical problems   one obtains a moduli space which is a topological manifold, and contains a distinguished open set naturally endowed with a holomorphic structure. A natural question asks if this holomorphic structure extends to the whole moduli space. The following  result gives a useful  tool for  dealing with this question:
 
\begin{lm}\label{glue}
Let ${\cal X}$ be a topological $2n$-dimensional manifold, and ${\cal Y}\subset {\cal X}$ be an open subset endowed with a complex manifold structure. Let ${\cal U}$ be an $n$-dimensional complex manifold, and let $f:{\cal U}\to {\cal X}$ be a map with the properties:
\begin{enumerate}[1.]
\item $f$ is continuous and injective,
\item ${\cal X}\setminus{\cal Y}\subset \im(f)$\,,
\item The restriction	$\resto{f}{f^{-1}({\cal Y})}:f^{-1}({\cal Y})\to {\cal Y}$ is holomorphic.
\end{enumerate}
Then 
\begin{enumerate}
\item $\im(f)$ is an open neighbourhood of ${\cal X}\setminus {\cal Y}$ in ${\cal X}$,
and $f$ induces a homeomorphism ${\cal U}\to \im(f)$. 
\item $f$ induces a biholomorphism $f^{-1}({\cal Y})\to \im(f)\cap {\cal Y}$ 	 with respect to the holomorphic structures induced by the open embeddings  $f^{-1}({\cal Y})\subset{\cal U}$, $\im(f)\cap {\cal Y}\subset {\cal Y}$.
\item There exists a unique complex manifold structure on ${\cal X}$ which extends the fixed complex structure on ${\cal Y}$, and such that $f$ becomes biholomorphic on its image.

\end{enumerate}
\end{lm}

\begin{proof}
(1) follows from the invariance of domain theorem. For (2)  note that $\resto{f}{f^{-1}({\cal Y})}:f^{-1}({\cal Y})\to {\cal Y}$ is an injective holomorphic map between smooth complex manifolds of the same dimension, hence it induces a biholomorphism 
$$f^{-1}({\cal Y})\to f(f^{-1}({\cal Y}))\,.$$
But $f(f^{-1}({\cal Y}))=\im(f)\cap {\cal Y}$.  The existence statement in (3) is proved as follows. Let ${\cal A}$ be a holomorphic atlas of  ${\cal Y}$ and ${\cal B}$ a holomorphic atlas of ${\cal U}$.   Using (2) it is easy to see that the union
$$\tilde {\cal A}:={\cal A}\cup \{h\circ f^{-1}:f(U_h)\to V_h|\ h:U_h\to V_h\in {\cal B}\} $$
is a holomorphic atlas on $X$. The unicity in (3) follows noting that any chart $\chi\in \tilde {\cal A}$ is holomorphic with a holomorphic structure on ${\cal X}$ satisfying the two conditions in (3).  
 \end{proof}

  \end{document}